\documentclass[11pt]{amsart}
\usepackage[letterpaper,margin=1.1in]{geometry}
\usepackage{amsmath,amsthm,amssymb,amsfonts,graphicx,mathrsfs,bm,amscd,latexsym,caption}
\usepackage{color,enumitem,hyperref}
\usepackage{blindtext}
\usepackage{textcomp}
\usepackage[utf8]{inputenc}
\usepackage[all]{xy}
\newtheorem{theorem}{Theorem}[section]
\newtheorem{proposition}[theorem]{Proposition}
\newtheorem{lemma}[theorem]{Lemma}
\newtheorem{corollary}[theorem]{Corollary}

\theoremstyle{definition}
\newtheorem{rem}[theorem]{Remark}
\newtheorem{definition}[theorem]{Definition}

\newtheorem{question}[theorem]{Question}

\def\R{\mathbb{R}}
\def\Z{\mathbb{Z}}
\def\N{\mathbb{N}}

\def\e{\epsilon}
\def\g{\gamma}

\def\G{\Gamma}
\def\a{\alpha}

\DeclareMathOperator{\wtan}{Tan}
\newcommand{\sq}{[0,1]^2}
\newcommand{\res}{\hbox{ {\vrule height .22cm}{\leaders\hrule\hskip.2cm} }}

\newcommand{\whocare}[1]{}

\newcommand{\eps}{\varepsilon}
\DeclareMathOperator{\dimh}{dim_H}
\DeclareMathOperator{\diam}{diam}
\DeclareMathOperator{\card}{card}
\DeclareMathOperator{\w}{{\bf w}}

\DeclareMathOperator{\dist}{dist}

\DeclareMathOperator{\Lip}{Lip}

\DeclareMathOperator{\exc}{excess}

\DeclareMathOperator{\tang}{Tan}
\DeclareMathOperator{\disth}{dist_H}

\numberwithin{equation}{section}

\title{H\"older curves with exotic tangent spaces}

\author{Eve Shaw}
\author{Vyron Vellis}

\thanks{E.S. and V.V. were partially supported by NSF DMS grant 2154918.}
\date{\today}
\subjclass[2020]{Primary 28A80; Secondary 26A16, 28A75, 53A04}
\keywords{H\"older curves, parameterization, iterated function systems, tangents}

\address{Department of Mathematics\\ The University of Illinois\\ Urbana, IL, 61801}
\email{emshaw@illinois.edu}
\address{Department of Mathematics\\ The University of Tennessee\\ Knoxville, TN 37966}
\email{vvellis@utk.edu}

\begin{document}

\begin{abstract}
An important implication of Rademacher's Differentiation Theorem is that every Lipschitz curve $\G$ infinitesimally looks like a line at almost all of its points in the sense that at $\mathcal{H}^1$-almost every point of $\G$, the only tangent to $\G$ is a straight line through the origin. In this article, we show that, in contrast, the infinitesimal structure of H\"older curves can be much more extreme. First, we show that for every $s>1$ there exists a $(1/s)$-H\"older curve $\G_s$ in a Euclidean space with $\mathcal{H}^s(\G_s)>0$ such that $\mathcal{H}^s$-almost every point of $\G_s$ admits infinitely many topologically distinct tangents. Second, we study the tangents of self-similar connected sets (which are canonical examples of H\"older curves) and prove that the curves $\G_s$ have the additional property that $\mathcal{H}^s$-almost every point of $\G_s$ admits infinitely many homeomorphically distinct tangents to $\G_s$ which are not admitted as (not even bi-Lipschitz to) tangents to any self-similar set at typical points. 
\end{abstract}

\maketitle

\section{Introduction}

Rademacher's Theorem, one of the most important theorems in geometric measure theory, states that every Lipschitz function defined on $[0,1]$ is differentiable at $\mathcal{H}^1$-almost every point of $[0,1]$. It is natural to ask whether a similar result may hold for more general functions. Calderon \cite{Calderon51} extended Rademacher's Theorem by proving that every function in the Sobolev class $\mathcal{W}^{1,p}$ with $p>1$ is $\mathcal{H}^1$-almost everywhere differentiable. However, any further generalization would be futile as for each $s>1$ there exists a Weierstrass function on $[0,1]$ which is $\frac1{s}$-H\"older but nowhere differentiable \cite{Zygmund}.

A major application of Rademacher's Theorem is towards the understanding of the ``infinitesimal structure'' of Lipschitz curves (i.e. Lipschitz images of $[0,1]$). To state this application, let us first define the notion of tangents. Following \cite{BL}, given a closed set $X\subset \R^n$ and a point $x\in X$, we say that a closed set $T$ is a \emph{tangent of $X$ at $x$} if there exists a sequence of positive scales $r_j$ that go to zero such that the blow-up sets $r_j^{-1}(X-x)$ converge to the set $T$ in the Attouch-Wets topology; see \textsection \ref{subsec:tangents} for the precise definition. Other notions of metric space convergences which produce similar tangents are known in the literature; see \cite{Gromov1,Gromov2,dreams,Furstenberg}. Another well-known notion of infinitesimal structure in geometric measure theory is that of the \emph{tangent cone} \cite[3.1.21]{Federer}. The tangent cone of $X$ at $x$ is the union of all tangents of $X$ at $x$, which means that some local information is lost. We denote by $\tang(X,x)$ the collection of all tangents of $X$ at $x$. It is well-known that $\tang(X,x)$ is nonempty and that if $T$ is in $\tang(X,x)$, then $\lambda T$ is in $\tang(X,x)$ for every $\lambda>0$. 

By Rademacher's Theorem and by a theorem of Besicovitch \cite{Besi} (see also \cite[Corollary 3.15]{falconerfracgeo}), every Lipschitz curve is infinitesimally a line at $\mathcal{H}^1$-almost every point. More precisely, the following is true, and we provide a proof in Section \ref{sec:liptan1}. 

\begin{theorem}\label{prop:liptan}
If $f : [0,1] \to \R^N$ is Lipschitz, then for $\mathcal{H}^1$-a.e. $x\in f([0,1])$, there exists a straight line $L\subset \R^N$ through the origin such that $\tang(f([0,1]),x) = \{L\}$.
\end{theorem}

%More generally, if $X$ is the Lipschitz (or even $p$-Sobolev for $p>n$) image of $[0,1]^n$ into some $\R^m$ with $m>n$, then for $\mathcal{H}^n$-a.e. $x \in X$, the tangent space of $X$ at $x$ contains exactly one element which is an $n$-plane \cite{BKV}. 
Here and for the rest of this paper, given $\a>0$, we denote by $\mathcal{H}^\a(X)$ the Hausdorff $\a$-dimensional measure of a metric space $X$. It is worthwhile to note that while typical points of a Lipschitz curve have a simple tangent space, exceptional points may exhibit extreme behaviors. In particular, there exists a Lipschitz curve $\Gamma \subset \R^2$ (or in any $\R^N$ with $N\geq 2$) and a point $x_0\in\Gamma$ such that $\tang(\Gamma,x_0)$ contains \emph{every possible tangent}; see Appendix \ref{sec:liptan2} for a precise statement and the proof.

The infinitesimal structure of Lipschitz curves plays an important role in the classification of 1-rectifiable sets by Jones \cite{jonesatsp} and Okikiolu \cite{oki}. A feature of the proofs is the use of Jones's \emph{beta numbers}, which roughly measure how well a given set $E$ can be approximated by lines at a given scale and location, and showing that if these values are small enough at all scales and locations then $E$ can be captured in a rectifiable curve. Roughly speaking, a set $E$ is contained in a Lipschitz curve if and only if at almost all points, tangent spaces are lines in a strong quantitative way. For a more complete discussion of the history of the Analyst's Traveling Salesperson Theorem, see \cite{raanan} and the citations therein.

Recent years have seen great interest in obtaining a H\"older version of the Analyst's Traveling Salesperson Theorem, that is, in characterizing all sets which are contained in a H\"older curve; see for example \cite{MM93,Remes,MM00,BV1,BNV,BZ,BV2,SV,BS24}. This notion of ``H\"older rectifiability'' is greatly motivated by the fact that many fractals in analysis on metric spaces admit a H\"older parameterization but not a Lipschitz one. For example, if the attractor of an iterated function system is connected (e.g. the von Koch curve, the Sierpi\'nski gasket, the Sierpi\'nski carpet), then it admits a H\"older parameterization by $[0,1]$; see \cite{Remes,BV2}. 

In lieu of the Lipschitz rectifiability results discussed above, a crucial step towards a comprehensive theory of H\"older rectifiable sets would be to understand the tangent spaces of H\"older curves at typical points. Unfortunately, unlike the Lipschitz case, it would be naive to expect that given $s>1$ there exists one single set $T_s$ which is the tangent of all $\frac1{s}$-H\"older curves at typical points. For example, let $\mathcal{S}$ be a self-similar Koch-like snowflake of dimension $s=\log{16}/\log{5}$ and let $\mathcal{C}$ be a Sierpi\'nski-type carpet obtained by dividing the unit square into 25 pieces and removing the 9 center squares; see Figure \ref{fig:IFS}.
%consider the two self-similar sets $\mathcal{S}$ and $\mathcal{C}$ constructed as follows. Set $S_{0}=T$ where $T$ is an isosceles triangle with side-lengths $1, 5^{-1/4}, 5^{-1/4}$, and for each $n\in\N$ set $S_{n} = \phi_1(S_{n-1})\cup \phi_2(S_{n-1})$ where $\phi_1,\phi_2$ are the two similarities in the left image of Figure \ref{fig:IFS}. The sets $S_{n}$ converge to a snowflake-like curve $\mathcal{S} = \bigcap_{n\in\N}S_{n}$. Similarly, set $C_{0}=[0,1]^2$ and for each $n\in\N$ set $C_{n} = \bigcup_{i=1}^{16}\psi_i(C_{n-1})$ where $\psi_1,\dots,\psi_{16}$ are the similarities in the right image of Figure \ref{fig:IFS}. The sets $C_{n}$ converge to a carpet-like set $\mathcal{C} = \bigcap_{n\in\N}S_{n}$.
Both $\mathcal{S}$ and $\mathcal{C}$ are connected self-similar sets which have Hausdorff dimension equal to $s=\log{16}/\log{5}$ and positive $\mathcal{H}^s$ measure, and by a theorem of Remes \cite{Remes}, both are $\frac1{s}$-H\"older curves. However, the tangents of $\mathcal{S}$ at $\mathcal{H}^s$-a.e. point are infinite snowflakes while the tangents of $\mathcal{C}$ at $\mathcal{H}^s$-a.e. point are infinite carpets; hence they are topologically different.

\begin{figure}[h]
    \centering
    \begin{minipage}{0.5\textwidth}
        \centering
        \includegraphics[width=0.8\textwidth]{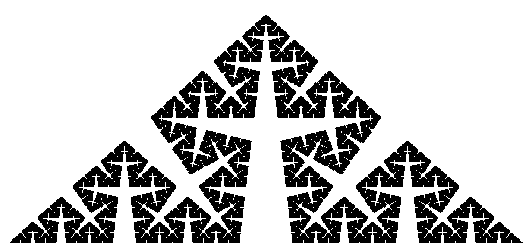} 
    \end{minipage}\hfill
    \begin{minipage}{0.4\textwidth}
        \centering
        \includegraphics[width=0.8\textwidth]{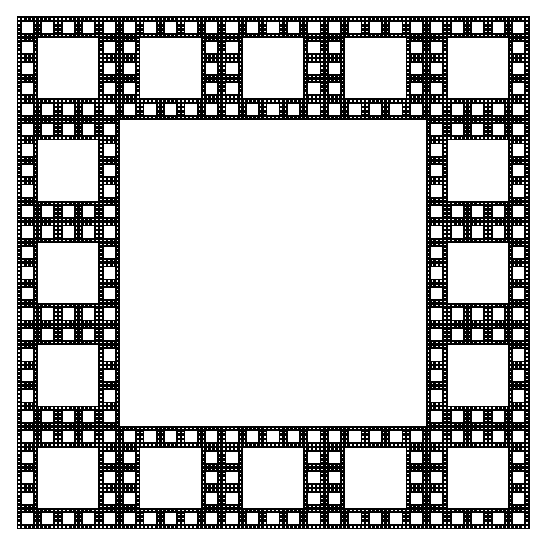} 
    \end{minipage}
    \caption{Snowflake $\mathcal{S}$ (left) and carpet $\mathcal{C}$ (right).}
    \label{fig:IFS}
\end{figure}

The previous two examples lead us to ask whether there are infinitely many H\"older curves $\G_1,\G_2,\dots$ and infinitely many sets $T_1,T_2,\dots$ such that $T_i$ are pairwise topologically different and each $T_i$ is a tangent of $\G_i$ at typical points. In our main result, we show that not only is this true but it is even worse: this situation can arise in one single H\"older curve. Furthermore, the H\"older exponent of such a curve can be chosen arbitrarily close to 1.

\begin{theorem}\label{thm:main}
For each $s>1$, there exists a $\frac{1}{s}$-H\"older curve $\Gamma_s$ in some Euclidean space such that $\mathcal{H}^s(\Gamma_s)>0$ and for $\mathcal{H}^s$-almost every $x\in \Gamma_s$, the space $\tang(\Gamma_s,x)$ contains infinitely many tangents which are pairwise not homeomorphic.
\end{theorem}

We note that little is currently known about the geometry of tangents at typical points. For example, are tangents at typical points connected? If so, are they H\"older images of $\R$? Theorem \ref{prop:liptan} answers the previous two questions in the affirmative for Lipschitz curves. 

\subsection{Tangents of self-similar sets}

Perhaps the richest source of H\"older curves with positive and finite Hausdorff measure is that of self-similar sets. In this paper, we say that a nonempty compact set $Q \subset\R^N$ is \emph{self-similar} if there exist contracting similarities $\phi_1,\dots,\phi_k$ on $\R^N$ and there exists a nonempty open set $U\subset \R^N$ such that $Q=\phi_1(Q)\cup\cdots\cup\phi_k(Q)$ and for all distinct $i,j$ we have $\phi_i(U)\subset U$ and $\phi_i(U)\cap \phi_j(U)=\emptyset$. A well-known theorem of Hutchinson \cite{Hutchinson} states that every self-similar set has finite and positive $\mathcal{H}^s$-measure for some $s>0$.

While there are examples of H\"older curves with positive and finite Hausdorff measure which are not self-similar sets (e.g. a bi-Lipschitz embedding of the snowflaked space $([0,1],|\cdot|^{\e})$ \cite{BH,JMW,RV}), these examples are all bi-Lipschitz equivalent to self-similar sets. Therefore, it is natural to ask if H\"older curves have tangents possessing some form of self-similarity at typical points. Note that tangents of Lipschitz curves at typical points are straight lines which do possess local self-similarity. This discussion motivates two questions. First, can we classify tangents of (connected) self-similar sets? Second, are tangents of general H\"older curves always bi-Lipschitz equivalent to tangents of connected self-similar sets?

Towards the first question, 
%we investigate the tangents of self-similar sets and 
we show that if $Q\subset \R^N$ is a self-similar set, then at $\mathcal{H}^{\dimh(Q)}$-almost every point $x$ of $Q$, every tangent $T$ is locally made up of large copies of $Q$: there exists a controlled number of rescaled copies $Q_1,\dots,Q_n$ of $Q$, of diameters comparable to $R$, that are ``essentially disjoint'' and
\[\overline{B}({\bf 0},R)\cap T\subset Q_1\cup\cdots\cup Q_n \subset \overline{B}({\bf 0},2R)\cap T\]
where ${\bf 0}$ denotes the origin in $\R^N$. See Proposition \ref{lem:selfsimtan} for the precise statement.
%\begin{proposition}\label{lem:selfsimtan1}
%Let $Q\subset \R^N$ be a self-similar set and let $s= \dimh{Q}$. There exists $C>1$ such that for $\mathcal{H}^s$-almost every $x\in Q$, every $T\in\tang(Q,x)$, and every $R>0$, there exist rescaled copies $Q_1,\dots,Q_n$ of $Q$ such that $n\leq C$, $\diam{Q_i}\geq C^{-1}R$ and
%\[\overline{B}({\bf 0},R)\cap T\subset Q_1\cup\cdots\cup Q_n \subset \overline{B}({\bf 0},2R)\cap T.\]
%\end{proposition}
Furthermore, for the special case of self-similar \emph{sponges} (including self-similar carpets such as the Sierpi\'nski carpet), we show in Section \ref{sec:selfsim} that at typical points of such a sponge, each tangent is \emph{locally self-similar}, which means roughly that every open ball contains a similar copy of every other open ball as an open subset.

Regarding the second question, following the construction in Theorem \ref{thm:main}, we show in Section \ref{sec:selfsim} that the curves $\G_s$ from Theorem \ref{thm:main}, infinitesimally, do not resemble self-similar sets.

\begin{theorem}\label{thm:ss}
For each $s>1$, at $\mathcal{H}^s$-a.e. $x\in \Gamma_s$ there exists $T\in \tang(\Gamma_s,x)$ with the following property. If $Q \subset \R^n$ is a self-similar set with $\dimh{Q} =s$, then the set of points $z\in Q$ for which $T$ is bi-Lipschitz homeomorphic to an element of $\tang(Q,z)$ is $\mathcal{H}^s$-null.
\end{theorem}

To the best of our knowledge, the curves of Theorem \ref{thm:ss} are the first examples of H\"older curves that $\mathcal{H}^s$-almost everywhere possess tangents which are not realizable as a tangent of a self-similar set at $\mathcal{H}^s$-almost any point. However, it is worth noting that the H\"older curves of Theorem \ref{thm:ss} have the property that at typical points, \emph{some} tangent does arise as a tangent of a self-similar set at typical points. In light of this, we leave open a weakened version of the second question as a conjecture, originally due to Matthew Badger and the second named author in 2019. 

\begin{question}[Badger, Vellis]
If $\Gamma$ is a H\"older curve in Euclidean space, then is it true that at typical points, there is some tangent to $\Gamma$ which is bi-Lipschitz to a tangent of a connected self-similar set at typical points?
\end{question}

\subsection{Outline of the proofs of main results}

To simplify the exposition, we first describe the construction of $\Gamma_s$ in the special case that $s=\log{24}/\log{6}$. The set $\Gamma_s$ is constructed by applying two sets of iterated function systems on $\R^2$ in a somewhat random fashion. We start with the unit square $\mathcal{S}^0$. Assume now that for some $k\geq 0$ we have defined $\mathcal{S}^k$ which is the union of $24^k$ closed squares with disjoint interiors $\mathcal{S}^k_{i_1\dots i_k}$, corresponding to strings $i_1\dots i_k\in\{1,\dots,24\}^k$ of length $k$ in $24$ letters. Each such square is replaced by a copy of Model 1 or a copy of Model 2 in Figure \ref{fig:1}, rescaled by $6^{-k}$. After replacing each square, we obtain $\mathcal{S}^{k+1}$. We have that $\mathcal{S}^0\supset \mathcal{S}^1 \supset \cdots$ and the set $\Gamma_s$ is the Hausdorff limit of these sets. The choice  between Model 1 or Model 2 at every stage of the construction is encoded by a \emph{choice function} $\eta$ which is defined on all finite words made from characters $\{1,\dots,24\}$ and takes values in $\{1,2\}$. See Section \ref{sec:IFS} for the rigorous construction. The important detail is that Model 1 contains one \emph{local cut-point} (i.e. a point $x$ which if removed makes a neighborhood of $x$ disconnected) while Model 2 has no such point.

\begin{figure}[h]
    \centering
    \begin{minipage}{0.4\textwidth}
        \centering
        \includegraphics[width=0.8\textwidth]{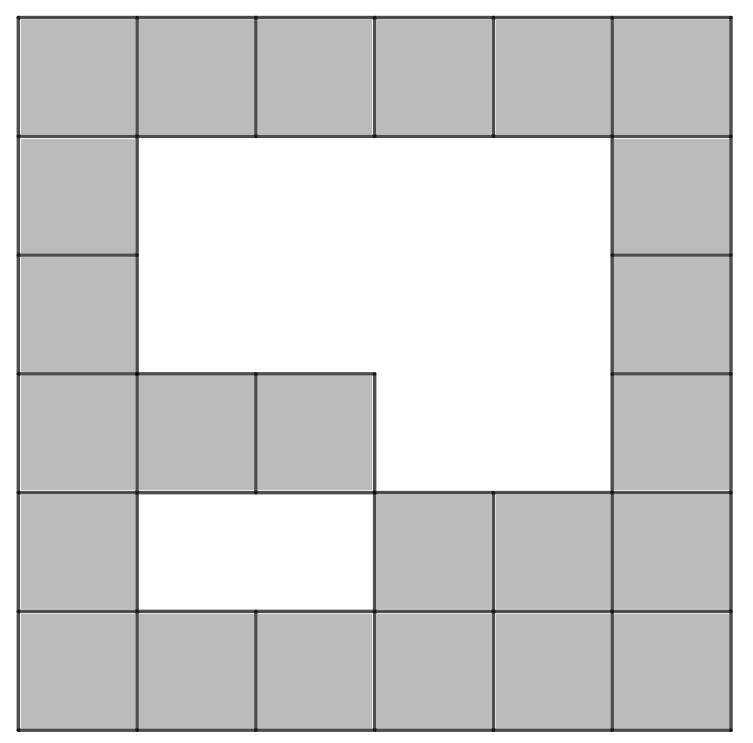} 
    \end{minipage}\hfill
    \begin{minipage}{0.4\textwidth}
        \centering
        \includegraphics[width=0.8\textwidth]{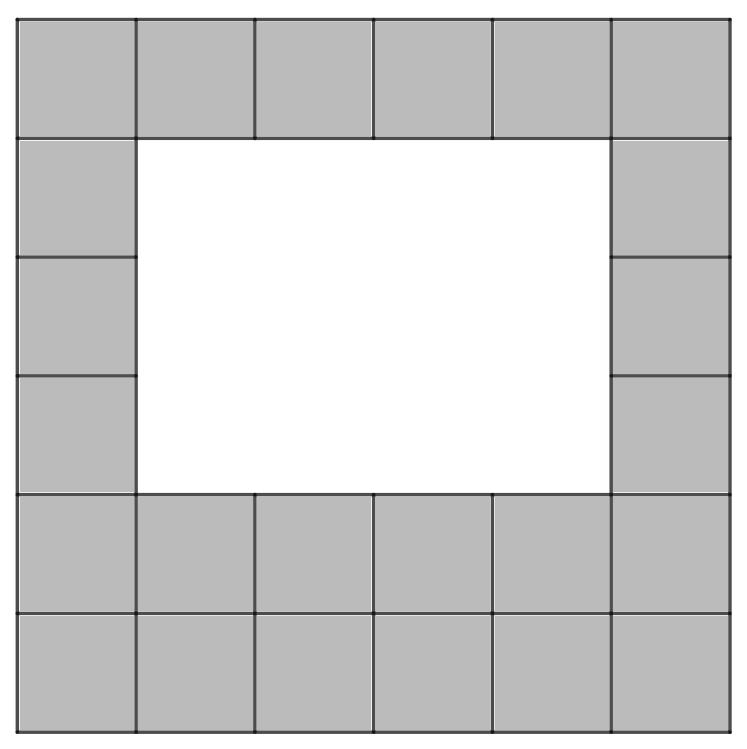} 
    \end{minipage}
    \caption{Model 1 (left) and Model 2 (right) for $s=\log{24}/\log{6}$.}
    \label{fig:1}
\end{figure}

In Section \ref{sec:measure} we show that, no matter what the choice function $\eta$ is, the resulting ``limit set'' $\Gamma_s$ will satisfy $0<\mathcal{H}^s(\Gamma_s) < \infty$. 

In Section \ref{sec:param}, perhaps the most technical part of this paper, drawing inspiration from the techniques in parameterization results of attractors of iterated function systems \cite{BV2}, we construct, for each choice function $\eta$, a sequence of Lipschitz curves which approximate $\Gamma_s$ and converge to a $\frac{1}{s}$-H\"older parameterization of $\Gamma_s$. We note that Remes's theorem \cite{Remes} asserts the existence of a H\"older parameterization for self-similar sets but does not give the parameterization itself. Constructions of such parameterizations exist in the literature \cite{RZ1,RZ2} for special cases of self-similar sets (including self-similar carpets). Our approach in Section \ref{sec:param} can be used to construct such parameterizations for statistically self-similar carpets and sponges.

Note that if a choice function takes only one value (say 1), then the resulting set $\Gamma_s$ is self-similar, and in fact, it is a self-similar \emph{sponge} which, as we prove, implies that it has very nice tangent spaces; see \textsection\ref{sec:sponge}. Thus in order to obtain a set where at typical points, the tangent space contains infinitely many topologically different elements which cannot be obtained from a self-similar set, the choice function necessarily must exhibit some form of randomness. In Section \ref{sec:choice}, using a measure-theoretic argument on the set of choice functions, we show that almost every choice function (in terms of a suitable probability measure) yields a continuum $\Gamma_s$ which ``sees'' both models in arbitrarily small neighborhoods at $\mathcal{H}^s$-almost every point. With such a choice function in hand, we show in Section \ref{sec:tangents} that at $\mathcal{H}^s$-almost every $x\in \Gamma_s$, and for each $k\in\N$, there exists a tangent $T\in \tang(\Gamma_s,x)$ that contains exactly $\frac{1}{23}(24^k-1)$ many \emph{local cut-points} (see \textsection\ref{sec:cut-points}). Therefore, there are infinitely many topologically different tangents. 

A construction similar to the one described above works for all values $s$ of the form $\a_n = \frac{\log(5n-6)}{\log(n)}$ where $n\geq 4$ is an even number. To obtain Theorem \ref{thm:main} for an arbitrary $s>1$, we first choose $n\geq 4$ even such that $s>\a_n$. Then in Section \ref{sec:proof} we appropriately ``snowflake'' the curve $\Gamma_{\a_n}$, apply Assouad's embedding theorem \cite{Assouad} to bi-Lipschitz embed the snowflaked $\Gamma_{\a_n}$ into some Euclidean space, and use a weak-tangent-snowflaking argument to show that the embedded image (which has dimension $s$) satisfies the conclusions of Theorem \ref{thm:main}.

Finally, in Section \ref{sec:selfsim} we prove that these tangents cannot be obtained as tangents of self-similar sets at generic points.

\subsection*{Acknowledgements} We thank the anonymous referees for their valuable comments which greatly improved the exposition of the paper.

\section{Preliminaries}\label{sec:prelim}

Given two points $a,b \in \R^N$, we denote by $|a-b|$ the Euclidean distance between these points. Given a point $a\in\R^N$ and a set $X\subset \R^N$ we write $\dist(a,X):= \inf_{x\in X}|a-x|$. Finally, given sets $X,Y \subset \R^N$, we write 
\[ \dist(X,Y):= \inf_{x\in X, y\in Y}|x-y| = \inf_{x\in X}\dist(x,Y) = \inf_{y\in Y}\dist(y,X).\]
Furthermore, given nonempty sets $A,B\subset \R^N$, we define the \emph{excess of A over B} as 
\[ \exc(A,B)=\sup_{a\in A}\inf_{b\in B} |a-b|;\] 
see \cite[\textsection 3.1]{Beer}, \cite{BL}, and \cite[Appendix A]{BET}. We define also the \emph{Hausdorff distance} between $A$ and $B$ by
\[\disth(A,B):=\max\left\{\exc(A,B),\exc(B,A)\right\}.\]
Unlike the Hausdorff distance, excess is not symmetric. Clearly, $\exc(A,B) \leq \disth(A,B)$.

In the next remark, we list six important properties of the excess which will be used heavily throughout this article. The first and fifth properties follow straight from the definition while the other four are given in \cite[Section 2]{BL}.

\begin{rem}\label{rem:excess}
The excess satisfies the following six properties.
\begin{enumerate}
\item \emph{Translation invariance.} For nonempty sets $A,B\subset \R^N$ and any point $x\in\R^N$, \[\exc(A,B)=\exc(A+x,B+x).\]
\item \emph{Triangle inequality.} For nonempty sets $A,B,C\subset\R^N$,
\[\exc(A,C)\leq\exc(A,B)+\exc(B,C).\]
\item \emph{Containment.} For nonempty sets $A,B\subset \R^N$, $\exc(A,B)=0$ if, and only if, $A\subset \overline{B}$.
\item \emph{Monotonicity.} If $A\subset A'$, $B'\subset B$ are all nonempty subsets of $\R^N$, then
\[\exc(A,B)\leq \exc(A',B').\]
\item \emph{Subadditivity.} If $A,B,C\subset \R^N$ are all nonempty, then
\[\exc(A\cup B,C)\leq \exc(A,C)+\exc(B,C).\]
\item \emph{Closure.} If $A,B\subset \R^N$ are nonempty, then
\[\exc(A,B)=\exc(\overline{A},\overline{B}).\]
\end{enumerate}
\end{rem}

\subsection{Tangents}\label{subsec:tangents}

Let $\mathfrak{C}(\R^N)$ be the collection of nonempty closed subsets of $\R^N$, and let $\mathfrak{C}(\R^N;{\bf 0})$ be the collection of nonempty closed subsets of $\R^N$ containing the origin ${\bf 0}$. We consider both of these spaces equipped with the \emph{Attouch-Wets} topology, which is defined in \cite[Definition 3.1.2]{Beer}. 

\begin{lemma}[{\cite[Theorem 3.1.7]{Beer}, \cite[Lemma 8.2]{dreams}}]
There exists a metrizable topology on $\mathfrak{C}(\R^N)$, called the Attouch-Wets topology, in which a sequence of sets $(X_m)_{m=1}^{\infty}\subset \mathfrak{C}(\R^N)$ converges to a set $X\in\mathfrak{C}(\R^N)$ if and only if for every $r>0$, 
\[ \lim_{m\to\infty} \exc(X_m\cap \overline{B}({\bf 0},r),X)=0 \quad\text{and}\quad \lim_{m\to\infty} \exc(X\cap \overline{B}({\bf 0},r),X_m)=0.\] 
Moreover, the subcollection $\mathfrak{C}(\R^N;{\bf 0})$ is sequentially compact; that is, for any sequence $(X_m)_{m=1}^{\infty}\subset \mathfrak{C}(\R^N;{\bf 0})$ there exists a subsequence $(X_{m_j})_{j=1}^{\infty}$ and a set $X \in \mathfrak{C}(\R^N;{\bf 0})$ such that $(X_{m_j})_{j=1}^{\infty}$ converges to $X$.
\end{lemma}

In the following lemma, we record a useful property of Attouch-Wets convergence in $\mathfrak{C}(\R^N;{\bf 0})$, which roughly says that sequences in $\mathfrak{C}(\R^N;{\bf 0})$ that converge in the Attouch-Wets topology satisfy a type of Cauchy condition with respect to excess.
\begin{lemma}\label{lem:excauchy}
If $(A_j)_{j\in\N}\subset \mathfrak{C}(\R^N;{\bf 0})$ is a sequence converging to a set $A\in\mathfrak{C}(\R^N;{\bf 0})$ with respect to the Attouch-Wets topology, then for every $r>0$ and each $\e>0$, there exists an integer $j_0\in\N$ such that for every pair of integers $j_1,j_2\geq j_0$, $\exc(A_{j_1}\cap \overline{B}({\bf 0},r),A_{j_2})<\e$.
\end{lemma}
\begin{proof}
Let $\e>0$ and let $r>0$. Since $A_j\to A$ in the Attouch-Wets topology as $j\to\infty$, there exists an integer $j_0\in\N$ such that for every integer $j\geq j_0$
\[\exc(A_j\cap\overline{B}({\bf 0},r),A)<\min\{\e/2,r\}\quad\text{ and }\quad\exc(A\cap \overline{B}({\bf 0},2r),A_j)<\min\{\e/2,r\}.\]
Then for every pair of integers $j_1,j_2\geq j_0$, by the triangle inequality for excess (see Remark \ref{rem:excess}),
\[\exc(A_{j_1}\cap\overline{B}({\bf 0},r),A_{j_2})\leq\exc(A_{j_1}\cap\overline{B}({\bf 0},r),A\cap\overline{B}({\bf 0},2r))+\exc(A\cap\overline{B}({\bf 0},2r),A_{j_2}).\]
Now since $j_1\geq j_0$, $\exc(A_{j_1}\cap \overline{B}({\bf 0},r),A)<\min\{\e/2,r\}$, thus for every $x\in A_{j_1}\cap \overline{B}({\bf 0},r)$ there exists a point $y\in A$ such that $|x-y|<\min\{\e/2,r\}$. Then it must hold that such a point $y$ is contained in $\overline{B}({\bf 0},2r)\cap A$, and therefore $\exc(A_{j_1}\cap\overline{B}({\bf 0},r),A\cap\overline{B}({\bf 0},2r))<\e/2$. Additionally, since $j_2\geq j_0$, we have also that $\exc(A\cap\overline{B}({\bf 0},2r),A_{j_2})<\e/2$, which completes the proof.
\end{proof}

\begin{definition}[{\cite[Definition 3.1]{BL}}]
Let $X\in \mathfrak{C}(\R^N)$ and let $x\in X$. We say that a set $T\in\mathfrak{C}(\R^N;{\bf 0})$ is a \emph{tangent set of $X$ at $x$} if there exists a sequence of scales $(r_m)_{m\in\N}>0$ such that $r_m\to 0$ and 
$\frac{X-x}{r_m}\to T$ with respect to the Attouch-Wets topology. We denote by $\tang(X,x)$ the set of all tangent sets of $X$ at $x$. 
\end{definition}

Since $\mathfrak{C}(\R^N;{\bf 0})$ is sequentially compact, if $X\in\mathfrak{C}(\R^N)$ and if $x_0\in X$, then the sequence of sets $(n(X-x_0))_{n\in\N}$ has a subsequential limit in $\mathfrak{C}(\R^N;{\bf 0})$ and such a limit must be a tangent of $X$ at $x_0$. Therefore, for any $X\in\mathfrak{C}(\R^N)$ and any $x_0\in X$, we have that $\emptyset \neq \tang(X,x_0) \subset \mathfrak{C}(\R^N;{\bf 0})$.

For a notion of tangents with respect to Hausdorff metric (\emph{microsets}) see \cite{Furstenberg} and see \cite{Buczolich2003,Buczolich2006} for microsets of continuous functions. 

\begin{lemma}\label{lem:unbounded}
If $X\subset \R^N$ is a nondegenerate continuum, $x\in X$, and $T\in \tang(X,x)$, then each component of $T$ is unbounded.
\end{lemma}

\begin{proof}
Fix $T\in \tang(X,x)$. First, we prove that the component of $T$ which contains ${\bf 0}$ is unbounded. To this end, let $(r_j)>0$ be a sequence of scales converging to $0$ such that $(r_j)^{-1}(X-x)\to T$ as $j\to\infty$ in the Attouch-Wets topology. Given $j\in\N$, define the similarity $g_j:X\to \R^N$ by $g_j(y)=(r_j)^{-1}(y-x)$, and for $R>0$ let $Q_{j,R}$ be the component of $g_j(X)\cap\overline{B}({\bf 0},R)$ containing ${\bf 0}$. Since $X$ is connected and nondegenerate, for each $R>0$ there exists a $j_R\in\N$ such that $Q_{j,R}\cap\partial B({\bf 0},R)\neq\emptyset$ for every integer $j\geq j_R$. Note that for every pair of scales $R_2>R_1>0$ and for each $j\in\N$, $Q_{j,R_1}\subset Q_{j,R_2}$. By compactness of $\mathfrak{C}(\R^N;{\bf 0})$, for each $R>0$ the sequence $(Q_{j,R})_{j\in\N}$ has an Attouch-Wets subsequential limit $K_R$, and for $R_2>R_1>0$ we have $K_{R_1}\subset K_{R_2}$. Additionally, since $g_j(X)$ converges to $T$ in the Attouch-Wets topology, we have that $K_R\subset T$, ${\bf 0}\in K_R$, and $K_R\cap\partial B({\bf 0},R)\neq\emptyset$ for every $R>0$. Therefore, $\bigcup_{m\in \N} K_m\subset T$ and the former union is unbounded, connected, and contains ${\bf 0}$, hence the component of $T$ containing ${\bf 0}$ is unbounded, and thus $T$ has an unbounded component.

Assume now for a contradiction that $T$ has a bounded component $K$ and let $\mathbb{S}^N$ denote the 1--point compactification of $\R^N$. Let $L$ be the component of $T'=T\cup\{\infty\}$ in $\mathbb{S}^N$ that contains $\infty$. Then clearly $K,L$ are distinct connected components of $T'$. Regarding $K$ and $L$ as quasi-components of $T'$, one can choose two disjoint closed sets $Y, Z \subset T'$ satisfying $K\subset Y$, $L\subset Z$, and $T'=Y\cup Z$. Then both $Y$ and $Z$ are closed in $\mathbb{S}^N$, and thus by normality there exists an open set $U$ such that
\[ Y \subset U \subset \overline{U} \subset \mathbb{S}^N \setminus Z.\]
Since $\infty \notin \overline{U}$, we have that $\overline{U}$ is a compact subset of $\R^N$. 
Set $d=\dist(Z,Y)$.

Fix $R>0$ large enough that $Y\subset \overline{B}({\bf 0},R)$ and $Z\cap \overline{B}({\bf 0},R)\neq\emptyset$. By definition of Attouch-Wets convergence, there exists a $j_0\in\N$ such that for every integer $j\geq j_0$, 
\[ \exc(r_j^{-1}(X-x)\cap \overline{B}({\bf 0},R),T)+\exc(T\cap \overline{B}({\bf 0},R),r_j^{-1}(X-x))<d/3.\] 
Let $y,z\in r_j^{-1}(X-x)$ satisfy $\dist(y,Y)<d/3$ and $\dist(z,Z)<d/3$. Since $X$ is connected, there is another point $p\in r_j^{-1}(X-x)$ such that $p\in \overline{B}({\bf 0},R)$ and $\dist(p,Y)= d/3$; thus $\dist(p,Z)\geq 2d/3$.
However, it follows that $\dist(p,T)\geq d/3$, which is a contradiction.
\end{proof}

\subsection{Dendrites}
A \emph{dendrite} is a Peano continuum (that is, compact, connected, and locally connected at every point) which contains no simple closed curves. The \emph{leaves} of a dendrite $X$ are exactly those points $x\in X$ such that $X\setminus \{x\}$ is connected.

\begin{lemma}\label{lem:uniondendrites}
Let $T_1,T_2$ be two dendrites in $\R^n$ that intersect at a point. Then $T_1\cup T_2$ is a dendrite.
\end{lemma}

\begin{proof}
Recall that a metric space is a dendrite if and only if any two distinct points can be separated by a third point \cite[Theorem X.10.2]{Nad}. Suppose that $T_1\cap T_2 = \{x_0\}$. Fix distinct $x,y \in T_1\cup T_2$. If $x,y \in T_1$, since $T_1$ is a dendrite, there exists a point $z\in T_1$ that separates $x$ and $y$. Similarly if $x,y \in T_2$. If $x\in T_1\setminus T_2$ and $y\in T_2\setminus T_1$, then $x_0$ separates $x$ and $y$.
\end{proof}

\subsection{Words}\label{sec:words}

For each even integer $n\geq 4$ we denote
\[ \mathcal{A}_n = \{1,\dots,5n-6\}.\]
Given $n$ as above and an integer $m\geq 0$, let $\mathcal{A}_n^m$ be the set of words of length $m$ formed by characters in $\mathcal{A}_n$, with the convention that $\mathcal{A}_n^0=\{\eps\}$, where $\eps$ is the empty word. Define $\mathcal{A}_n^*:=\bigcup_{m=0}^\infty \mathcal{A}_n^m$ and $\mathcal{A}_n^{\N}$ be the set of infinitely countable words. 

Given a word $w\in \mathcal{A}_n^*$, we denote by $|w|$ the length of $w$. Also, given $w = i_1\cdots i_m \in \mathcal{A}_n^m$ and $j\leq m$, we write $w(j) = i_1\cdots i_j$, and similarly given an integer $j\geq 0$ and an infinite word $\tau=i_1 i_2 \dots\in\mathcal{A}_n^\N$ we write $\tau(j)=i_1 i_2\dots i_j$. For any $w\in \mathcal{A}_n^*$, define the \emph{cylinder set} $\mathcal{A}_{n,w}^\N:=\{\tau\in \mathcal{A}_n^\N:\tau(m)=w\}$, that is, the set of infinite words which agree with $w$ for the first $m$ characters. Similarly, for $j,m\in \N$ with $m\geq j$ and for each $w\in \mathcal{A}_n^j$, we define $\mathcal{A}_{n,w}^m:=\{v\in \mathcal{A}_n^m: v(j)=w\}$.

Denote by $\Sigma_n$ the $\sigma$-algebra generated by the cylinders $\mathcal{A}^{\N}_{n,w}$ where $w\in \mathcal{A}_n^{*}$. Then there exists a unique probability measure $\nu_n:\Sigma_n \to [0,1]$ such that $\nu_n(\mathcal{A}^{\N}_{n,w})= (5n-6)^{-|w|}$ for all $w\in \mathcal{A}_{n}^{*}$; see for example \cite[\textsection 3.1]{Stroock}.

\subsection{Combinatorial graphs and trees}
A \emph{combinatorial graph} is a pair $G=(V,E)$ of a finite or countable vertex set $V$ and an edge set 
\[E \subset \{\{v,v'\} : v,v' \in V\text{ and }v\neq v'\}.\]
If $\{v,v'\}\in E$, we say that the vertices $v$ and $v'$ are \textit{adjacent} in $G$.

If $V\subset \R^n$ for some $n\in\N$, then we define the \emph{image of $G$} to be the set
\[ \text{Im}(G) := \bigcup_{\{v,v'\}\in E}[v,v']\]
where $[v,v']$ denotes the line segment from $v$ to $v'$. Recall that if $v\in V$ is a vertex, then the \emph{valence} of $v$ in $G$ is the number of vertices $u\in V\setminus\{v\}$ so that $\{u,v\}\in E$.

A \emph{simple path} in $G$ is a set $ \g = \{\{v_1,v_2\}, \{v_2,v_3\}, \dots, \{v_{n-1},v_n\}\} \subset E$ such that for all distinct $i,j \in \{1,\dots,n\}$ we have $v_i\neq v_j$; in this case we say that $\g$ joins $v_1$, $v_n$. 
A graph $T = (V,E)$ is a \emph{combinatorial tree} if for any distinct $v,v'$ there exists a unique simple path $\g$ that joins $v$ with $v'$.

\subsection{Local cut-points}\label{sec:cut-points}

Recall if $X\subset \R^N$, then a point $x\in X$ is called a \emph{cut-point} if $C\setminus\{x\}$ is not connected, where $C\subset X$ is the component of $X$ containing $x$. For a nondegenerate closed set $X\subset \R^N$ in a Euclidean space, following Whyburn \cite{whyburn}, we say that a point $x_0\in X$ is a \emph{local cut-point of $X$} if there exists some $r>0$ such that $C_X(x_0,r)\setminus\{x_0\}$ is not connected, where $C_X(x_0,r)$ is the component of $\overline{B}(x_0,r)\cap X$ containing $x_0$. That is, $x_0$ is a local cut-point if $x_0$ is a cut-point in sufficiently small neighborhoods of itself.

\begin{lemma}\label{lem:cuthom}
Let $X\subset \R^N$ and $Y\subset \R^M$ be nondegenerate closed subsets of Euclidean spaces and let $f:X\to Y$ be a homeomorphism. If $p$ is a local cut-point of $X$, then $f(p)$ is a local cut-point of $Y$.
\end{lemma}

\begin{proof}
Let $p$ be a local cut-point of $X$ and let $r>0$ such that $C_X(p,r)\setminus\{p\}$ is disconnected. As a matter of notation, throughout this proof when we write $B_X(p,r)$ we mean $B(p,r)\cap X$, and similarly when we write $B_Y(f(p),r)$ we mean $B(f(p),r)\cap Y$. Let $\e>0$ such that $\overline{B}_Y(f(p),\e)\subset f(\overline{B}_X(p,r))$, and let $\delta\in (0,r)$ satisfy
\[f(\overline{B}_X(p,\delta))\subset\overline{B}_Y(f(p),\e)\subset f(\overline{B}_X(p,r)).\]
Then we have that
\[f(C_X(p,\delta))\subset C_Y(f(p),\e)\subset f(C_X(p,r)).\]
Let $A,B\subset C_X(p,r)\setminus\{p\}$ be a disjoint pair of nonempty closed subsets of $C_X(p,r)\setminus\{p\}$ such that $A\cup B=C_X(p,r)\setminus\{p\}$. Note that $A\cup\{p\}$ and $B\cup\{p\}$ are both closed sets in $X$. We claim that
$A\cap C_X(p,\delta)\neq\emptyset$.

First, we show that $A\cup\{p\}$ is connected. If $A\cup\{p\}$ is not connected, then let $C,D$ be a disjoint pair of nonempty closed subsets of $A\cup\{p\}$ with $C\cup D=A\cup\{p\}$.
Then since $A\cup\{p\}$ is closed in $X$, we have that $C$ and $D$ are both closed in $X$. Without loss of generality, assume that $p\in C$. Then the pair $B\cup C$ and $D$ is a disjoint pair of nonempty closed subsets of $C_X(p,r)$ with $(B\cup C)\cup D=C_X(p,r)$, which is a contradiction, hence $A\cup\{p\}$ is connected.

Let $C_A$ be the component of $(A\cup\{p\})\cap \overline{B}_X(p,\delta) $ which contains $p$. Then $C_A=A\cup\{p\}$ or $C_A\cap \partial B_X(p,\delta)\neq\emptyset$,
and in either case we have that $C_A\neq \{p\}$. Furthermore, since $C_A$ is connected, we have that $C_A\setminus\{p\}\subset C_X(p,\delta)\setminus\{p\}$, 
so since $C_A\setminus\{p\}\subset A$, we have that $A\cap C_X(p,\delta)\neq\emptyset$. Similarly, $B\cap C_X(p,\delta)\neq\emptyset$.

Finally, we have 
\[C_Y(f(p),\e)\setminus\{f(p)\}\subset f(A)\cup f(B),\]
therefore
\[C_Y(f(p),\e)\setminus\{f(p)\}=(f(A)\cap C_Y(f(p),\e))\cup(f(B)\cap C_Y(f(p),\e)).\]
Furthermore, the sets on the right hand side form a disjoint pair of nonempty closed subsets of $C_Y(f(p),\e)\setminus\{f(p)\}$ with $(f(A)\cap C_Y(f(p),\e))\cup(f(B)\cap C_Y(f(p),\e))=C_Y(f(p),\e)\setminus\{f(p)\}$, therefore $f(p)$ is a local cut-point
of $Y$.
\end{proof}

\subsection{Self-similarity}\label{subsec:selfsim}

If a function $f:\R^N\to \R^M$ is Lipschitz continuous, then we denote by $\Lip(f)$ (the \emph{Lipschitz norm} of $f$) the smallest $L\geq 0$ such that for every pair of points $x,y\in \R^N$, $|f(x)-f(y)|\leq L|x-y|$. If $\Lip(f) < 1$, then $f$ is a \emph{contraction}. A map $f:\R^N\to \R^M$ is \emph{affine} if there exists a linear map $L:\R^N\to\R^M$ such that for every $x\in\R^N$, $f(x)=L(x)+f({\bf 0})$. A map $f:\R^N\to\R^N$ is called a \emph{similarity} if there exists $\lambda>0$ such that for every pair of points $x,y\in\R^N$, $|f(x)-f(y)|=\lambda|x-y|$. Every similarity is a composition of an orthogonal transformation, a scalar multiplication, and a translation; in particular similarity maps are affine. If $f:\R^N\to\R^N$ is a rotation-free and reflection-free similarity, then for every $x\in\R^N$, $f(x)=\Lip(f)x+f({\bf 0})$.

\begin{rem}\label{rem:affine}
A map $f:\R^N\to\R^M$ is affine if and only if there exists a linear transformation $L:\R^N\to\R^M$ such that for every $x,y\in \R^N$, $f(x)-f(y)=L(x-y)$. In particular, if $f:\R^N\to\R^N$ is a rotation-free and reflection-free similarity, then for every $x,y\in\R^N$, $f(x)-f(y)=\Lip(f)(x-y)$.
\end{rem}

An \emph{iterated function system} (IFS for short) is a finite collection $\mathcal{F}$ of contractions on $\R^N$. By a theorem of Hutchinson \cite{Hutchinson}, for each IFS $\mathcal{F}$ there exists a unique nonempty compact set $K$ (called the \emph{attractor} of $\mathcal{F}$) such that $K = \bigcup_{\phi\in\mathcal{F}}\phi(K)$. 

We say that an IFS $\mathcal{F}$ satisfies the \emph{open set condition} (OSC for short) if there exists a nonempty open set $U\subset \R^N$ such that for any distinct $\phi,\psi \in \mathcal{F}$,
\[ \phi(U) \subset U, \qquad \phi(U) \cap \psi(U) = \emptyset.\]
By a theorem of Schief \cite[Theorem 2.2]{Schief}, the OSC is equivalent to the \emph{strong open set condition} (SOSC): if $K$ is the attractor of $\mathcal{F}$, then there exists a nonempty open set $U\subset\R^N$ for which the OSC holds such that $U\cap K \neq \emptyset$.

It is well-known that if $K$ is the attractor of an IFS $\mathcal{F}$ of similarities with the OSC, then the Hausdorff dimension, the Minkowski dimension, and the Assouad dimension are all equal to $s$, where $s$ is the unique solution of the equation
\[ \sum_{\phi \in \mathcal{F}} (\text{Lip}(\phi))^s =1.\]
Moreover, $0<\mathcal{H}^{\dimh(K)}(K) < \infty$ where $\dimh(K)$ is the Hausdorff dimension of $K$ \cite{Hutchinson}. 

Henceforth, we say that a compact set $K \subset\R^N$ is \emph{self-similar} if there exists an IFS $\mathcal{F}$ of similarities with the OSC such that $K$ is the attractor of $\mathcal{F}$.

\subsection{Self-similar sponges}\label{subsec:selfsimsponge}
We say that a set $K\subset [0,1]^N$ is a \emph{self-similar sponge} if $K$ is the attractor of a system 
\begin{equation}\label{eq:defnsponge}
\{S_i(y) = k^{-1}y+p_i :[0,1]^N \to [0,1]^N\}_{i=1}^m,
\end{equation}
where $k\in \{2,3,\dots\}$, and points $\{p_1,\dots,p_m\}$ are mutually distinct and contained in the set $\{0,\frac{1}{k},\dots,\frac{k-1}{k}\}^N$. It is easy to see that the system $\{S_i\}_{i=1}^m$ satisfies the OSC with $U=(0,1)^N$. 

We say that a subset $C\subset [0,1]^N$ is a \emph{face} of $[0,1]^N$ if $C$ is of the form $C=I_1\times \dots\times I_N$, where each $I_j$ is either equal to $[0,1]$, to $\{0\}$, or to $\{1\}$. Additionally, we call $C=I_1\times \dots\times I_N$ an $M$-face of $[0,1]^N$ if exactly $M$-many of the $I_1,\dots,I_N$ are nondegenerate.

\begin{lemma}\label{lem:capface}
Let $\{S_1,\dots,S_m\}$ be a system of similarities as in \eqref{eq:defnsponge}. If $i_1,\dots,i_l\in\{1,\dots,m\}$ are distinct indices such that $\bigcap_{j=1}^l S_{i_j}([0,1]^N)\neq\emptyset$, then this intersection is equal to $S_{i_1}(C)$, where $C$ is a face of $[0,1]^N$.
\end{lemma}

\begin{proof}
For $z\in\R^N$ and $m\in\{1,\dots,N\}$, denote by $z(m)$ the $m$-th coordinate of $z$.

We begin by proving the result in the case $l=2$. For simplicity, assume that $i_1=1$ and $i_2=2$. Since $S_{1}([0,1]^N)\cap S_{2}([0,1]^N) \neq \emptyset$ and since $S_1((0,1)^N)\cap S_2((0,1)^N)=\emptyset$, we have that 
\[e:= S_1({\bf 0})-S_2({\bf 0})\in\{-k^{-1},0,k^{-1}\}^N.\] Write $e=(e(1),\dots,e(N))$. For $m\in\{1,\dots,N\}$, set $I_m=[0,1]$ if $e(m)=0$, $I_m=\{0\}$ if $e(m)>0$, and $I_m=\{1\}$ if $e(m)<0$. We claim that 
\[ S_1([0,1]^N)\cap S_2([0,1]^N)=S_1(I_1\times\dots\times I_N). \]
To see this, fix $z_1,z_2\in [0,1]^N$ such that $S_1(z_1)=S_2(z_2)$. Then $z_2-z_1=k e$ and there are three cases to consider for $m\in\{1,\dots,N\}$. If $e(m)=0$, then $z_2(m)-z_1(m)=0$ which means that $z_1(m)$ can take any value in $[0,1]$. If $e(m)>0$, then $e(m)=k^{-1}$, and $z_2(m)-z_1(m)=1$ which implies that $z_1(m)=0$. If $e(m)<0$, then $e(m)=-k^{-1}$, and $z_2(m)-z_1(m)=-1$ which implies that $z_1(m)=1$. In either of the three cases, $z_1\in I_1\times \dots\times I_N$. This proves the claim and since $I_1\times\dots\times I_N$ is a face of $[0,1]^N$, the lemma for $l=2$.

For the case $l\geq 3$, we have that
\[\bigcap_{j=1}^l S_{i_j}([0,1]^N)=\bigcap_{j=2}^l(S_{i_1}([0,1]^N)\cap S_{i_j}([0,1]^N)).\]
By the previous case, each $S_{i_1}([0,1]^N)\cap S_{i_j}([0,1]^N)=S_{i_1}(C_j)$, where $C_j$ is a face of $[0,1]^N$. The nonempty intersection of faces of $[0,1]^N$ is also a face of $[0,1]^N$, and the result follows. 
\end{proof}

\section{Tangents of Lipschitz curves at typical points}\label{sec:liptan1}

The goal of this section is to prove Proposition \ref{prop:liptan}. The proof is based on two results. The first is Rademacher's Theorem.

\begin{lemma}[{Rademacher's Theorem}]\label{lem:Rademacher}
If $f:\R^N \to \R^M$ is a locally Lipschitz continuous function, then $f$ is differentiable $\mathcal{L}^N$-a.e., where $\mathcal{L}^N$ is the $N$-dimensional Lebesgue measure on $\R^N$.
\end{lemma}

For a proof see for example \cite[Theorem 3.2]{EvansG}. The second ingredient is a result of Falconer which roughly says that Lipschitz curves look flat around typical points. Following \cite[\textsection 3.2]{falconerfracgeo}, for a point $x_0\in \R^N$, for a line $L\subset \R^N$ through the origin, and for a number $\a>0$, define 
\[ C(x_0,L,\a) :=\{x\in\R^N:d(x,L+x_0)<\a|x-x_0|\}\] 
to be the open cone centered at $x_0$ in the direction of $L$ with aperture $\a$. Given a closed set $K \subset \R^N$, we say that a point $x_0\in K$ is \emph{flat in $K$} if there exists a line $L\subset \R^N$ through the origin such that for every $\a>0$, there exists $r_{\a}>0$ with $B(x_0,r_{\a})\cap X\subset C(x_0,L,\a)$.

\begin{lemma}\label{lem:emptycone}
If $X\subset \R^N$ is a continuum with $\mathcal{H}^1(X)<\infty$, then $\mathcal{H}^1$-a.e. $x\in X$ is flat in $X$.
\end{lemma}

\begin{proof}
The claim is trivially true if $\mathcal{H}^1(X) =0$, so we may assume that $\mathcal{H}^1(X)>0$.

By \cite[Corollary 3.15]{falconerfracgeo}, for $\mathcal{H}^1$-almost every point $x\in X$, there exists a unique line $L$ passing through ${\bf 0}$ such that for every $\a>0$,
\begin{equation}\label{eq:cone}
\lim_{r\to 0}\frac{\mathcal{H}^1((X\cap \overline{B}(x,r))\setminus C(x,L,\a))}{r}=0.
\end{equation}
Let $x_0\in X$ be such a point, let $L$ be the line given above.

Assume for a contradiction that there exists $\a\in (0,1)$, there exists a sequence of positive scales $(r_j)_{j\in\N}$ going to zero, and there exists a sequence $(y_j)_{j\in\N}$ of points in $X$ such that 
\[ y_j \in (\overline{B}(x_0,r_j)\cap X)\setminus C(x_0,L,\a).\]
Since $X$ is a continuum, the component of $(\overline{B}(x_0,\tfrac{3}{2}|y_j-x_0|)\cap X)\setminus C(x_0,L,\a/2)$ which contains $y_j$ must intersect at least one of $\partial B(x_0,\tfrac{3}{2}|y_j-x_0|)$ and $\overline{C(x_0,L,\a/2)}\cap \overline{B}(x_0,\tfrac{3}{2}|y_j-x_0|)$. Therefore, 
\begin{align*}
&\mathcal{H}^1\left((\overline{B}(x_0,\tfrac{3}{2}|y_j-x_0|)\cap X)\setminus C(x_0,L,\a/2)\right)\\
&\qquad \geq \min\left\{\dist\left(y_j,\partial B(x_0,\tfrac{3}{2}|y_j-x_0|)\right),\dist\left(y_j,\overline{C(x_0,L,\a/2)}\cap \overline{B}(x_0,\tfrac{3}{2}|y_j-x_0|)\right) \right\}\\
&\qquad \geq \min\left\{\tfrac{1}{2}|y_j-x_0|,\dist\left(y_j,\overline{C(x_0,L,\a/2)}\cap \overline{B}(x_0,\tfrac{3}{2}|y_j-x_0|)\right) \right\}.
\end{align*}

To estimate the second distance, fix $z \in \overline{C(x_0,L,\a/2)}\cap \overline{B}(x_0,\frac{3}{2}|y_j-x_0|)$. Since $y_j\notin C(x_0,L,\a)$, we have $\dist(y_j,L+x_0)\geq \a|y_j-x_0|$. Moreover, $\dist(z,L+x_0)\leq (\a/2)|z-x_0|$. Fix a point $p\in L+x_0$ satisfying $|z-p|\leq \tfrac{\alpha}{2}|z-x_0|$, hence $|z-p|\leq \tfrac{3}{4}\alpha|y_j-x_0|$. Note that $|y_j-p|\geq \a|y_j-x_0|$. Therefore, $|y_j-z|\geq \tfrac{\a}{4} |y_j-x_0|$. Consequently, for each $j\in\N$,
\[\frac{\mathcal{H}^1((\overline{B}(x_0,\tfrac{3}{2}|y_j-x_0|)\cap X)\setminus C(x_0,L,\a/2))}{\tfrac{3}{2}|y_j-x_0|}\geq \frac{\a}{6}\]
which contradicts \eqref{eq:cone} as $j\to \infty$ and $\frac{3}{2}|y_j-x_0|\to 0$.
\end{proof}

We are now ready to prove Proposition \ref{prop:liptan}.

\begin{proof}[Proof of Proposition \ref{prop:liptan}]
Let $f:[0,1] \to \R^N$ be Lipschitz and let $X=f([0,1])$. By \cite[Thoerem 4.4]{AO}, there exists $g:[0,1]\to X$, an essentially two-to-one Lipschitz parameterization with constant speed equal to $2\mathcal{H}^1(X)$. By Rademacher's Theorem, for $\mathcal{H}^1$-a.e. $x\in X$, there exists $t\in g^{-1}(\{x\})$ such that $g'(t)$ exists. Furthermore, by Lemma \ref{lem:emptycone}, $\mathcal{H}^1$-a.e. point $x\in X$ is flat in $X$. Let $x_0\in X$ be a point such that $x_0$ is flat in $X$ and there exists $t\in g^{-1}(\{x_0\})$ such that $g'(t)$ exists. Let $T\in\tang(X,x_0)$, and let $L\subset \R^N$ be the line through the origin in the direction of $g'(t)$.

First, we show that $T\subset L$. To this end, let $r_j\to 0$ be a sequence of positive scales satisfying $\frac{1}{r_j}(X-x_0)\to T$ in the Attouch-Wets topology. Fix $\e>0$ and $R>0$. By Lemma \ref{lem:emptycone}, there is $j_0\in \N$, such that for every $j\geq j_0$, 
\[ B(x_0,r_j(R+\e))\cap X\subset C(x_0,L,\tfrac{\e}{2(R+\e)})\] and 
\[\exc(T\cap \overline{B}({\bf 0},R),\tfrac{1}{r_j}(X-x_0))<\e/2.\]
Note that 
\[\exc(T\cap \overline{B}({\bf 0},R),(\tfrac{1}{r_j}(X-x_0))\cap \overline{B}({\bf 0},R+\e))<\e/2\] 
as well. Then by the triangle inequality for excess from Remark \ref{rem:excess} and since $x_0$ is flat in $X$,
\begin{align*}
&\exc(T\cap \overline{B}({\bf 0},R),L)\\
&\leq\exc(T\cap \overline{B}({\bf 0},R),\tfrac{1}{r_j}(X-x_0)\cap \overline{B}({\bf 0},R+\e))+\exc(\tfrac{1}{r_j}(X-x_0)\cap \overline{B}({\bf 0},R+\e),L)\\
&<\e.
\end{align*}
Since the latter is true for all $\e>0$ and $R>0$, it follows from Remark \ref{rem:excess} that $T\subset L$.

Next, we show that $L\subset T$. Fix $\e>0$ and $R>0$. There exists $h>0$ such that for every $s\in (-hR,hR)$, 
\[\frac{1}{h}|g(t+s)-x_0-g'(t)s|<\e/2.\] 
Moreover, there exists $j_0\in\N$ such that for every $j\geq j_0$ and every $s\in (-r_j R,r_j R)$ we have
\[\frac{1}{r_j}|g(t+s)-x_0-g'(t)s|<\e.\]

Fix $y\in L\cap \overline{B}({\bf 0},\frac{R}{2\mathcal{H}^1(X)})$. For each $j\geq j_0$ there exists $s_j\in (-r_j R,r_j R)$ such that $y=g'(t)\frac{s_j}{r_j}$. Hence,
\[\left|\frac{g(t+s_j)-x_0}{r_j}-y\right| = \left |\frac{g(t+s_j)-x_0}{r_j}-g'(t)\frac{s_j}{r_j}\right| <\e.\]
Noting that $\frac{g(t+s)-x_0}{r_j}\in \frac{1}{r_j}(X-x_0)$, we obtain 
\[ \exc(L\cap \overline{B}({\bf 0},(2\mathcal{H}^1(X))^{-1}R),r_j^{-1}(X-x_0))<\e.\] 
Therefore, for every $R>0$ and every $\e>0$, there exists some $j_0\in\N$ such that for all $j\geq j_0$,
\[ \exc(L\cap \overline{B}({\bf 0},R),r_j^{-1}(X-x_0))<\e.\] 
Proceeding to the conclusion as in the previous paragraph, we have that $L\subset T$, and therefore $L=T$ as desired.
\end{proof}

\section{A ``statistically self-similar" carpet}\label{sec:IFS}

Fix for the rest of this section an even integer $n\geq 4$. Recall the alphabets $\mathcal{A}_n$ and associated word spaces from \textsection\ref{sec:words}. Divide the unit square $ [0,1]^2$ into $n^2$-many closed squares of side-length $1/n$ that have disjoint interiors and let $S_1,\dots,S_{4n-4}$ be those squares that intersect with $\partial [0,1]^2$. 

We define two iterated function systems 
\[ \mathcal{F}_n^1 =\{\psi^1_{n,j}:[0,1]^2 \to [0,1]^2\}_{j\in \mathcal{A}_n}\quad \text{and}\quad \mathcal{F}_n^2 =\{\psi^2_{n,j}:[0,1]^2 \to [0,1]^2\}_{j\in \mathcal{A}_n}\] 
as follows. 
\begin{enumerate}
\item For $j \in \{1, \dots,4n-4\}$, $\psi_{n,j}^1=\psi_{n,j}^2$ and $\psi_{n,j}^1$ is the unique composition of a translation and a dilation that maps $\sq$ onto $S_j$. 
\item For $j \in \{4n-3,\dots,4n+\frac{n}2-4\}$ define
\[ \psi_{n,j}^1(x) = \psi_{n,j}^2(x) = \tfrac1{n}x + \left( \tfrac12+\tfrac{j-1}{n},\tfrac1{n}\right).\]
\item For $j\in \{4n+\frac{n}2-3, 5n-6\}$ define
\begin{align*} 
\psi_{n,j}^1(x) = \tfrac1{n}x + \left( \tfrac{j-4n+\frac{n}2+3}{n},\tfrac2{n}\right), \qquad \psi_{n,j}^2(x) = \tfrac1{n}x + \left( \tfrac{j-4n+\frac{n}2+3}{n},\tfrac1{n}\right).
\end{align*}
\end{enumerate}
See Figure \ref{fig:1} for the first iterations of $\mathcal{F}_n^1$ and $\mathcal{F}_n^2$ in the case $n=6$.

For $j\in\{1,2\}$ we let $\psi_{n,\eps}^j$ be the identity map and for $w=i_1  \dots i_m\in \mathcal{A}_n^m$ with $m\geq 1$ we let 
\[\psi_{n,w}^j:=\psi_{n,i_1}^j\circ\cdots\circ\psi_{n,i_m}^j.\]

We set
\[ \mathscr{C}_n =\{\eta:\mathcal{A}_n^*\to\{1,2\}\}\]
and each function $\eta \in \mathscr{C}_n$ will henceforth be called a \emph{choice function}. 

Fix now a choice function $\eta\in \mathscr{C}_n$.  Set $\phi_\eps^{\eta}$ to be the identity on $\sq$, and for any $w=i_1\cdots i_m\in \mathcal{A}_n^m$ with $m\geq 1$, define 
\[ \phi_w^{\eta}:=\psi_{n,i_1}^{\eta(\varepsilon)}\circ \cdots \circ\psi_{n,i_m}^{\eta(i_1\cdots i_{m-1})}.\] 
Define now the ``attractor'' associated with the choice function $\eta$ to be
\[K^\eta:=\bigcap_{m=0}^\infty\bigcup_{w\in \mathcal{A}_n^m} \phi_w^{\eta}(\sq).\]

We list a couple of elementary facts about the set $K^{\eta}$.

\begin{lemma}
For every $\eta\in \mathscr{C}_n$, the set $K^{\eta}$ is compact.
\end{lemma}

\begin{proof}
This is immediate noting that each set $\bigcup_{w\in \mathcal{A}_n^m} \phi_w^{\eta}(\sq)$ is compact as a finite union of compact sets, and that $K^{\eta}$ is the countable intersection of compact sets.
\end{proof}

Given $w\in \mathcal{A}_n^*$, we define
\[ K_w^{\eta}:=\phi_w^{\eta}(\sq)\cap K^{\eta}.\] 

The following version of the open set condition is satisfied. 

\begin{lemma}\label{lem:OSC}
For every $\eta\in \mathscr{C}_n$, every $m\in \N$, and every pair of distinct words $w,v\in \mathcal{A}_n^m$, 
\[\phi_w^{\eta}((0,1)^2)\cap \phi_v^{\eta}((0,1)^2)=\emptyset\quad\text{and}\quad \phi_w^{\eta}((0,1)^2)\subset (0,1)^2.\] 
\end{lemma}

\begin{proof}
We proceed by induction on $m$. For $m=1$, if $i,j\in \mathcal{A}_n$ are distinct characters, then by definition we have that $\phi_j^\eta(\sq)$ and $\phi_i^\eta(\sq)$ are distinct squares of side length $1/n$ with disjoint interiors contained in $\sq$. Thus, $\phi_i^\eta((0,1)^2)\subset (0,1)^2$ and $\phi_i^\eta((0,1)^2)\cap\phi_j^\eta((0,1)^2)=\emptyset$.

Fix now an $m\in\N$ and assume that for all distinct words $v,w\in \mathcal{A}^m_n$, we have $\phi_w^\eta((0,1)^2)\cap\phi_v^\eta((0,1)^2)=\emptyset$ and $\phi_w^\eta((0,1)^2)\subset (0,1)^2$. Let $u,u'\in \mathcal{A}^{m+1}_n$ be distinct words. Then by the base case we have $\phi_u^\eta((0,1)^2)\subset \phi_{u(m)}^\eta((0,1)^2)$, which is contained in $(0,1)^2$ by the induction hypothesis.

The remainder of the proof falls to a case study.

\emph{Case 1.} If $u(m)\neq u'(m)$, then we have that $\phi_u^\eta((0,1)^2)\subset \phi_{u(m)}^\eta((0,1)^2)$ and $\phi_{u'}^\eta((0,1)^2)\subset \phi_{u'(m)}^\eta((0,1)^2)$ by the base case. Then, since $u(m)\neq u'(m)$, we have by the induction hypothesis that $\phi_{u(m)}^\eta((0,1)^2)\cap \phi_{u'(m)}^\eta((0,1)^2)=\emptyset$, therefore $\phi_u^\eta((0,1)^2)\cap \phi_{u'}^\eta((0,1)^2)=\emptyset$.

\emph{Case 2.} If $u(m)=u'(m)$, then there exist distinct characters $i,j\in \mathcal{A}_n$ such that $u=u(m)j$ and $u'=u'(m)i$. By definition of $\phi_u^\eta$ we then have
\[\phi_u^\eta((0,1)^2)=\phi_{u(m)}^\eta\circ\psi_j^{\eta(u(m))}((0,1)^2)\]
and
\[\phi_{u'}^\eta((0,1)^2)=\phi_{u(m)}^\eta\circ\psi_i^{\eta(u(m))}((0,1)^2).\]
Thus we can see that
\[\phi_{u'}^\eta((0,1)^2)\cap\phi_u^\eta((0,1)^2)=\phi_{u(m)}^\eta(\psi_j^{\eta(u(m))}((0,1)^2)\cap\psi_i^{\eta(u(m))}((0,1)^2)),\]
and the latter set is empty by the base case.
\end{proof}

\begin{lemma}\label{lem:fact}
Let $\eta\in \mathscr{C}_n$, $w,v\in \mathcal{A}_n^m$, $x\in K_w^{\eta}$, and $y\in K_v^{\eta}$. If $|x-y|>2\sqrt{2}\cdot n^{-m} $, then $K_w^{\eta}\cap K^{\eta}_v=\emptyset$. If $|x-y|<n^{-m}$, then $K_w^\eta\cap K_v^\eta\neq\emptyset$. Furthermore, for each $\eta\in\mathscr{C}_n$ and every $w\in\mathcal{A}_n^*$, we have that $\phi^\eta_w(\partial\sq)\subset K^\eta$. In particular, $\partial\sq\subset K^\eta$.
\end{lemma}

\begin{proof}
Let $X_0=\sq$, and for $m\in \N$, define $X_m=\bigcup_{w\in \mathcal{A}_n^m}\phi_w^{\eta}(\sq)$. Recall that for $j\in \{1,\dots,4n-4\}$ we have $\psi_{n,j}^1 = \psi_{n,j}^2$. Therefore, for every integer $m\geq 0$, we have that 
\[ \bigcup_{j=1}^{4n-4}\psi_{n,j}^1(X_m)\subset X_{m+1} \]
and
\[ \partial\sq\subset \bigcup_{j=1}^{4n-4}\psi_{n,j}^1(\partial\sq).\] 
Thus, since $\partial\sq\subset X_0$, we have that $\partial\sq\subset X_m$ for every integer $m\geq 0$, and since $K^{\eta}=\bigcap_{m=1}^\infty X_m$, we obtain $\partial\sq\subset K^{\eta}$. Similarly, for every $w\in\mathcal{A}_n^*$ we have that $\phi_w^\eta(\partial\sq)\subset K_w^\eta$. Let $m\geq 1$ and let $v,w\in\mathcal{A}^m_n$. If $K_v^\eta\cap K_w^\eta\neq\emptyset$, then $\diam(K_v^\eta\cup K_w^\eta)\leq 2\sqrt{2}n^{-m}$. Furthermore, if $K_v^\eta\cap K_w^\eta=\emptyset$, then $\phi_v^\eta(\partial\sq)\cap \phi_w^\eta(\partial\sq)=\emptyset$, thus $\dist(\phi_v^\eta(\sq),\phi_w^\eta(\sq))\geq n^{-m}$, and the result follows as $K_v^\eta\subset \phi_v^\eta(\sq)$.
\end{proof}

\section{Ahlfors regularity of carpets $K^{\eta}$}\label{sec:measure}

Fix for the rest of this section an even integer $n\geq 4$ and set $\a_n:=\frac{\log(5n-6)}{\log(n)}$. We show that for every choice function $\eta\in\mathscr{C}_n$, we have $0<\mathcal{H}^{\a_n}(K^{\eta})<\infty$. In fact, we show the following stronger statement.

\begin{proposition}\label{prop:reg}
There exists $C_n>1$ such that for every $\eta\in \mathscr{C}_n$, every $x\in K^{\eta}$, and every $r\in (0,1/n)$ we have
\[ C_n^{-1}r^{\a_n} \leq \mathcal{H}^{\a_n}(K^{\eta}\cap B(x,r)) \leq C_n r^{\a_n}.\]
\end{proposition}

For a choice function $\eta\in \mathscr{C}_n$ define $\pi_\eta:\mathcal{A}_n^\N\to K^\eta$ by 
\begin{equation}\label{eq:pi}
\pi_\eta(\tau):=\lim_{m\to\infty} \phi_{\tau(m)}^\eta({\bf 0}).  
\end{equation} 
We start with an elementary topological fact.

\begin{lemma}\label{lem:borel}
The $\sigma$-algebra generated by the collection $\{K_w^\eta:w\in\mathcal{A}_n^*\}$ is equal to the Borel $\sigma$-algebra on $K^{\eta}$.
\end{lemma}

\begin{proof}
Fix $\eta \in \mathscr{C}_n$. Let $\Sigma$ denote the $\sigma$-algebra generated by the collection $\{K^\eta_w:w\in\mathcal{A}_n^*\}$, and let $\mathcal{B}(K^{\eta})$ be the Borel $\sigma$-algebra on $K^{\eta}$. That $\Sigma\subset \mathcal{B}(K^{\eta})$ is clear, as each $K_w^{\eta}$ is a Borel set. For the reverse inclusion, we show that for any $x\in K^{\eta}$ and any $r>0$, $B(x,r)\cap K^{\eta}\in \Sigma$. To see this, let 
\[ W(x,r)=\{w\in \mathcal{A}_n^*: K^{\eta}_w\subset B(x,r)\}.\] 
Note that $W(x,r)$ is a countable set since $\mathcal{A}_n^*$ is countable, hence $\bigcup_{w\in W(x,r)} K^{\eta}_w\in \Sigma$. Furthermore, for each point $y\in B(x,r)\cap K^{\eta}$ there is a word $\tau\in \mathcal{A}_n^\N$ with $\pi_{\eta}(\tau)=y$. Then since $B(x,r)\cap K^{\eta}$ is an open subset of $K^\eta$ and $\lim_{j\to\infty}\diam(K^{\eta}_{\tau(j)})=0,$ there is some $m\in \N$ large enough that $\tau(m)\in W(x,r)$. Since $y\in K^\eta_{\tau(m)}$ and $\bigcup_{w\in W(x,r)} K^{\eta}_w=B(x,r)\cap K^{\eta}$, we have that $\Sigma=\mathcal{B}(K^{\eta})$.
\end{proof}

Recall from \textsection\ref{sec:words} the $\sigma$-algebra $\Sigma_n$ on $\mathcal{A}^{\N}_n$ and the probability measure $\nu_n: \Sigma_n \to [0,1]$. 
We say a word $\tau\in \mathcal{A}_n^\N$ is an \emph{injective word} if for every $\eta\in \mathscr{C}_n$, we have $\pi_\eta^{-1}(\{\pi_\eta(\tau)\})=\{\tau\}$. That is, $\tau$ is called an injective word if it uniquely defines a point in $K^{\eta}$.

\begin{lemma}\label{lem:injfull}
The set of non-injective words $\tau\in \mathcal{A}_n^\N$ is a $\nu_n$-null set.
\end{lemma}

\begin{proof}
Fix $\eta\in \mathscr{C}_n$. Recall that if $j\in\mathcal{A}_n$ satisfies $j>4n-4$, then
\[(\psi^1_j(\sq)\cup\psi^2_j(\sq))\cap\partial\sq=\emptyset.\]
Thus if $\tau= i_1 i_2\dots \in \mathcal{A}_n^\N$ and $i_{m+1}> 4n-4$ for some $m\in\N$, then $\pi_{\eta}(\tau)\notin\bigcup_{v\in \mathcal{A}_n^m} \phi^{\eta}_v (\partial\sq)$. Furthermore, if $\tau,\zeta\in \mathcal{A}_n^\N$ are distinct words with $\pi_{\eta}(\tau)=x=\pi_{\eta}(\zeta)$, then there exists some $N\in\N$ such that $\tau(N)\neq \zeta(N)$. Thus for every integer $m\geq N$, we have $\tau(m)\neq \zeta(m)$ and $x\in\phi_{\tau(m)}^{\eta}(\sq)\cap\phi_{\zeta(m)}^{\eta}(\sq)$. By Lemma \ref{lem:OSC}, this implies that $x\in\phi_{\tau(m)}^{\eta}(\partial\sq)\cap\phi_{\zeta(m)}^{\eta}(\partial\sq)$, so
\[\{x\in K^{\eta} : \pi_{\eta}^{-1}(\{x\})\text{ is not a singleton}\}\subset \bigcup_{m=1}^\infty\bigcup_{v\in \mathcal{A}_n^m} \phi_v^{\eta}(\partial\sq).\]

Moreover, for every $m\in\N$, 
\[ \bigcup_{v\in \mathcal{A}_n^m}\phi_v^{\eta}(\partial\sq)\subset\bigcup_{v\in \mathcal{A}_n^{m+1}}\phi_v^{\eta}(\partial\sq).\] 
Thus if a character $j\in\{4n-3,\dots,5n-6\}$ appears infinitely often in $\tau\in \mathcal{A}_n^{\N}$, then $\tau$ is injective; therefore the set of non-injective words is contained in the set of words for which every character $j\in\{4n-3,\dots,5n-6\}$ appears only finitely often. To see that this set is $\nu_n$-null, for each $m\in\N$ let $U_m:=\{\tau=i_1i_2\dots\in \mathcal{A}_n^\N:i_m >4n-4\}$. We have $\nu_n(U_m)=\frac{n-2}{5n-6}$ for each $m$, and if $\{U_{m_1},\dots,U_{m_j}\}$ is a collection of these sets, then
\[\bigcap\limits_{i=1}^j U_{m_i}=\{\tau=i_1 i_2\dots \in \mathcal{A}^\N_n:i_{m_i}>4n-4\text{ for each }i\in\{1,\dots,j\}\}\]
and this intersection has $\nu_n(\bigcap_{i=1}^j U_{m_i})=\left(\frac{n-2}{5n-6}\right)^j$.
Thus, $V_m:=\bigcap_{j=m}^\infty U_j$ has $\nu_n(V_m)=0$ for every $m\in\N$. Noting further that the set $F$ of words for which every character $j\in\{4n-3,\dots,5n-6\}$ appears only finitely often is contained in $\bigcup_{m=0}^\infty V_m$, we have that $F$ is $\nu_n$-null.
\end{proof}

We are now ready to prove Proposition \ref{prop:reg}.

\begin{proof}[{Proof of Proposition \ref{prop:reg}}]
Fix $\eta\in \mathscr{C}_n$ and define the pushforward measure $\pi_{\eta}\#\nu_n:\mathcal{B}(K^{\eta})\to [0,1]$ by
\[ \pi_{\eta}\#\nu_n(A) := \nu_n(\pi_{\eta}^{-1}(A)).\]
We claim that there exists $C_n>1$ such that
%the measure $\pi_{\eta}\#\nu_n$ defined on $\mathcal{B}(K^{\eta})$ by Lemma \ref{lem:borel} satisfies
\begin{equation}\label{eq:Ahlfors regularity}
C_n^{-1}r^{\a_n} \leq \pi_{\eta}\#\nu_n(B(x,r)\cap K^{\eta}) \leq C_n r^{\a_n}
\end{equation}
for all $x\in K^{\eta}$ and $r\in (0,1/n)$. Assuming \eqref{eq:Ahlfors regularity}, by \cite[\textsection 1.4.3]{Mackay}, we have that $\mathcal{H}^{\alpha_n}\res K^\eta$ satisfies \eqref{eq:Ahlfors regularity} (perhaps with a different constant $C_n$), since the measure $\pi_\eta\#\nu_n$ has the same null sets as the restricted Hausdorff measure $\mathcal{H}^{\alpha_n}\res K^\eta$.

For any $w\in \mathcal{A}_n^*$, $\mathcal{A}^\N_{n,w} \subset \pi_\eta^{-1}(K^{\eta}_w)\subset\mathcal{N}_n\cup \mathcal{A}^\N_{n,w}$, where $\mathcal{N}_n$ denotes the set of non-injective words in $\mathcal{A}^{\N}_n$. Hence by Lemma \ref{lem:injfull}, for any $w\in \mathcal{A}^*_n$ we have 
\[ \pi_{\eta}\#\nu_n(K^{\eta}_w)=(5n-6)^{-|w|}=n^{-|w|\alpha_n}.\]

Fix for the rest of the proof a point $x\in K^{\eta}$, a radius $r\in (0,1/n)$, and a word $\tau\in \mathcal{A}_n^\N$ such that $\pi_{\eta}(\tau)=x$.

For the upper bound of \eqref{eq:Ahlfors regularity}, let $m\in \N$ satisfy
\[n^{-m}\sqrt{2}>2r\geq n^{-m-1}\sqrt{2}.\]
Let $\{v_1,\dots,v_j\}\subset \mathcal{A}_n^m$ be those words with $K_{v_i}^{\eta}\cap K_{\tau(m)}^{\eta}\neq\emptyset$. By Lemma \ref{lem:OSC}, $j\leq 9$ and by Lemma \ref{lem:fact}, we have that $B(x,n^{-m})\cap K^{\eta}\subset \bigcup_{i=1}^j K^{\eta}_{v_i}$. We have also that $n^{-m}\geq r$, hence 
\[B(x,r)\cap K^{\eta}\subset B(x,n^{-m})\cap K^{\eta}\subset\bigcup\limits_{i=1}^j K^{\eta}_{v_i},\]
and therefore 
\begin{align*}
\pi_{\eta}\#\nu_n(B(x,r)\cap K^{\eta})\leq\sum_{i=1}^j\pi_{\eta}\#\nu_n(K^{\eta}_{v_i})\leq 9\pi_{\eta}\#\nu_n(K^{\eta}_{\tau(m)})&=9 n^{-m\alpha_n}\leq 9\left(\tfrac{8}{\sqrt{2}}\right)^{\alpha_n} r^{\alpha_n}.
\end{align*}

For the lower bound of \eqref{eq:Ahlfors regularity}, let $m\in\N$ satisfy
\[n^{-m}\sqrt{2}<\frac{r}{2}\leq n^{-m+1}\sqrt{2}.\]
Since $r<1/n$, we have that $m$ is at least $1$. If $y\in K^{\eta}_{\tau(m)}$, then 
\[|x-y|\leq n^{-m}\sqrt{2}<\frac{r}{2},\]
so $K^\eta_{\tau(m)}\subset B(x,r)\cap K^\eta$. Thus 
\[
\pi_{\eta}\#\nu_n(B(x,r)\cap K^{\eta}) \geq \pi_{\eta}\#\nu_n(K^{\eta}_{\tau(m)}) =  n^{-m\alpha_n} \geq (8\sqrt{2})^{-\alpha_n}r^{\alpha_n}. \qedhere
\]
\end{proof}

\section{H\"older parameterizations of carpets $K^{\eta}$}\label{sec:param}

In this section, we show that sets $K^{\eta}$ defined in Section \ref{sec:IFS} are H\"older curves. Recall the numbers $\a_n:=\frac{\log(5n-6)}{\log(n)}$ from Section \ref{sec:measure}.

\begin{proposition}\label{thm:holder}
For any even integer $n\geq 4$ and any $\eta\in \mathscr{C}_n$, there exists a $(1/\alpha_n)$-H\"older continuous surjection $f:[0,1]\to K^\eta$.
\end{proposition}

For the rest of this section we fix an even integer $n\geq 4$ and a choice function $\eta\in \mathscr{C}_n$. The dependence of sets and functions on $n$ and $\eta$ in this section is omitted.

Define the set $T\subset \sq$ by 
\begin{align*}
T &:= \left(\partial [0,1-\tfrac1{n}]^2 \setminus \left(\{1-\tfrac1{n}\} \times[1-\tfrac2{n},1-\tfrac1{n}]\right)\right) \cup \left(\{\tfrac{1}{2}\}\times[0,\tfrac{1}{n}]\right)\cup\left([\tfrac{1}{2},1-\tfrac{2}{n}]\times\{\tfrac{1}{n}\}\right)
\end{align*}
and the sets
\[T^1:=T\cup\left([0,\tfrac{1}{2}-\tfrac{1}{n}]\times\{\tfrac{2}{n}\}\right)\quad \text{and}\quad T^2:=T\cup\left([0,\tfrac{1}{2}-\tfrac{1}{n}]\times\{\tfrac{1}{n}\}\right).\]
See Figure \ref{fig2} for $n=6$. Some elementary properties of the sets $T^1,T^2$ are given in the next lemma and its proof is left to the reader.

\begin{lemma}\label{lem:T^i}
Both $T^1$ and $T^2$ are dendrites and are contained in $[0,1)^2$.
\end{lemma}

\begin{figure}[h]
    \centering
    \begin{minipage}{0.45\textwidth}
        \centering
        \includegraphics[width=0.7\textwidth]{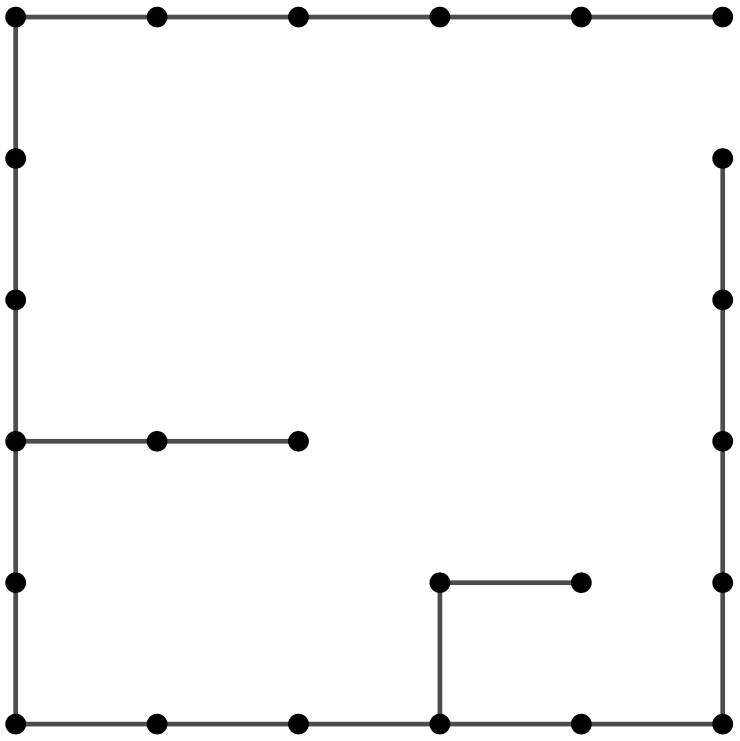} 
    \end{minipage}\hfill
    \begin{minipage}{0.45\textwidth}
        \centering
        \includegraphics[width=0.7\textwidth]{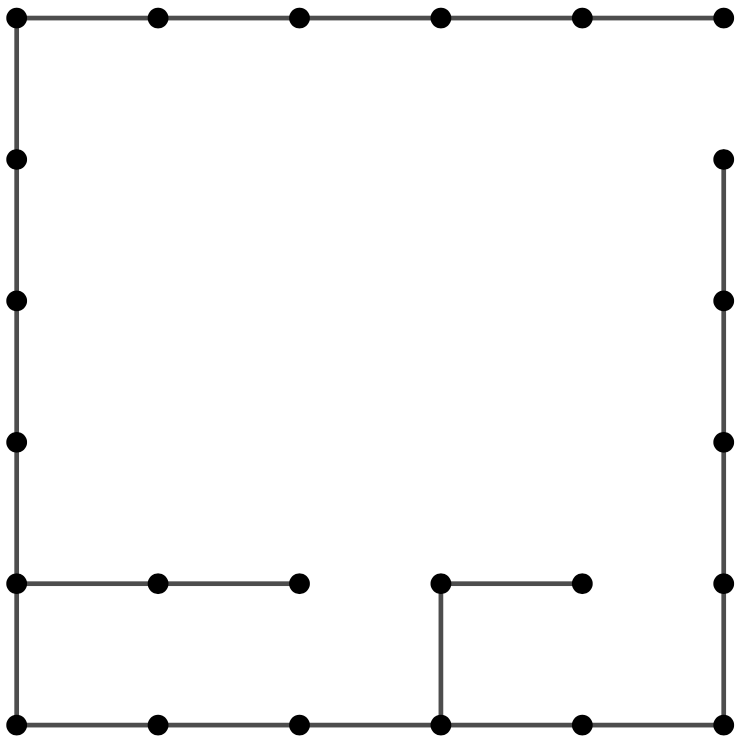} 
    \end{minipage}
    \caption{Dendrites $T^1$ (left) and $T^2$ (right) for $n=6$.}
    \label{fig2}
\end{figure}

For each $i\in\{1,2\}$ define combinatorial graphs $G^i=(V^i,E^i)$ by 
\[V^i:=\{\psi_j^i({\bf 0}):j\in \mathcal{A}_n\} \qquad\text{and}\qquad E^i:=\{\{p,q\}:p,q\in V^i,[p,q]\subset T^i,|q-p|=1/n\}.\] 
Next, we define a sequence of sets $(T_m)_{m\geq 0}$ by $T_0=\{{\bf 0}\}$, and for any integer $m\geq 0$, 
\[ T_{m+1}=T_m\cup\left(\bigcup_{w\in \mathcal{A}_n^m}\phi_w^{\eta}(T^{\eta(w)})\right).\] 
Define also for $m\geq 0$ the combinatorial graph $G_m=(V_m,E_m)$ via 
\[ V_m=\{\phi_w^{\eta}({\bf 0}):w\in \mathcal{A}_n^m\} \qquad\text{and}\qquad E_m=\{\{p,q\}:p,q\in V_m,[p,q]\subset T_m,|p-q|=n^{-m}\}.\]

\begin{lemma}\label{lem:leaf}
The sequence of sets $(T_m)_{m=0}^\infty$ is a nested sequence of dendrites contained in $K^{\eta}$. Moreover, for every $m\in \N$, a vertex $u\in V_m$ is a leaf of $T_m$ if and only if $u$ has no edges emanating to the right or upward in $T_m$.
\end{lemma}

\begin{proof}
By Lemma \ref{lem:fact}, for each $w\in \mathcal{A}_n^*$ we have that $\phi_w^{\eta}(\partial\sq)\subset K^{\eta}$. Therefore, 
\[ T_m\subset\bigcup_{w\in \mathcal{A}_n^m}\phi_w^{\eta}(\partial\sq) \subset K^{\eta}.\]

Next, we show that each $T_m$ is a dendrite by induction on $m$. The set $T_0=\{{\bf 0}\}$ is clearly a dendrite. Now assume that $T_m$ is a dendrite for some $m\geq 0$. We make three observations. First, by Lemma \ref{lem:T^i}, $\phi_w^{\eta}(T^{\eta(w)})$ is a dendrite for any $w\in \mathcal{A}_n^{m+1}$. Second, by Lemma \ref{lem:OSC} and by Lemma \ref{lem:T^i}, if $v,w\in \mathcal{A}_n^{m+1}$ are distinct, then
\[ \phi_w^{\eta}(T^{\eta(w)})\cap\phi_v^{\eta}(T^{\eta(v)}) \subset \phi_w^\eta([0,1)^2)\cap\phi_v^\eta([0,1)^2) = \emptyset.\]
Third, for any $w\in \mathcal{A}_n^{m+1}$ we have that $\phi_w^{\eta}(T^{\eta(w)}) \cap T_{m+1}$ is either $\phi_w^{\eta}(\{0\}\times[0,1-\tfrac1{n}])$, or $\phi_w^{\eta}([0,1-\tfrac1{n}]\times\{0\})$, or the union of the latter two sets. In either case, $\phi_w^{\eta}(T^{\eta(w)}) \cap T_{m+1}$ is an arc. Write now $T_{m+1} = T_m\cup\bigcup_{k=1}^N X_k$ where $X_1,\dots,X_N$ are the components of the closure of $T_{m+1}\setminus T_m$. Note the sets $X_k$ are pairwise disjoint, by the first and second observation each $X_k$ is a dendrite, and by the first and third observation, $X_i\cap T_m$ is a point. Applying Lemma \ref{lem:uniondendrites}, we conclude that $T_{m+1}$ is a dendrite.

Finally, we prove the claim about the leaves of $T_m$ by induction on $m$. For $m=1$, we note that the leaves of $T^1$ are exactly the points $(1-\frac1n,1-\frac1n)$, $(1-\frac1n,1-\frac2n)$, $(\frac2n,\frac2n)$, $(1-\frac2n,\frac1n)$ for which it is easy to check the claim. Similarly for $T^2$.

Suppose now that the claim holds for some $m\in \N$ and let $u\in V_{m+1}$ be a vertex. Then there is a unique word $w\in \mathcal{A}_n^m$ such that $u\in \phi_w^\eta(T^{\eta(w)})$, and $u$ is a leaf in $T_{m+1}$ if and only if $(\phi_w^{\eta})^{-1}(u)$ is a leaf in $T^{\eta(w)}$, which is a leaf if and only if it has no edges in $T^{\eta(w)}$ emanating to the right or upward. This holds if and only if $u$ has no edges emanating upward or to the right in $T_{m+1}$.
\end{proof}

Note that for each integer $m\geq 0$, 
\[T_m=\text{Im}(G_m) = V_m\cup(\bigcup\limits_{\{p,q\}\in E_m} (p,q)),\]
so when we refer to a ``vertex'' of $T_m$ we mean a point $p\in V_m$, and when we refer to an ``edge $(p,q)$ of $T_m$'' we mean the open segment $(p,q)$ in $\R^2$ associated with the edge $\{p,q\}\in E_m$. Set 
\[e_b = (0,1)\times\{0\}\qquad\text{and}\qquad e_l = \{0\}\times(0,1), \]
the bottom and left, respectively, open edges of $(0,1)^2$. Note that if $e\subset T_m$ is an edge in $T_m$, then $e$ is parallel to either the $x$-axis or to the $y$-axis, as this is true for $T^1$ and $T^1$ and the maps $\psi_j^i$ are all rotation-free and reflection-free similarities.

\begin{lemma}\label{lem:edge}
For each $m\in\N$, if $(p,q)$ is an edge of $T_m$, then there exists $w\in \mathcal{A}_n^m$ such that either $(p,q) = \phi_w^{\eta}(e_b)$ or $(p,q)=\phi_w^{\eta}(e_l)$.
\end{lemma}

\begin{proof}
We proceed by induction on $m$. For $m=1$, the result follows by simple inspection, as either $T_1=T^1$ or $T_1=T^2$. Now assume the result holds for some $m\in\N$. Let $(p,q)\subset T_{m+1}$ be an edge. Then there exist words $u,v\in \mathcal{A}_n^{m+1}$ satisfying $p=\phi_u^\eta({\bf 0})$ and $q=\phi_v^\eta({\bf 0})$. If there is a word $w\in \mathcal{A}_n^m$ such that $(p,q)\subset \phi_w^\eta(T^{\eta(w)})$, then there is some edge $e\subset T^{\eta(w)}$ with $(p,q)=\phi_w^\eta(e)$, and the result follows from the $m=1$ case and from the definition of $\phi_w^\eta$.

If there is no such word $w$, then we have that $(p,q)\subset T_m$. Since the only nondegenerate line segments contained in $T_m$ are contained in its edges, there is some edge $e\subset T_m$ such that $(p,q)\subset e$, and since $(p,q)$ is not a subset of $\phi_v^\eta(T^{\eta(v)})$ for any $v\in \mathcal{A}_n^m$, we have that $[p,q]$ is not a subset of $\phi_v^\eta([0,1)^2)$ for any $v\in\mathcal{A}_n^m$. Then by the induction hypothesis there is a word $w\in \mathcal{A}^m_n$ satisfying either $e=\phi_w^\eta(e_b)$ or $e=\phi_w^\eta(e_l)$. Note that the length of $(p,q)$ is $n^{-m-1}$, that $[p,q]\subset \phi_w^\eta(\sq)$, and that $(p,q)\subset \phi_w^\eta([0,1)^2)$. Then either $p$ or $q$ is $\phi_w^\eta((0,1))$ or $\phi_w^\eta((1,0))$ since the edge $e$ is equal to either $\phi_w^\eta(e_b)$ or to $\phi_w^\eta(e_l)$. Then there is a $j\in\mathcal{A}_n$ such that the other of $p$ and $q$ is equal to $\phi_{wj}^\eta({\bf 0})$, therefore $\phi_{wj}^\eta(e_b)=(p,q)$ or $\phi_{wj}^\eta(e_l)=(p,q)$. \end{proof}

In the next lemma, which is the crux of the proof of Proposition \ref{thm:holder}, we construct intermediate parameterizations $f_m:[0,1]\to T_m$.

\begin{lemma}\label{lem:intermedparam}
There exist a sequence of piecewise linear continuous surjections $(f_m:[0,1]\to T_m)_{m\geq0}$, a sequence $(\mathscr{N}_m)_{m\geq0}$ of families of nondegenerate closed intervals contained in $[0,1]$, a sequence $(\mathscr{E}_m)_{m\geq0}$ of families of open intervals in $[0,1]$, another sequence $(\mathscr{F}_m)_{m\geq0}$ of families of nondegenerate closed intervals in $[0,1]$, and a sequence of functions $(\w_m:\mathscr{N}_m\to \mathcal{A}_n^m)_{m\geq 0}$ satisfying the following properties.
\begin{enumerate}
\item[(P1)] For each $m\geq 0$, $\w_m$ is a bijection between $\mathscr{N}_m$ and $\mathcal{A}^m_n$.
\item[(P2)] For each $m\geq 0$, the families $\mathscr{E}_m,\mathscr{N}_m,\mathscr{F}_m$ are pairwise disjoint and the elements of $\mathscr{E}_m\cup\mathscr{N}_m\cup\mathscr{F}_m$ are pairwise disjoint. Moreover, $[0,1]=\bigcup(\mathscr{E}_m\cup\mathscr{N}_m\cup\mathscr{F}_m)$.
\item[(P3)] For each $m\geq 0$ and $J\in\mathscr{E}_m$, there exists $w\in \mathcal{A}_n^m$ such that $f_m|J$ is a linear bijection onto either $\phi^{\eta}_w(e_b)$ or $\phi^{\eta}_w(e_l)$, and there exists a unique interval $J'\in\mathscr{E}_m\setminus \{J\}$ such that $f_m|J'=(f_m|J) \circ \zeta_{J'}$, where $\zeta_{J'} :J' \to J$ is the unique linear orientation reversing map. Furthermore, $f_{m+1}(J)\subset \phi^\eta_w([0,1)^2)$. Conversely, if $(p,q)$ is an edge of $T_m$, then there exist exactly two intervals $J,J' \in \mathscr{E}_m$ such that $f_m(J)=f_m(J') = (p,q)$.
\item[(P4)] For each $m\geq 0$ and $I\in\mathscr{N}_m$, $f_m|I$ is constant equal to $\phi^\eta_{\w_m(I)}({\bf 0})$. Conversely,
\newline $f_m(\bigcup\mathscr{N}_m)=V_m$. Moreover, $f_{m+1}(I)\subset \phi^\eta_{\w_m(I)}([0,1)^2)$.
\item[(P5)] For each $m\geq 0$ and $I\in\mathscr{F}_m$, $f_m|I$ is constant and $f_m(I)\in V_m$. Furthermore, $f_{m+1}|I = f_m|I$.
\item[(P6)] For each $m\geq 0$ and $J\in\mathscr{E}_{m+1}\cup\mathscr{N}_{m+1}$, there exists $I\in\mathscr{E}_m\cup\mathscr{N}_m$ such that $J\subset I$. Moreover, if $J\in\mathscr{N}_{m+1}$ and $I\in\mathscr{N}_m$ with $J\subset I$, then $\w_m(I) =\w_{m+1}(J)(m)$. 
\item[(P7)] For each $m\geq 0$ and $J\in\mathscr{E}_m$, there exists $I\in\mathscr{N}_{m+1}$ such that $I\subset J$ and $f_{m+1}(I)$ is a leaf in $T_{m+1}$.
\end{enumerate}
\end{lemma}

\begin{proof}
We start by defining two pairs of preliminaries maps.

First, fix for each $i\in\{1,2\}$ a continuous surjection $\tau_i:[0,1]\to T^i$ with the following properties.
\begin{enumerate}
\item The map $\tau_i$ is a 2--to--1 tour of the edges of $T^i$ and is linear along edges with $\tau_i(0)=\tau_i(1)={\bf 0}$.
\item For every vertex $v\in V^i$ other than ${\bf 0}$, the preimage $\tau_i^{-1}(\{v\})$ has a number of components equal to the valence of the vertex $v$ and each component is a nondegenerate closed interval.
\item The preimage $\tau_i^{-1}(\{{\bf 0}\})$ is made up of three disjoint nondegenerate closed intervals.
\item There exists $t_i$ in the component of $\tau_i^{-1}(\{{\bf 0}\})$ which does not contain $0$ or $1$ such that $\tau_i^{-1}(\{(0,1-1/n)\})$ lies to the left of $t_i$ and $\tau_i^{-1}(\{(1-1/n,0)\})$ lies to the right of $t_i$.
\end{enumerate} 

Next, for each $i\in\{1,2\}$, define the combinatorial tree $\tilde{G}^{i}=(\tilde{V}^{i},\tilde{E}^{i})$ via 
\[ \tilde{V}^{i}:=V^i\cup\{(1,0),(0,1)\}, \quad \tilde{E}^{i}=E^i\cup\{\{(1-1/n,0),(1,0)\},\{(0,1-1/n),(0,1)\}\},\] 
and let $\tilde{T}^{i}:=\text{Im}(\tilde{G}^{i})$. Fix for each $i\in\{1,2\}$ a continuous surjection $\tilde{\tau}_{i}:[0,1]\to \tilde{T}^{i}$ with the following properties. 
\begin{enumerate}
\item The map $\tilde{\tau}_{i}$ is a 2--to--1 tour of the edges of $\tilde{T}^{i}$ and is linear along edges with $\tilde{\tau}_{i}(0)=\tilde{\tau}_{i}(1)={\bf 0}$.
\item For every vertex $v\in \tilde{V}^{i}$ other than ${\bf 0}$, $(0,1)$, and $(1,0)$, the preimage $\tilde{\tau}_{i}^{-1}(\{v\})$ has a number of components equal to the valence of the vertex $v$ and each component is a nondegenerate closed interval.
\item The preimage $\tilde{\tau}_{i}^{-1}(\{{\bf 0}\})$ is made up of three disjoint nondegenerate closed intervals.
\item The preimages $\tilde{\tau}_{i}^{-1}(\{(1,0)\})$ and $\tilde{\tau}_{i}^{-1}(\{(0,1)\})$ are both singletons, denoted by $s_{i,U}$ and $s_{i,L}$, respectively.
\item There exists $\tilde{t}_{i}$ in the component of $\tilde{\tau}_{i}^{-1}(\{{\bf 0}\})$ which does not contain $0$ or $1$ such that $s_{i,U}$ lies to the left of $\tilde{t}_{i}$ and $s_{i,L}$ lies to the right of $\tilde{t}_{i}$.
\item There are disjoint nondegenerate closed intervals $I_{i,1},I_{i,2},I_{i,3},I_{i,4}\subset [0,1]$, denoted by $I_{i,j}=[a_{i,j},b_{i,j}]$, $j=1,\dots,4$ equal to the preimages of the four leaves of $T^i$, such that
\[0<a_{i,1}<s_{i,U}<a_{i,2}<\tilde{t}_{i}<a_{i,3}<s_{i,L}<a_{i,4}<b_{i,4}<1.\]
\end{enumerate}

Let $T$ be one of the trees $T^i$ or $\tilde{T}^i$ above, let $V$ and $E$ be the corresponding vertex and edge sets, respectively, and let $\tau:[0,1]\to T$ be the map defined above corresponding to $T$.
\begin{itemize}
\item Define $\mathcal{E}(\tau)$ to be the set of components of preimages of open edges in $E$.
\item If $\tau\in\{\tau_1,\tau_2\}$, then for each $v\in V$ choose a nondegenerate component $I_v$ of the preimage $\tau^{-1}(\{v\})$ and let $\mathcal{N}(\tau)=\{I_v:v\in V\}$.
\item If $\tau\in\{\tilde{\tau}_1,\tilde{\tau}_2\}$, then for a vertex $v\in V$ that is not a leaf, choose a nondegenerate component $I_v$ of the preimage $\tau^{-1}(\{v\})$, and for each vertex $v\in V$ that is a leaf other than $(1,0)$ or $(0,1)$, let $I_v=I_{i,j}$ where $j\in\{1,2,3,4\}$ is chosen to satisfy $\tau(I_{i,j})=v$. We then define $\mathcal{N}(\tau)=\{I_v:v\in V\setminus\{(1,0),(0,1)\}\}$.

\item Define $\mathcal{F}(\tau)$ to be the set of nondegenerate components of preimages under $\tau$ of vertices in $T$ which are not already in $\mathcal{N}(\tau)$.
\end{itemize} 

The construction of maps $f_m$ and families $\mathscr{E}_m,\mathscr{F}_m,\mathscr{N}_m$ is done in an inductive fashion. 

Set $f_0:[0,1]\to T_0$ to be the constant map $f_0\equiv {\bf 0}$, set $\mathscr{N}_0=\{[0,1]\}$, $\mathscr{E}_0=\emptyset$, $\mathscr{F}_0=\emptyset$, and set $\w_0([0,1])=\eps$. Properties (P1)--(P7) are all either clear or vacuous for $m=0$.

Assume now that for some $m\geq 0$ we have constructed the map $f_m$, families $\mathscr{E}_m,\mathscr{F}_m$, and $\mathscr{N}_m$, and bijection $\w_m:\mathscr{N}_m\to \mathcal{A}_n^m$ satisfying (P1)--(P7). For elements $I\in\mathscr{E}_m\cup\mathscr{N}_m$, we will define the collections $\mathscr{E}_{m+1}(I),\mathscr{N}_{m+1}(I)$, and $\mathscr{F}_{m+1}(I)$. We then set 
\begin{align*}
\mathscr{E}_{m+1}&=\bigcup_{I\in\mathscr{E}_m\cup\mathscr{N}_m}\mathscr{E}_{m+1}(I),\\
\mathscr{N}_{m+1}&=\bigcup_{I\in\mathscr{E}_m\cup\mathscr{N}_m}\mathscr{N}_{m+1}(I),\\
\mathscr{F}_{m+1}&=\mathscr{F}_m \cup \bigcup_{I\in\mathscr{E}_m\cup\mathscr{N}_m}\mathscr{F}_{m+1}(I).
\end{align*} 

\medskip

\emph{Intervals in $\mathscr{F}_m$.} If $I \in \mathscr{F}_m$, then set $f_{m+1}|I = f_m$.

\medskip

\emph{Intervals in $\mathscr{E}_m$.} Let $I\in \mathscr{E}_m$ and write $I=(a,b)$. By (P3) there exists $I'=(a',b')\in\mathscr{E}_m \setminus \{I\}$ such that $f_m|I'=(f_m|I) \circ \zeta_{I'}$ where $\zeta_{I'} :I' \to I$ is the unique linear orientation reversing map and there exists a word $v\in\mathcal{A}_n^m$ such that $f_m(I)=\phi_v^\eta(e_b)$ or $f_m(I)=\phi_v^\eta(e_l)$. If $f_m(I) = \phi^{\eta}_v(e_l)$, then, without loss of generality, we may assume that the $y$-coordinate of $f_m$ is increasing inside $I$ and decreasing inside $I'$. Let $\tilde{\theta}_I:I\to (0,s_{\eta(v),U})$ be linear, bijective, and increasing, and let $\tilde{\xi}_{I'}:I'\to (s_{\eta(v),U},\tilde{t}_{\eta(v)})$ be linear, bijective, and decreasing. We define 
\[ f_{m+1}|I=\phi_v^{\eta}\circ\tilde{\tau}_{\eta(v)}\circ\tilde{\theta}_I \qquad \text{and}\qquad f_{m+1}|I'=\phi_v^{\eta}\circ\tilde{\tau}_{\eta(v)}\circ\tilde{\xi}_{I'}.\] 
Define $\mathscr{E}_{m+1}(I)=\{\tilde{\theta}_I^{-1}(J):J\in\mathcal{E}(\tilde{\tau}_{\eta(v)})\}$ and $\mathscr{E}_{m+1}(I')=\{\tilde{\xi}_{I'}^{-1}(J) : J\in\mathcal{E}(\tilde{\tau}_{\eta(v)})\}$. We define $\mathscr{N}_{m+1}(I)$, $\mathscr{N}_{m+1}(I')$, $\mathscr{F}_{m+1}(I)$, and $\mathscr{F}_{m+1}(I')$ in a similar manner. If $f_m(I) = \phi^{\eta}_v(e_b)$, then we proceed in a similar manner, interchanging $s_{\eta(v),U}$ with $s_{\eta(v),L}$ above. 

If $J\in \mathscr{N}_{m+1}(I)$, then there is a character $j\in\mathcal{A}_n$ such that $f_{m+1}(J)=\phi_{vj}^\eta({\bf 0})$, and we define $\w_{m+1}(I)=vj$. Note that if $J,Q\in\mathscr{N}_{m+1}(I)\cup\mathscr{N}_{m+1}(I')$ are distinct, then $\w_{m+1}(J)\neq \w_{m+1}(Q)$ as $\tilde{\theta}_I(J)\neq\tilde{\theta}_I(Q)$ are distinct intervals in, for example, $\mathcal{N}(\tilde{\tau}_{\eta(v)})$ ($\tilde{\theta}_I$ may be $\tilde{\xi}_{I'}$ instead).

\medskip

\emph{Intervals in $\mathscr{N}_m$.} Let $I\in\mathscr{N}_m$ and write $I=[a,b]\subset [0,1]$. We consider four cases. 

\emph{Case 1.} If the vertex $f_m(I)$ in $T_m$ has an edge in $T_m$ emanating to the right and no edge emanating upward, then define $\zeta_I:I\to [0,t_{\eta(\w_m(I))}]$ to be the unique linear increasing bijection, and let 
\[ f_{m+1}|I=\phi^\eta_{\w_m(I)}\circ\tau_{\eta(\w_m(I))}\circ\zeta_I.\] 
Define $\mathscr{F}_{m+1}(I),\mathscr{E}_{m+1}(I),\mathscr{N}_{m+1}(I)$ in a manner similar to those for $I\in\mathscr{E}_m$ above.

\emph{Case 2.} If the vertex $f_m(I)$ in $V_m$ has an edge in $T_m$ emanating upward and no edge emanating to the right, then define $\xi_I:I\to[t_{\eta(\w(I))},1]$ to be the unique increasing linear bijection, and let 
\[ f_{m+1}|I=\phi^\eta_{\w_m(I)}\circ\tau_{\eta(\w_m(I))}\circ\xi_I.\]
Define $\mathscr{F}_{m+1}(I),\mathscr{E}_{m+1}(I),\mathscr{N}_{m+1}(I)$ in a manner similar to those for $I\in\mathscr{E}_m$ above.

\emph{Case 3.} If the vertex $f_m(I)$ is a leaf in $T_m$, then let $\sigma_I:I\to[0,1]$ be the unique increasing linear bijection, and let
\[ f_{m+1}|I=\phi^\eta_{\w_m(I)}\circ\tau_{\eta(\w_m(I))}\circ\sigma_I.\] 
Define $\mathscr{F}_{m+1}(I),\mathscr{E}_{m+1}(I),\mathscr{N}_{m+1}(I)$ in a manner similar to those for $I\in\mathscr{E}_m$ above.

\emph{Case 4.} If the vertex $f_m(I)$ has an edge in $T_m$ emanating upward and another edge emanating to the right, then simply set $f_{m+1}|I=f_m|I$ and let $\mathscr{N}_{m+1}(I)=\{I\}$, $\mathscr{E}_{m+1}(I)=\emptyset$, and $\mathscr{F}_{m+1}(I)=\emptyset$. 

By definition of $\tau_{\eta(\w_m(I))}$, for any $Q\in\mathscr{N}_{m+1}(I)$, there is a character $j\in \mathcal{A}_n$ satisfying $f_{m+1}(Q)=\{\phi^\eta_{\w_m(I)j}({\bf 0})\}$, and we set $\w_{m+1}(Q)=\w_m(I)j$.

\medskip

We now prove (P1)--(P7) for $f_{m+1}$, $\mathscr{E}_{m+1}$, $\mathscr{N}_{m+1}$, $\mathscr{F}_{m+1}$, and $\w_{m+1}$.

We start with (P2). Since (P2) holds for $\mathscr{E}_m,\mathscr{N}_m$, and $\mathscr{F}_m$, and since intervals in $\mathscr{F}_{m+1}$ are not partitioned, it suffices to show that for every $I\in\mathscr{E}_m\cup\mathscr{N}_m$, the families $\mathscr{E}_{m+1}(I),\mathscr{N}_{m+1}(I)$, and $\mathscr{F}_{m+1}(I)$ are pairwise disjoint with disjoint elements covering $I$. If $I\in\mathscr{E}_m$, then this holds by definitions of $\mathcal{E}(\tilde{\tau}_i)$, $\mathcal{N}(\tilde{\tau}_i)$, and $\mathcal{F}(\tilde{\tau}_i)$. If $I\in\mathscr{N}_m$, then this follows from the definitions of $\mathcal{E}(\tau_i)$, $\mathcal{N}(\tau_i)$, and $\mathcal{F}(\tau_i)$.

The first part of (P3) follows from the analogous properties defining $\tilde{\tau}_i$. The converse part of (P3) follows from (P3) and (P6) at $m$ and by the definitions of $\tilde{\tau}_i$ and $\tau_i$.

The first part of (P4) follows from design of $\mathcal{N}(\tau_i)$ and $\mathcal{N}(\tilde{\tau}_i)$ along with (P1) and (P4) at $m$. The converse part follows from (P1), (P3), (P4), and (P6) at $m$ along with design of $\tau_i$ and $\tilde{\tau}_i$.

Property (P5) follows immediately from the definitions of $\mathcal{F}(\tau_i)$ and $\mathcal{F}(\tilde{\tau}_i)$.

The first part of (P6) follows from (P2) at $m$, as for every $I\in\mathscr{E}_m\cup \mathscr{N}_m$ and every $J\in\mathscr{E}_{m+1}(I)\cup\mathscr{N}_{m+1}(I)$, $J\subset I$, while every $J\in\mathscr{E}_{m+1}\cup\mathscr{N}_{m+1}$ is in $\mathscr{E}_{m+1}(I)\cup\mathscr{N}_{m+1}(I)$ for some $I\in\mathscr{E}_m\cup\mathscr{N}_m$. The converse part of (P6) is immediate from the definition of $\w_{m+1}(J)$ for $J\in \mathscr{N}_{m+1}(I)$ with $I\in\mathscr{N}_m$.

Property (P7) follows from properties (4), (5), and (6) in the definition of $\tilde{\tau}_i$.

Finally, we verify (P1). We first show that $\w_{m+1}$ is injective. Let $I,J\in\mathscr{N}_{m+1}$ be distinct intervals. Then by (P6), let $\tilde{I},\tilde{J}\in\mathscr{E}_m\cup\mathscr{N}_m$ such that $I\in\mathscr{N}_{m+1}(\tilde{I}),J\in\mathscr{N}_{m+1}(\tilde{J})$. If $\tilde{I}=\tilde{J}$ are the same element of $\mathscr{E}_m$, then $\w_{m+1}(I)\neq \w_{m+1}(J)$ by the definitions of $\mathcal{N}(\tilde{\tau}_i)$. If $\tilde{I}=\tilde{J}$ are the same element of $\mathscr{N}_m$, then by (P6) there are characters $i,j\in\mathcal{A}_n$ with $\w_{m+1}(I)=\w_m(\tilde{I})i$, $\w_{m+1}(J)=\w_m(\tilde{I})j$. Then by the definitions of $\mathcal{N}(\tau_i)$, we have that $i\neq j$. If $\tilde{I}\neq \tilde{J}$ are distinct intervals, both in $\mathscr{N}_m$, then by (P6) at $m+1$ and (P1) at $m$, we have that $\w_{m+1}(I)(m)\neq \w_{m+1}(J)(m)$, hence $\w_{m+1}(I)\neq \w_{m+1}(J)$. If $\tilde{I}\neq \tilde{J}$ with $\tilde{J}\in \mathscr{E}_m$ and $\tilde{I}\in\mathscr{N}_m$, then by definitions of $\tau_i$ and $\tilde{\tau}_i$, $f_{m+1}(\tilde{I})\cap f_{m+1}(\tilde{J})=\emptyset$, thus by (P3) and (P4), $\w_{m+1}(I)\neq \w_{m+1}(J)$. Similarly $\w_{m+1}(I)\neq \w_{m+1}(J)$ if $\tilde{I}\neq\tilde{J}$ with $\tilde{J}\in \mathscr{E}_m$ and $\tilde{I}\in\mathscr{E}_m\setminus\{\tilde{J},\tilde{J}'\}$. If $\tilde{J}\in\mathscr{E}_m$ and $\tilde{I}=\tilde{J}'$, then $\w_{m+1}(I)\neq \w_{m+1}(J)$ by definition of $\mathcal{N}(\tilde{\tau}_i)$.

To finish the proof, we show that $\w_{m+1}:\mathscr{N}_{m+1}\to \mathcal{A}_n^{m+1}$ is surjective. Let $u\in\mathcal{A}_n^{m+1}$ be a word. By (P1) at $m$, there exists a unique interval $I\in\mathscr{N}_m$ such that $\w_m(I)=u(m)$. If $f_m(I)$ is a leaf in $T_m$, then by definition of $\mathcal{N}(\tau_i)$, for each $j\in\mathcal{A}_n$ there is an interval $J\in\mathscr{N}_{m+1}(I)$ satisfying $\w_{m+1}(J)=\w_{m+1}(I)j$, in particular one of them has $\w_{m+1}(J)=u$. If $f_m(I)$ is not a leaf in $T_m$, then the result follows by the definitions of $\mathcal{N}(\tilde{\tau}_i),\mathcal{N}(\tau_i)$ as well as (P3) at $m$.
\end{proof}

Throughout the remainder of this section, when we write any of (P1)--(P7) we mean the appropriate property (P1)--(P7) from Lemma \ref{lem:intermedparam}.

\begin{rem}\label{rem:w}
 It is possible that $\mathscr{N}_{m+1}\cap\mathscr{N}_m\neq \emptyset$ for some $m \in \N$. Define a map $\w:\bigcup_{m\geq 0}\mathscr{N}_m\to \mathcal{A}_n^*$ as follows. Set $\w([0,1])=\varepsilon$. If $m\in\N$ and $I\in\mathscr{N}_m\setminus\mathscr{N}_{m-1}$, set $\w(I)=\w_m(I)$. If $m\in\N$ and $I\in\mathscr{N}_m\cap\mathscr{N}_{m-1}$, set $\w(I)=\w_{m-k}(I)$, where $k\in\N$ is the greatest integer such that $I\in\mathscr{N}_{m-k}$. Note that $\w$ is injective and $V_m=\{\phi_{\w(I)}^\eta({\bf 0}):I\in\mathscr{N}_m\}$ for every $m\in\N$.
\end{rem}

\begin{corollary}\label{cor:limit}
The maps $(f_m)_{m\geq 0}$ of Lemma \ref{lem:intermedparam} converge uniformly to a continuous surjection $f: [0,1] \to K^{\eta}$.
\end{corollary}

\begin{proof}
Since $T_m\subset K^\eta$, the Hausdorff distance
\[ \dist_H(T_m,K^\eta) \leq \dist_H(V_m,K^{\eta}) \leq \max_{w\in \mathcal{A}^m_n}\diam{K_w^{\eta}} \leq n^{-m}\sqrt{2}. \]
Therefore, $\overline{\bigcup_{m\in\N}T_m}=K^\eta$, as $K^\eta$ is compact. 

Next we claim that for every $m\in\N$, 
\begin{equation}\label{eqn:inftydist}
\|f_m-f_{m+1}\|_\infty\leq n^{-m}\sqrt{2}.
\end{equation}
Fix $m\in\N$ and $x\in [0,1]$. By (P1), there exists $I\in\mathscr{N}_m\cup\mathscr{E}_m\cup\mathscr{F}_m$ such that $x\in I$. If $I\in\mathscr{F}_m$, then by (P5) we have that $f_m(x)=f_{m+1}(x)$. If $I\in\mathscr{N}_m$, then by (P4) we have that $f_m(x),f_{m+1}(x)\in\phi^\eta_{\w_m(I)}(\sq)$. Finally, if $I\in\mathscr{E}_m$, then by (P3) we have that there is some word $v\in \mathcal{A}_n^m$ such that $f_m(x),f_{m+1}(x)\in\phi^\eta_v(\sq)$. In any case, we have the claimed estimate. 

Hence, the functions $f_m$ converge uniformly to a continuous map $f:[0,1]\to \sq$. Furthermore, because the dendrites $T_m\subset K^\eta$ are nested and $f_m:[0,1]\to T_m$ are surjective, we have that the uniform limit $f$ is a continuous parameterization of $K^\eta$, as desired.
\end{proof}

\begin{lemma}\label{lem:selfish}
For each $m\in\N$ and every pair of distinct intervals $I,J\in\mathscr{E}_m\cup\mathscr{N}_m$, we have that $\mathcal{H}^{\alpha_n}(f(I)\cap f(J))=0$.
\end{lemma}

Before proving the lemma we recall the well-known definition of porosity. Given a metric space $X$, we say that a set $Y\subset X$ is \emph{porous} in $X$ if there exists $c>1$ such that for any $r<\diam{X}$ and any $x\in X$, there exists $y \in B(x,r)$ such that $B(y,c^{-1}r) \subset B(x,r)\setminus Y$.

\begin{proof}[Proof of Lemma \ref{lem:selfish}]
We claim that $f(I)\cap f(J)$ is porous in $K^\eta$. Assuming the claim, since $K^\eta$ is Ahlfors $\alpha_n$-regular by Proposition \ref{prop:reg}, it follows from \cite[Lemma 3.12]{BHR} that the Hausdorff dimension (in fact, the Assouad dimension) of $f(I)\cap f(J)$ is less than $\a_n$ which gives the lemma.

To prove porosity, note first that by (P3), (P4), and (P6) there exists a word $v\in \mathcal{A}_n^m$ such that $f(I)\cap f(J)\subset K^\eta_v$. Fix $x\in K^\eta$ and $r\in (0,n^{-m}]$. There exists an integer $k_0\in \N$ such that $n^{-k_0}\leq r< n^{-k_0+1}$, thus there exists $w\in \mathcal{A}_n^{k_0+1}$ with $K^\eta_w\subset \overline{B}(x,r)\cap K^\eta$.

We now prove that there are characters $i,j\in \mathcal{A}_n$ such that $f(I)\cap f(J)\cap K^\eta_{wji}=\emptyset$. This falls to a case study.

\emph{Case 1.} Suppose that $I,J\in\mathscr{N}_m$. Then $f(I)\cap\phi^\eta_{\w_m(J)}((0,1)^2)=\emptyset$ and $f(J)\cap\phi^\eta_{\w_m(I)}((0,1)^2)=\emptyset$ by Lemma \ref{lem:OSC}, (P4), and the fact that $\w_m(I), \w_m(J)$ are distinct words of length $m$. Moreover, at least one of $\phi^\eta_w((0,1)^2)\cap \phi^\eta_{\w_m(I)}(\sq)$ and $\phi^\eta_w((0,1)^2)\cap \phi^\eta_{\w_m(J)}(\sq)$ is empty. Since there are $i,j\in \mathcal{A}_n$ for which $K^\eta_{wji}\subset \phi^\eta_w((0,1)^2)$, we have $f(I)\cap f(J)\cap K^\eta_{wji}=\emptyset$.

\emph{Case 2.} Suppose that at least one of $I$, $J$ is in $\mathscr{E}_m$. By (P3), (P4), and (P6), there exist words $u,v\in \mathcal{A}_n^m$ such that $f(J)\subset K^\eta_u$ and $f(I)\subset K^\eta_v$. 

If $u\neq v$, then we proceed as in Case 1 with $u$ playing the role of $\w_m(J)$ and $v$ playing the role of $\w_m(I)$. If $f(I)$ or $f(J)$ is a singleton, then the existence of such characters $i,j\in \mathcal{A}_n$ is clear. 

Therefore, we may assume for the rest of Case 2 that $u=v$ and that neither of $f(I),f(J)$ is a singleton. If $f_{k_0+1}(J)\cap K^\eta_w=\emptyset$, then $f(J)\cap \phi^\eta_w((0,1)^2)=\emptyset$ by Lemma \ref{lem:OSC} and by (P3), (P4), and (P5), hence the existence of such characters $i,j\in \mathcal{A}_n$ is clear. Similarly if $f_{k_0+1}(I)\cap K^\eta_w=\emptyset$. Now assume that $f_{k_0+1}(I)\cap K^\eta_w\neq\emptyset$ and $f_{k_0+1}(J)\cap K^\eta_w\neq \emptyset$. Then there are distinct leaves $p_1,p_2\in T_{k_0+2}$ with $p_1,p_2\in K^\eta_w$ and distinct intervals $Q_1,Q_2\in\mathscr{N}_{k_0+2}$ such that $Q_1\subset J$ and $Q_2\subset I$ satisfy $f_{k_0+2}|Q_1\equiv p_1$ and $f_{k_0+2}|Q_2\equiv p_2$. Then $f([0,1]\setminus Q_1)\cap \phi^\eta_{\w_{k_0+2}(Q_1)}((0,1)^2)=\emptyset$ and $\w(Q_1)(k_0+1)=w$, therefore there is a pair of characters $i,j\in \mathcal{A}_n$ such that $wj=\w_{k_0+2}(Q_1)$ and $K^\eta_{wji}\subset \phi^\eta_{wj}((0,1)^2)$. Since $I,J$ are distinct, and thus disjoint, we have that $f(I)\cap \phi^\eta_{wj}((0,1)^2)=\emptyset$, concluding Case 2.

Additionally, there exists a point $y\in K^\eta_{wji}$ with $\overline{B}(y,n^{-k_0-5})\cap K^\eta\subset K^\eta_{wji}$. Thus, for every $r\in (0,n^{-m}]$ and each $x\in K^\eta$, there exists $y\in K^\eta$ such that 
\[\overline{B}(y,n^{-5}r)\cap K^\eta\subset K^\eta\cap (\overline{B}(x,r)\setminus(f(I)\cap f(J))). \qedhere\]
\end{proof}

For the rest of this section, $f$ is the uniform limit of the maps $f_m$ from Corollary \ref{cor:limit}, $\nu_n$ is the probability measure from \textsection\ref{sec:words}, and $\pi: \mathcal{A}_n^{\N} \to K^{\eta}$ is the map from \eqref{eq:pi}.

\begin{lemma}\label{lem:sizeofEm}
For every $m\in\N$ and every $J\in\mathscr{E}_m$, $\pi_\eta\#\nu_n(f(J))\geq (5n-6)^{-m-1}$.
\end{lemma}

\begin{proof}
Note that by (P7), there is some $I\in\mathscr{N}_{m+1}$ such that $f_{m+1}(I)$ is a leaf in $T_{m+1}$ and $I\subset J$. Then since $f_{m+1}(I)=\phi^\eta_{\w(I)}({\bf 0})$ is a leaf in $T_{m+1}$, we have that 
\[ f_{m+1}([0,1]\setminus I)\cap \phi^\eta_{\w(I)}(T^{\eta(\w(I))})=\emptyset.\]

We claim that $f_k([0,1]\setminus I)\cap \phi_{\w(I)}^\eta((0,1)^2)=\emptyset$ for every integer $k\geq m+1$. To this end, fix an integer $k\geq m+1$ and let $t\in [0,1]\setminus I$. By Lemma \ref{lem:OSC}, it is sufficient to show that there exists a word $u\in\mathcal{A}_n^{m+1}$ different from $\w(I)$ such that $f_k(t)\in \phi_u^\eta(\sq)$. By (P2), there exists a unique interval $Q\in(\mathscr{E}_{m+1}\cup\mathscr{F}_{m+1}\cup\mathscr{N}_{m+1})\setminus\{I\}$ with $t\in Q$.

If $Q\in \mathscr{F}_{m+1}$, then by (P5) there exists a word $u\in\mathcal{A}_n^{m+1}$ satisfying
\[f_k(t)=f_{m+1}(t)=\phi_u^\eta({\bf 0}).\]
Then since 
\[ f_{m+1}([0,1]\setminus I)\cap \phi^\eta_{\w(I)}(T^{\eta(\w(I))})=\emptyset,\] 
we have that $u\neq \w(I)$.

If $Q\in \mathscr{E}_{m+1}$, then by (P3) and Lemma \ref{lem:edge}, there exists a word $u\in\mathcal{A}_n^{m+1}$ such that either $f_{m+1}(Q)=\phi_u^\eta(e_b)$ or $f_{m+1}(Q)=\phi_u^\eta(e_l)$. In either case, by (P3), (P4), and (P5) we have that $f_k(Q)\subset \phi_u^\eta([0,1)^2)$. Note that $\phi_{\w(I)}^\eta(e_l)$ emanates from $\phi_{\w(I)}^\eta({\bf 0})$ upward and $\phi_{\w(I)}^\eta(e_b)$ emanates from $\phi_{\w(I)}^\eta({\bf 0})$ to the right. Since $\phi_{\w(I)}^\eta({\bf 0})$ is a leaf in $T_{m+1}$, by Lemma \ref{lem:leaf} we must have that $u\neq \w(I)$.

Finally, if $Q\in\mathscr{N}_{m+1}\setminus\{I\}$, then by (P1), $\w_{m+1}(Q)\neq \w(I)$. By (P3), (P4), and (P5), $f_k(Q)\subset \phi_{\w_{m+1}(Q)}^\eta([0,1)^2)$. Therefore, $f_k(t)\notin \phi_{\w(I)}^\eta((0,1)^2)$, completing the proof of the claim.

Thus, if $J\in\mathscr{E}_m$, then there is an interval $I\in\mathscr{N}_{m+1}$ with $I\subset J$ and this $I$ satisfies 
\[f([0,1]\setminus I)\cap \phi^\eta_{\w(I)}((0,1)^2)=\emptyset,\]
so $\phi^\eta_{\w(I)}((0,1)^2)\cap K^\eta\subset f(I)$. Since $I$ is compact, we have $K^\eta_{\w(I)}\subset f(I)$, hence $K^\eta_{\w(I)}\subset f(J)$ and $|\w(I)|=m+1$. Therefore, $\pi_\eta\#\nu_n(f(J))\geq  (5n-6)^{-m-1}$.
\end{proof}

We are now ready to prove Proposition \ref{thm:holder}.

\begin{proof}[Proof of Proposition \ref{thm:holder}]
Let $f_m$, $\mathscr{E}_m,\mathscr{N}_m,\mathscr{F}_m$, and $\w$ be as in Lemma \ref{lem:intermedparam} and Remark \ref{rem:w}. There exists a non-decreasing surjection $\zeta:[0,1]\to [0,1]$ such that 
\[ |\zeta(J)|=\pi_\eta\#\nu_n(f(J)), \quad \text{for every integer $m\geq 0$ and each $J\in\mathscr{E}_m\cup\mathscr{N}_m\cup\mathscr{F}_m$}.\]
Let $\mathscr{F}_\infty=\bigcup_{m\geq 0}\mathscr{F}_m$ and let $\mathscr{L}_\infty$ be the collection of intervals $I\in\bigcup_{m\geq 0}\mathscr{N}_m$ such that $\phi^\eta_{\w(I)}({\bf 0})$ has an edge emanating upward and an edge emanating to the right in $T_{|\w(I)|}$. Note that $\mathscr{L}_\infty=\bigcup_{m\geq 0}\{I\in\mathscr{N}_m:|\w(I)|<m\}$. By (P7) and by design of the families $\mathscr{F}_m$, $\mathscr{E}_m$, and $\mathscr{N}_m$ we have that the restriction of $\zeta$ on $[0,1]\setminus\bigcup(\mathscr{F}_\infty\cup\mathscr{L}_\infty)$ is strictly increasing and injective.

For any interval $J\in\mathscr{F}_\infty\cup\mathscr{L}_\infty$, let $p_J$ be the unique element of $\zeta(J)$ and let $q_J$ be the unique element of $f(J)$. Define the function $F:[0,1]\to K^\eta$ by 
\[ F(t) =
\begin{cases}
f\circ\zeta^{-1}(t), &\text{if $t\in \zeta([0,1]\setminus\bigcup(\mathscr{F}_\infty\cup\mathscr{L}_\infty))$},\\
q_J, &\text{if $t\in \zeta(J)$ where $J\in\mathscr{F}_\infty\cup\mathscr{L}_\infty$}.
\end{cases} \]

To complete the proof, it remains to show that $F$ is $(1/\a_n)$-H\"older continuous. To this end, fix distinct points $x,y\in [0,1]$ and assume that $x<y$, and let $m\geq 0$ be the maximal integer for which there exists an interval $J\in\mathscr{E}_m\cup\mathscr{N}_m$ with $x,y\in\overline{\zeta(J)}$. The remainder of the proof falls to a case study.

\emph{Case 1.} If there is an interval $I\in\mathscr{E}_{m+1}(J)$ such that $x$ and $y$ are separated by $\zeta(I)$, then by (P3) and (P4), $|F(x)-F(y)|\leq n^{-m}\sqrt{2}$ and $|x-y|\geq (5n-6)^{-m-2}$, thus $|F(x)-F(y)|\leq n^2\sqrt{2}|x-y|^{\a_n}$.

\emph{Case 2.} If there is an interval $I\in\mathscr{N}_{m+1}(J)\setminus\mathscr{L}_\infty$ such that $\zeta(I)$ separates $x$ and $y$, then proceeding similarly as in Case 1. we obtain $|F(x)-F(y)|\leq n^3 \sqrt{2}|x-y|^{\a_n}$.

\emph{Case 3.} If neither Case 1 nor Case 2 holds, then there exist intervals $I_1,I_2\in\mathscr{E}_{m+1}(J)\cup(\mathscr{N}_{m+1}(J)\setminus\mathscr{L}_\infty)$ such that $x\in\overline{\zeta(I_1)}$, $y\in\overline{\zeta(I_2)}$. Then the intersection $\overline{\zeta(I_1)}\cap\overline{\zeta(I_2)} =\{p\}$ for some $p$, thus we may apply Case 1 to the pairs $x,p$ and $p,y$ to obtain $|F(x)-F(y)|\leq n^3\sqrt{2}|x-y|^{\a_n}$.
\end{proof}

\section{Random choice functions}\label{sec:choice}

Fix an even integer $n\geq 4$. Note that the set $\mathscr{C}_n$ of choice functions $\eta:\mathcal{A}_n^* \to \{1,2\}$ can be identified with $\{1,2\}^{\mathcal{A}_n^*}$ which is a countable product of finite sets. Then by \cite[\textsection 3.1]{Stroock}, there exists a probability measure $\mu_n$ on $\mathscr{C}_n$ such that for any $w \in \mathcal{A}_n^*$
\[ \mu_n(\{\eta \in \mathscr{C}_n : \eta(w)=1\}) = \mu_n(\{\eta \in \mathscr{C}_n : \eta(w)=2\}) = \frac12.\]

We say that a choice function $\eta \in \mathscr{C}_n$ is a \emph{random choice function} if $\nu_n$-a.e. word $w\in \mathcal{A}^{\N}_n$ satisfies the following property: for each $N \in \N$ and every integer $k\geq 0$, there exists $\ell \in \N$ depending on $w,k,N$ such that
\begin{enumerate}
\item[(R1)] the restrictions
\[ \eta\left|\bigcup_{j=0}^{k-1}\mathcal{A}^{\ell+j}_{n,w(\ell)} = 1 \right. \qquad \text{and} \qquad \eta\left| \bigcup_{j=0}^{2N+k-1}\mathcal{A}^{\ell-N+j}_{n,w(\ell-N)} \setminus\bigcup_{j=0}^{k-1} \mathcal{A}^{\ell+j}_{n,w(\ell)}\right. =2, \]
\item[(R2)] if $w=i_1 i_2\cdots$, then $i_{\ell-N+1}>4n-4$.
\end{enumerate} 
For completeness, if $k=0$, then the union $\bigcup_{j=0}^{k-1}\mathcal{A}^{\ell+j}_{n,w(\ell)}$ is taken over an empty set thus we regard the union itself to be empty, and so we consider only the second restriction of (R1) in this case. See Figure \ref{fig:choice} for a schematic representation of property (R1).

\begin{figure}[h]
\centering
\includegraphics[width=0.5\textwidth]{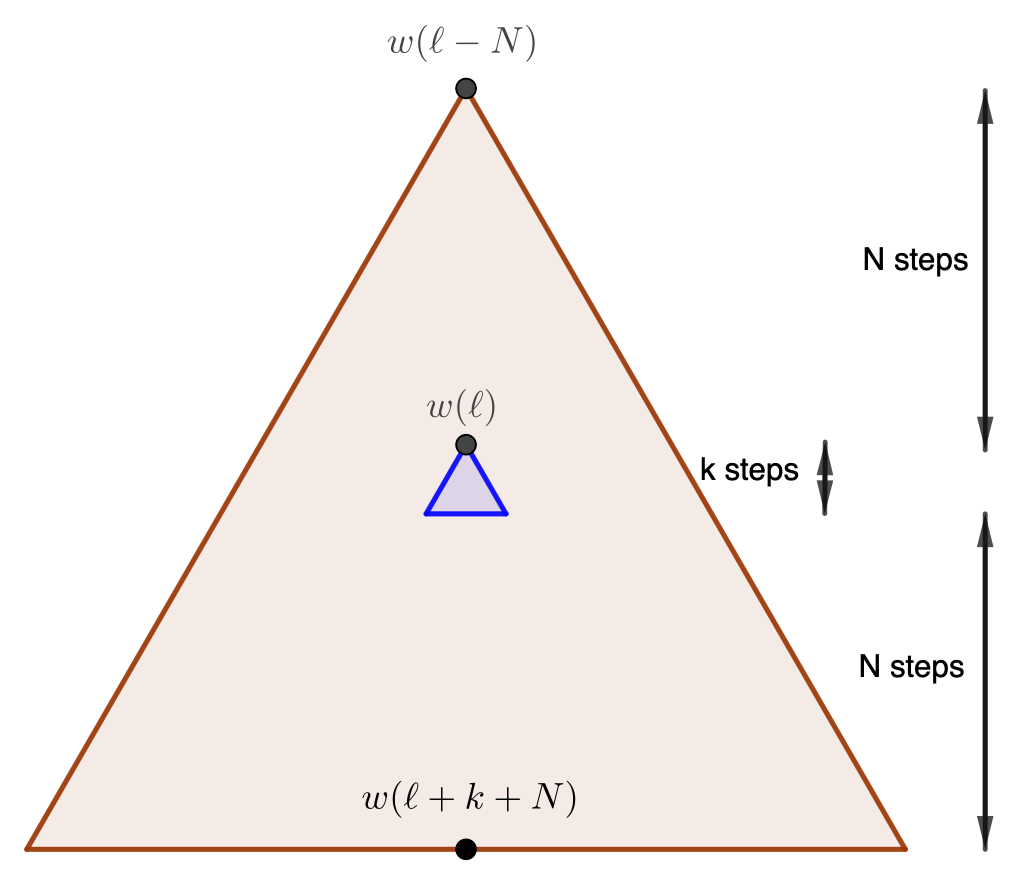}
\caption{The big triangle represents all words in $\mathcal{A}_n^*$ that start with $w(\ell-N)$ and have at most $\ell+k+N$ letters. Words at the same altitude have the same number of letters. The part of this set where $\eta$ takes the value 2 is colored by red, and the part of the set where $\eta$ takes the value 1 is colored by blue.}
\label{fig:choice}
\end{figure}

To simplify our notation, for a word $w\in \mathcal{A}_n^\N$, a function $\eta\in \mathscr{C}_n$, and two integers $N\geq 1,k\geq 0$, we say $(\eta,w,N,k)$ satisfies (R1), (R2) if there exists $\ell\in \N$ for which (R1) and (R2) hold. For a word $w\in \mathcal{A}^{\N}_n$ and a function $\eta \in \mathscr{C}_n$, we say that $(\eta,w)$ satisfies (R1) and (R2) if for every pair of integers $N\geq 1,k\geq 0$ , we have that $(\eta,w,N,k)$ satisfies (R1) and (R2). Thus, a choice function $\eta$ is random if for $\nu_n$-a.e. $w\in \mathcal{A}^{\N}_n$, $(\eta,w)$ satisfies (R1) and (R2).

The goal of this section is to show that random choice functions exist. In fact, we prove the following stronger statement.

\begin{proposition}\label{prop:randomchoice}
For every even integer $n\geq 4$, $\mu_n$-a.e. $\eta \in \mathscr{C}_n$ is random.
\end{proposition}

Since $\pi_\eta\#\nu_n$ and $\mathcal{H}^{\alpha_n}\res K^\eta$ have the same null sets in $K^\eta$ for every $\eta\in \mathscr{C}_n$, Proposition \ref{prop:randomchoice} yields immediately the following corollary.

\begin{corollary}\label{thm:pointchoice}
For every even integer $n\geq 4$, there exists $\eta \in \mathscr{C}_n$ such that $\mathcal{H}^{\alpha_n}$-almost every point $x\in K^\eta$ has $\pi_\eta^{-1}(\{x\})=\{w\}$ and $(\eta,w)$ satisfies (R1) and (R2). 
\end{corollary}

Before proving Proposition \ref{prop:randomchoice}, we establish some notation. For $w\in \mathcal{A}_n^\N$ define 
\[ D^w :=\{\eta\in \mathscr{C}_n: (\eta,w)\text{ satisfies (R1) and (R2)}\}\]
and for $\eta\in \mathscr{C}_n$ define 
\[C_\eta := \{w\in \mathcal{A}_n^\N :(\eta,w)\text{ satisfies (R1) and (R2)}\}.\]
Define also
\[\Delta_n:=\{(\eta,w)\in \mathscr{C}_n\times \mathcal{A}_n^\N:w\in C_\eta\}=\{(\eta,w)\in \mathscr{C}_n\times \mathcal{A}_n^\N: \eta\in D^w\}.\]

\begin{proof}[{Proof of Proposition \ref{prop:randomchoice}}]
Using the notation above, we claim that 
\begin{equation}\label{eq:randomchoice}
 \mu_n(D^w)=1, \quad \text{for $\nu_n$-almost every word $w\in \mathcal{A}_n^\N$.}  
\end{equation}
Assuming \eqref{eq:randomchoice}, by Fubini's theorem we obtain
\[\int_{\mathscr{C}_n} \nu_n(C_\eta)d\mu_n(\eta)=
(\mu_n\times\nu_n)(\Delta_n) = 
\int_{\mathcal{A}_n^\N}\mu_n(D^w)d\nu_n(w) =1.\]
Therefore, $\nu_n(C_\eta)=1$ for $\mu_n$-a.e. $\eta \in \mathscr{C}_n$ which implies that $\mu_n$-a.e. $\eta \in \mathscr{C}_n$ is random.

For the proof of \eqref{eq:randomchoice}, recall that
\[ \nu_n(\{i_1i_2\cdots \in \mathcal{A}_n^\N :i_j>4n-4\text{ infinitely often}\})=1,\] 
so it suffices to show \eqref{eq:randomchoice} for every word $w$ in the above set. 

Fix $w=i_1i_2\cdots\in \mathcal{A}_n^\N$ such that $i_j>4n-4$ infinitely often. Fix also integers $N\geq 1$ and $k\geq 0$. We construct a sequence $(M_m)_{m\in\N}\subset \N$ in an inductive fashion. Let $M_1\in \N$ be such that $i_{M_1+1}>4n-4$. Assuming that we have defined $M_m$ for some $m\in\N$, let $M_{m+1} \in \N$ such that $M_{m+1} > M_m + 2N+k+1$ and $i_{M_{m+1}+1} >4n-4$.

For each $m\in\N$ let $E_{m,w}$ be the set of all choice functions $\eta\in\mathscr{C}_n$ such that $(\eta,w,N,k)$ satisfies (R1) and (R2) with $\ell=M_m+N$. Since $i_{M_m+1}>4n-4$ for each $m\in\N$, (R2) is satisfied for $\ell=M_m+N$. For each $m\in\N$, let 
\[B_m=\bigcup_{j=0}^{2N+k-1}\mathcal{A}_{n,w(M_m)}^{M_m+j},\] 
that is, the set of of all finite words $u$ starting with $w(M_m)$ and having length at most $M_m+2N+k-1$. Define also for each $m\in\N$ the function $\xi_m:B_m\to \{1,2\}$ via 
\[ \xi_m \left|  \bigcup_{j=0}^{k-1}\mathcal{A}_{n,w(M_m)}^{M_m+j}=1 \right. \qquad\text{and}\qquad \xi_m\left| B_m\setminus \bigcup_{j=0}^{k-1}\mathcal{A}_{n,w(M_m)}^{M_m+j}=2. \right.\] 
Then $E_{m,w}=\{\eta\in \mathscr{C}_n:\eta|_{B_m}=\xi_m\}$ and $|B_m|=\tfrac{(5n-6)^{2N+k}-1}{5n-7}$ for every $m\in\N$. By uniformity of the measure $\mu_n$, we have that for every $m\in\N$, $\mu_n(E_{m,w})=2^{-|B_m|}$, in particular each set $E_{m,w}$ has the same measure independent of $m$, and that measure is positive. 

Next we claim that the sets $(E_{m,w})_{m\in\N}$ are independent. To see this, let $\{E_{m_1,w},\cdots,E_{m_j,w}\}$ be a finite collection of the sets $(E_{m,w})_{m\in\N}$, and note that
\[\bigcap\limits_{i=1}^j E_{m_i,w}=\{\eta\in\mathscr{C}_n:\eta|_{B_{m_i}}=\xi_{m_i}\text{ for }i\in\{1,\cdots,j\}\}.\]
Since the sets $\{B_m\}_{m\in\N}$ are pairwise disjoint as $M_{m+1}-M_m>2N+k+1$ for each $m\in\N$, we have that $|\bigcup_{i=1}^j B_{m_i}|=\sum_{i=1}^j|B_{m_i}|$ and $\mu_n(\bigcap_{i=1}^j E_{m_i,w})=2^{-|\bigcup_{i=1}^j B_{m_i}|}$. Thus, 
\[ \mu_n(\bigcap_{i=1}^j E_{m_i,w})=\prod_{i=1}^j\mu_n(E_{m_i,w}),\] hence the sets $(E_{m,w})_{m\in\N}$ are independent. 

Furthermore, each set $E_{m,w}$ has the same positive measure, therefore 
\[ \sum_{m=1}^\infty \mu_n(E_{m,w})=\infty.\] 
Thus, by the second Borel-Cantelli Lemma, 
\[ \mu_n(\bigcap_{j=1}^\infty\bigcup_{m=j}^\infty E_{m,w})=1,\] 
or equivalently, for $\mu_n$-almost every choice function $\eta\in\mathscr{C}_n$, $(\eta,w,N,k)$ satisfies (R1) and (R2) infinitely often. In particular, for every pair of integers $N\geq 1,k\geq 0$ we have that for $\mu_n$-almost every $\eta\in\mathscr{C}_n$, $(\eta,w,N,k)$ satisfies (R1) and (R2). Since there are countably many choices of $N$ and $k$, we have immediately that $\mu_n(D^w)=1$ for $\nu_n$-almost every word $w\in\mathcal{A}^\N_n$.
\end{proof}

\section{Tangents of carpets $K^{\eta}$ at typical points}\label{sec:tangents}

Fix for the rest of this section an even integer $n\geq 4$. Recall the definitions of choice functions $\eta\in\mathscr{C}_n$ and carpets $K^{\eta}$ from Section \ref{sec:IFS}, the number $\a_n = \frac{\log(5n-6)}{\log(n)}$ from Section \ref{sec:measure}, and the definition of random choice functions from Section \ref{sec:choice}.

In this section we prove the following result about tangents of carpets $K^{\eta}$ at typical points when $\eta$ is random.

\begin{proposition}\label{prop:tangents}
If $\eta\in \mathscr{C}_n$ is a random choice function, then for $\mathcal{H}^{\alpha_n}$-a.e. $x\in K^\eta$ there exist $T_{n,0},T_{n,1},T_{n,2},\dots\in \tang(K^\eta,x)$ such that $T_{n,k}$ has exactly $\frac1{5n-7}((5n-6)^k-1)$ many cut points.
\end{proposition}

In \textsection\ref{sec:K^n,k} we study the local cut-points of a certain class of ``almost self-similar'' carpets and in \textsection\ref{sec:tangentsofK^eta}, we relate the tangents of $K^{\eta}$ with these carpets.

\subsection{Local cut-points in a class of carpets}\label{sec:K^n,k} 

Define, for an integer $k\geq 0$, the set 
\begin{equation}\label{eq:setK}
K^{n,k}:=\bigcup\limits_{w\in \mathcal{A}_n^k}\psi_{n,w}^1\left(\bigcap\limits_{m=1}^\infty\bigcup\limits_{v\in \mathcal{A}_n^m} \psi_{n,v}^2(\sq)\right)
\end{equation}
where maps $\psi_{n,v}^i$ and $\phi_v^\eta$ for $v\in\mathcal{A}_n^*$ and $i\in\{1,2\}$ are as in Section \ref{sec:IFS}.

Roughly speaking, $K^{n,k}$ is obtained by first iterating $k$ times the maps from the system $\mathcal{F}_n^1$, and then applying these iterates to the self-similar attractor of the IFS of similarities $\mathcal{F}_n^2$ (recall the systems $\mathcal{F}_n^1,\mathcal{F}_n^2$ from Section \ref{sec:IFS}).

The next lemma is the main result of this subsection.

\begin{lemma}\label{lem:cut}
For each integer $k\geq 0$, $K^{n,k}$ has exactly $\frac{1}{5n-7}((5n-6)^k-1)$ many local cut-points.
\end{lemma}

The proof of Lemma \ref{lem:cut} is by induction on $k$. The base case $k=0$ is given in Lemma \ref{lem:cutzero}. In the proof, we make use of the following definition.

Let $\{f_j(x)=Lx + b_j\}_{j=1}^l$ with $b_j\in\R^2$ be similarities on $\R^2$ where $L>0$ 
and $l\geq 2$ such that 
\begin{enumerate}
\item if $i,j\in\{1,\dots,l\}$ are distinct, then, $f_{i}((0,1)^2)\cap f_{j}((0,1)^2)=\emptyset$, and
\item $\bigcap_{i=1}^l f_i(\sq)\neq\emptyset$.
\end{enumerate}
We say that $f_1,\dots,f_l$ \emph{meet at edges of $\sq$} if for every $i\in\{1,\dots,l\}$, there exists $j\in\{1,\dots,l\}\setminus\{i\}$ such that $f_{i}(\sq)\cap f_{j}(\sq)$ is an edge of $f_{i}(\sq)$. Note that necessarily $l\in\{2,3,4\}$.

\begin{rem}\label{rem:edges}
By Lemma \ref{lem:fact} and by definition of $K^{n,k}$, we have that for any $k\in\N$,
\begin{align*}
\bigcup_{j\in\mathcal{A}_n}\psi_{n,j}^2(\partial\sq)\subset &K^{n,0} \subset \bigcup_{j\in\mathcal{A}_n}\psi_{n,j}^2(\sq)\\ 
\bigcup_{j\in\mathcal{A}_n}\psi_{n,j}^1(\partial\sq)\subset &K^{n,k}\subset \bigcup_{j\in\mathcal{A}_n}\psi_{n,j}^1(\sq).
\end{align*}
Note also that if there exist distinct $i,j\in\mathcal{A}_n$ and $x\in\sq$ with $x\in\psi_{n,j}^2(\sq)\cap\psi_{n,i}^2(\sq)$, then the maps in $\{\psi_{n,l}^2:x\in\psi_{n,l}^2(\sq),l\in\mathcal{A}_n\}$ meet at edges of $\sq$. Note further that if there exist distinct $i,j\in\mathcal{A}_n$ and $x\in\sq$ with $x\in\psi_{n,j}^2(\sq)\cap\psi_{n,i}^2(\sq)$, then either $x=(1/2,2/n)$ or the functions in the set $\{\psi_{n,l}^1:x\in\psi_{n,l}^1(\sq),l\in\mathcal{A}_n\}$ meet at edges of $\sq$. See for example Figure \ref{fig:1} for the case $n=6$. 
\end{rem}

\begin{lemma}\label{lem:claim1}
Let $\eta\in\mathscr{C}_n$, let $m\in\N$, let $u,v,w\in\mathcal{A}_n^m$ be distinct, and let
\[x\in\phi_u^\eta(\sq)\cap\phi_v^\eta(\sq)\cap\phi_w^\eta(\sq).\]
Then the maps $\{\phi_\beta^\eta:x\in\phi_\beta^\eta(\sq),\beta\in\mathcal{A}_n^m\}$ meet at edges of $\sq$.
\end{lemma}

\begin{proof}
Recall that for all $\beta\in \mathcal{A}_n^m$, $\phi_\beta^\eta(x)=n^{-m}y+\phi_\beta^\eta({\bf 0})$ with $\phi_\beta^\eta({\bf 0})\in\{0,\frac{1}{n^m},\dots,\frac{n^m-1}{n^m}\}^2$.

Let $E=\{\beta\in\mathcal{A}_n^m:x\in\phi_\beta^\eta(\sq)\}$ which, by assumption, contains at least three words. Clearly the intersection $\bigcap_{\beta\in E}\phi_\beta^\eta(\sq)$ is nonempty as it contains $x$. 

Suppose that $a,b\in E$ are distinct words such that $\phi_a^\eta,\phi_b^\eta$ do not meet at edges of $\sq$. Then  $\phi_a^\eta({\bf 0})-\phi_b^\eta({\bf 0})\in\{-n^{-m},n^{-m}\}^2$. Since $\card(E)\geq 3$, there exists $c\in E\setminus\{a,b\}$ with
\[\phi_a^\eta({\bf 0})-\phi_c^\eta({\bf 0})\in\{-n^{-m},0,n^{-m}\}^2\quad\text{and}\quad \phi_b^\eta({\bf 0})-\phi_c^\eta({\bf 0})\in\{-n^{-m},0,n^{-m}\}^2.\]
If, for example, $\phi_a^\eta({\bf 0})-\phi_c^\eta({\bf 0})\in\{-n^{-m},n^{-m}\}^2$, then 
\[\phi_c^\eta({\bf 0})-\phi_b^\eta({\bf 0})=(\phi_c^\eta({\bf 0})-\phi_a^\eta({\bf 0}))+(\phi_a^\eta({\bf 0})-\phi_b^\eta({\bf 0}))\in\{-2n^{-m},0,2n^{-m}\}^2\cap\{-n^{-m},0,n^{-m}\}^2.\]
But then $\phi_c^\eta({\bf 0})=\phi_b^\eta({\bf 0})$, and this is a contradiction. Hence, $\phi_a^\eta({\bf 0})-\phi_c^\eta({\bf 0})\notin\{-n^{-m},n^{-m}\}^2$, and by a simple calculation we get $\phi_a^\eta(\sq)\cap\phi_c^\eta(\sq)$ is equal to an edge of $\phi_a^\eta(\sq)$ containing $x$, and similarly for $\phi_b^\eta(\sq)\cap \phi_c^\eta(\sq)$. From this we may conclude that the maps in $\{\phi_\beta^\eta:\beta\in E\}$ meet at edges of $\sq$.
\end{proof}

\begin{lemma}\label{lem:claim2}
Let $m\in\N$, let $u,v\in\mathcal{A}_n^m$ be distinct, and let $x\in K^{n,0}$ with $x\in\psi_{n,u}^2(K^{n,0}) \cap \psi_{n,v}^2(K^{n,0})$. Then the maps in $\{\psi_{n,w}^2 :x\in\psi_{n,w}^2(K^{n,0}),w\in\mathcal{A}_n^m\}$ meet at edges of $\sq$.
\end{lemma}

\begin{proof}
The proof is by induction on $m$. The base case $m=1$ holds by Remark \ref{rem:edges}. 

Assume now that the claim holds for some $m\in\N$. Let $x\in K^{n,0}$ and let $u,v\in\mathcal{A}_n^{m+1}$ be distinct words such that $x\in\psi_{n,u}^2(K^{n,0})\cap \psi_{n,v}^2(K^{n,0})$. By Lemma \ref{lem:claim1}, we may assume that $\{w\in\mathcal{A}_n^{m+1}:x\in\psi_{n,w}^2(K^{n,0})\}=\{u,v\}$. We prove that $\psi_{n,u}^2(K^{n,0})\cap \psi_{n,w}^2(K^{n,0})$ is an edge of $\psi_{n,u}^2(\sq)$. 

Let $u=a_1\dots a_{m+1}$, $v=b_1\dots b_{m+1}$. If $a_1=b_1$, then applying $(\psi_{n,a_1}^2)^{-1}$ returns us to the claim at $m$, hence by the inductive hypothesis we may assume that $a_1\neq b_1$. If $x$ is not a vertex of $\psi_{n,u}^2(\sq)$, then the claim holds by Lemma \ref{lem:capface}. Thus we may assume that $x$ is a vertex of $\psi_{n,u}^2(\sq)$ and of $\psi_{n,v}^2(\sq)$.  If there exists a word $w\in\mathcal{A}_n^m\setminus\{u(m),v(m)\}$ for which $x\in\psi_{n,w}^2(K^{n,0})$, then there exists a $j\in\mathcal{A}_n$ with $x\in\psi_{n,u}^2(K^{n,0})\cap \psi_{n,v}^2(K^{n,0})\cap\psi_{n,wj}^2(K^{n,0})$, and this is a contradiction. 

Therefore, we may further assume that $\{w\in\mathcal{A}_n^m:x\in\psi_{n,w}^2(K^{n,0})\}=\{v(m),u(m)\}$. Then by the inductive hypothesis, $\psi_{n,u(m)}^2(K^{n,0})\cap\psi_{n,v(m)}^2(K^{n,0})$ is an edge of $\psi_{n,u(m)}^2(\sq)$. The remainder of the proof is a case study on which edge of $\psi_{n,u(m)}^2(\sq)$ this intersection is equal to, but these cases are all essentially identical so we show only one of them and leave the remaining three to the reader. Assume that
\[\psi_{n,u(m)}^2(K^{n,0})\cap \psi_{n,v(m)}^2(K^{n,0})=\psi_{n,u(m)}^2([0,1]\times\{0\}),\]
and as the maps $\psi_{n,u(m)}^2$ and $\psi_{n,v(m)}^2$ are both rotation-free and reflection-free similarity maps with Lipschitz norms $n^{-m}$, by Remark \ref{rem:affine} this intersection must also be equal to the edge $\psi_{n,v(m)}^2([0,1]\times\{1\})$. Then 
\[ \psi_{n,u(m)}^2({\bf 0})-\psi_{n,v(m)}^2({\bf 0})=(0,n^{-m}), \quad \psi_{n,a_{m+1}}^2({\bf 0})=(\tfrac{k}{n},0), \quad \psi_{n,b_{m+1}}^2({\bf 0})=(\tfrac{k\pm 1}{n},1)\] 
with $k,k\pm 1\in\{0,1,\dots,n-1\}$. Then there is some $j\in\mathcal{A}_n$ such that $\psi_{n,j}^2({\bf 0})=(\frac{k}{n},1)$, thus by a straightforward calculation we obtain $x\in\psi_{n,v(m)j}^2(K^{n,0})$. As $a_1\neq b_1$ and necessarily $j\neq b_{m+1}$, $v(m)j\in\mathcal{A}_n^{m+1}\setminus\{v,u\}$, and this contradicts our assumption that $\{w\in\mathcal{A}_n^{m+1}:x\in\psi_{n,w}^2(K^{n,0})\}=\{u,v\}$, concluding the proof.
\end{proof}

The next lemma is the base case for the proof of Lemma \ref{lem:cut}.

\begin{lemma}\label{lem:cutzero}
The set $K^{n,0}$ does not contain local cut-points.
\end{lemma}

\begin{proof}
It suffices to show that $K^{n,0}$ contains no cut-points. Assuming this to be true, the proof of the lemma follows from \cite[Theorem 1.6]{DLRW}, as Lemma \ref{lem:claim1} and Lemma \ref{lem:claim2} alongside the fact that $K^{n,0}$ is a self-similar set with no cut-points are sufficient to imply the hypotheses of this theorem.

To prove that $K^{n,0}$ contains no cut-points, fix $y\in K^{n,0}$ and distinct $p,q\in K^{n,0}\setminus\{y\}$. Let $m\in\N$ such that there exist distinct $u,v,w\in\mathcal{A}_n^m$ for which $\psi_{n,u}^2(K^{n,0})$, $\psi_{n,v}^2(K^{n,0})$, and $\psi_{n,w}^2(K^{n,0})$ are pairwise disjoint with $p\in\psi_{n,v}^2(K^{n,0})$, $q\in\psi_{n,u}^2(K^{n,0})$, and $y\in\psi_{n,w}^2(K^{n,0})$. Then by Lemma \ref{lem:claim2} and by \cite[Lemma 3.1]{BV2}, there exist distinct $u_1,\dots,u_k\in\mathcal{A}_n^m$ for which $p\in\psi_{n,u_1}^2(K^{n,0})$, $q\in\psi_{n,u_k}^2(K^{n,0})$, and for every $j\in\{1,\dots,k-1\}$, $\psi_{n,u_j}^2(K^{n,0})\cap \psi_{n,u_{j+1}}^2(K^{n,0})$ contains at least two points. Now let $z_0,\dots,z_k\in K^{n,0}$ be distinct points satisfying $z_0=p$, $z_k=q$, and $z_j\in(\psi_{n,u_j}^2(K^{n,0})\cap\psi_{n,u_{j+1}}^2(K^{n,0}))\setminus\{y\}$ for every $j\in\{1,\dots,k-1\}$ . Then since $\partial\sq\subset K^{n,0}$ by Lemma \ref{lem:fact}, for every $j\in\{1,\dots,k-2\}$ there is a continuous function $\g_j:[0,1]\to \psi_{n,u_{j+1}}^2(\partial\sq)\setminus\{y\}$ such that $\g_j(0)=z_j$, $\g_j(1)=z_{j+1}$ and necessarily $\g_j([0,1])\subset K^{n,0}\setminus\{y\}$. By \cite[Proposition 3.2]{BV2}, there exist continuous $\g_0:[0,1]\to \psi_{n,u_1}^2(K^{n,0})$ with $\g_0(0)=p$, $\g_0(1)=z_1$ and continuous $\g_{k-1}:[0,1]\to \psi_{n,u_k}^2(K^{n,0})$ with $\g_{k-1}(0)=z_{k-1}$, $\g_{k-1}(1)=q$. Concatenating these curves yields a path inside of $K^{n,0}$ connecting $p$ to $q$ and avoiding $y$, therefore $y$ cannot be a cut-point for $K^{n,0}$. 
\end{proof}

As a straightforward consequence, we obtain the following corollary.
\begin{corollary}\label{cor:cutedge}
If $f_1,\dots, f_j:\R^2\to \R^2$ with $j\in\{2,3,4\}$ is a collection of rotation-free and reflection-free similarities with a common Lipschitz norm which meet at edges of $\sq$, then $\bigcup_{i=1}^j f_i(K^{n,0})$ does not contain local cut-points.
\end{corollary}

\begin{proof}
We leave most of this proof to the reader. A proof follows from the observation that, up to re-scaling and translating, there are only $7$ ways for such a collection of maps to meet at edges of $\sq$, and all of these arrangements are present in $\{\psi_{n,j}^2:j\in\mathcal{A}_n\}$.
\end{proof}

We are now finally ready to prove Lemma \ref{lem:cut}.

\begin{proof}[{Proof of Lemma \ref{lem:cut}}]
Denote by $N_k$ the number of local cut-points in $K^{n,k}$. By Lemma \ref{lem:cutzero} we have $N_0=0$. It remains to show that $N_{k}=(5n-6) N_{k-1}+1$ for every $k\in\N$. 

By definition of sets $K^{n,k}$,
\begin{align*}
K^{n,k}&=\bigcup\limits_{v\in \mathcal{A}_n^{k}}\psi_{n,v}^1\left(\bigcap\limits_{m\in\N}\bigcup\limits_{w\in \mathcal{A}_n^m}\psi_{n,w}^2(\sq)\right)\\
&=\bigcup\limits_{j\in \mathcal{A}_n} \psi_{n,j}^1\left(\bigcup\limits_{v\in \mathcal{A}_n^{k-1}}\psi_{n,v}^1\left(\bigcap\limits_{m\in\N}\bigcup\limits_{w\in \mathcal{A}_n^m}\psi_{n,w}^2(\sq)\right)\right)\\
&=\bigcup\limits_{j\in \mathcal{A}_n} \psi_{n,j}^1(K^{n,k-1}).
\end{align*}
By Lemma \ref{lem:OSC}, each $\psi_{n,j}^1(K^{n,k-1}\cap (0,1)^2)$ contributes $N_{k-1}$ many local cut-points, and there is one additional local cut-point at $(1/2,2/n)$. Then by the last part of Remark \ref{rem:edges}, every point of $K^{n,k}\setminus(\bigcup_{j\in\mathcal{A}_n}\psi_{n,j}^1(K^{n,k-1}\cap(0,1)^2)$ other than $(1/2,2/n)$ is either contained in exactly one image $\psi_{n,v}^1(K^{n,0})$ for some $v\in\mathcal{A}_n^{k}$ or in an intersection of at least two such images of maps which meet at edges of $\sq$, hence there are no other local cut-points of $K^{n,k}$ by Corollary \ref{cor:cutedge}. Therefore, $N_{k}=(5n-6) N_{k-1}+1$. 
\end{proof}

\subsection{Proof of Proposition \ref{prop:tangents}}\label{sec:tangentsofK^eta}

Fix for the rest of this subsection an integer $k\geq 0$, a random choice function $\eta \in \mathscr{C}_n$, and an injective word 
\[w=i_1i_2\dots \in \mathcal{A}^{\N}_n\]
such that $(\eta,w)$ satisfy (R1) and (R2) from the definition of a random choice function (see the paragraph before Lemma \ref{lem:injfull} for the definition of injective words). Set $x=\pi_{\eta}(w)$. Given $N\in\N$, let $\ell_N\in\N$ be the integer given in properties (R1) and (R2). For $N\in\N$, define 
\begin{align*} 
X_N &:=n^{\ell_N}(K^\eta-x),\\ 
y_N &:=(\phi^\eta _{w(\ell_N)})^{-1}\circ\phi^\eta_{w(\ell_N+N+k-1)}({\bf 0}),\\
x_N &:=n^{\ell_N}(x-\phi^\eta_{w(\ell_N)}({\bf 0})).
\end{align*} 
Intuitively, $x_N$ are the bottom-left corners of the blow-ups of pieces of $K^\eta$ that contain $x$, while the points $(y_N)_{N\in\N}$ form a sequence of points in $\sq$ which approximate the points $x_N$. Additionally, the sets $X_N$ are a sequence of blow-ups of $K^\eta$ centered at the point $x$ which we use to find the desired tangents in $\tang(K^\eta,x)$.

By (R2), $i_{\ell_N-N+1}>(4n-4)$, hence $\psi^{\eta(w(\ell_N-N))}_{n,i_{\ell_N-N+1}}(\sq)$ does not intersect $\partial\sq$. Thus, for every $N\geq 3$ we get 
\begin{equation}\label{eqn:dist}
 \dist\left(K^\eta_{w(\ell_N)},K^\eta\setminus K^\eta_{w(\ell_N-N)}\right)\geq n^{-\ell_N-1+N}. 
\end{equation}

The general strategy of the proof of Proposition \ref{prop:tangents} is to first establish estimates for the Hausdorff distance between $X_N$ and translates of $K^{n,k}$ in a region near $\sq$, then establish estimates for the Hausdorff distance between the rest of $X_N$ and some blown-up and translated image of $K^{n,0}$ out to some large ball containing the origin. Then, exploiting the triangle inequality for excess (as in Remark \ref{rem:excess}), we show that some subsequence of $X_N$ converges to a tangent set for which all local cut-points (see \textsection\ref{sec:cut-points}) lie inside a unit square near the origin, for which outside of this square the tangent ``looks like" a tangent of $K^{n,0}$, and for which inside this unit square the tangent has exactly $\frac{1}{5n-7}((5n-6)^k-1)$ many local cut-points. 

\begin{lemma}\label{lem:closepoint}
For each $N\in\N$, $|y_N-x_N|\leq n^{-N-k+3}$.
\end{lemma}

\begin{proof}
Since $x=\pi_\eta(w)$, the point $x_N$ is in the set
\begin{align*}
n^{\ell_N}(\phi_{w(\ell_N+N+k-1)}^\eta(\sq)&-\phi_{w(\ell_N)}^\eta({\bf 0}))\\
&=n^{\ell_N}\left(\phi^\eta_{w(\ell_N)}\circ(\phi_{w(\ell_N)}^\eta )^{-1}\circ\phi^\eta_{w(\ell_N+N+k-1)}(\sq)-\phi^\eta_{w(\ell_N)}({\bf 0})\right).
\end{align*}
Now by Remark \ref{rem:affine}, the latter set is equal to $(\phi_{w(\ell_N)}^\eta )^{-1}\circ\phi^\eta_{w(\ell_N+N+k-1)}(\sq)$, which has diameter no more than $n^{-N-k+2}\sqrt{2}$ and contains $y_N$, thus $|y_N-x_N|\leq n^{-N-k+3}$.
\end{proof}

In the following lemma, we show that the parts of the sets $X_N$ inside of some unit square containing the origin approximate translates of $K^{n,k}$ in terms of Hausdorff distance.
\begin{lemma}\label{lem:lowexc}
For each $N\in\N$, the Hausdorff distance
\[\disth(n^{\ell_N}(K^\eta_{w(\ell_N)}-x),K^{n,k}-y_N)\leq n^{-N-k+4}.\]
\end{lemma}

\begin{proof}
Recall from \textsection\ref{subsec:tangents} that excess is given by $\exc(A,B):=\sup_{x\in A}\inf_{y\in B}|x-y|$.
Thus to prove the lemma, it suffices to show
\[\exc(n^{\ell_N}(K^\eta_{w(\ell_N)}-x),K^{n,k}-y_N)\leq n^{-N-k+4},\]
\[\exc(K^{n,k}-y_N,n^{\ell_N}(K^\eta_{w(\ell_N)}-x))\leq n^{-N-k+4}.\]
For the first inequality, by Lemma \ref{lem:closepoint} and by the triangle inequality for excess as in Remark \ref{rem:excess}, it is enough to show that
\[\exc(n^{\ell_N}(K^\eta_{w(\ell_N)}-x),K^{n,k}-x_N)\leq n^{-N-k+3}.\]
Note that since excess is translation invariant (see Remark \ref{rem:excess}), this excess is equal to
\[\exc(n^{\ell_N}(K^\eta_{w(\ell_N)}-\phi^\eta_{w(\ell_N)}({\bf 0})),K^{n,k}),\]
which is at most $n^{-N-k+2}\sqrt{2}$ by (R1) and by definition of $K^{n,k}$. Via a similar argument, we obtain the other inequality
\[\exc(K^{n,k}-y_N,n^{\ell_N}(K^\eta_{w(\ell_N)}-x))\leq n^{-N-k+4}. \qedhere\]
\end{proof}

Now define
\begin{align*}
F_N&:=n^{\ell_N}(K^{n,0}-\phi_{w(\ell_N)}^\eta({\bf 0}))-y_N,\\
Y_N&:=(F_N\setminus(\sq-y_N))\cup(K^{n,k}-y_N).
\end{align*}
The sets $F_N$ are simply blown-up shifts of the self-similar set $K^{n,0}$. The sets $Y_N$ are constructed by removing the copy of $K^{n,0}$ in $F_N$ which contains the origin, then replacing it by a shift of $K^{n,k}$. Before continuing, note that for $k=0$, $F_N=Y_N$.

\begin{lemma}\label{lem:samelim}
For every $r>0$, 
\begin{align*}
\lim\limits_{N\to\infty}\exc(Y_N\cap \overline{B}({\bf 0},r),X_N)=0 \quad \text{and} \quad
\lim\limits_{N\to\infty}\exc(X_N\cap \overline{B}({\bf 0},r),Y_N)=0.
\end{align*}
\end{lemma}

\begin{proof}
Fix $r>0$. By \eqref{eqn:dist} there exists $N_0\in\N$ such that for every integer $N\geq N_0$,
\begin{equation}\label{eqn:dist2}
\dist(K^\eta_{w(\ell_N)},K^\eta\setminus K^\eta_{w(\ell_N-N)})\geq n^{-\ell_N}2r.
\end{equation}
By (R2), $\dist(\partial\sq,(\phi_{w(\ell_N-N)}^\eta)^{-1}\circ\phi_{w(\ell_N)}^\eta({\bf 0}))\geq n^{-1}$. Since $y_N\in\sq$,
\[ \overline{B}({\bf 0},r)\subset n^N(\sq-(\phi_{w(\ell_N-N)}^\eta)^{-1}\circ\phi_{w(\ell_N)}^\eta({\bf 0}))-y_N.\]

Fix $y\in Y_N\cap \overline{B}({\bf 0},r)$. If $y\in\sq-y_N$, then by Lemma \ref{lem:lowexc} there exists some $z\in X_N$ such that $|y-z|\leq n^{-N-k+4}$. 

Assume now that $y\in Y_N\cap \overline{B}({\bf 0},r)\setminus(\sq-y_N)$. Let 
\[ Q=\bigcup_{v\in\mathcal{A}_n^{2N+k-1}}\psi_{n,v}^2(\sq) \quad \text{and} \quad P=\bigcup_{v\in\mathcal{A}_n^{2N+k-1}} \psi_{n,v}^2(\partial\sq).\] 
By (R2) and since $\phi_v^\eta(\partial\sq)\subset K^\eta$ for every finite word $v\in\mathcal{A}_n^*$ by Lemma \ref{lem:fact}, we have that 
\[ n^{\ell_N}((\phi_{w(\ell_N-N)}^\eta(P)-x)\setminus(\phi_{w(\ell_N)}^\eta(\sq)-x))\subset X_N.\] 
Since the maps $\psi_{n,j}^i$ are all similarities with Lipschitz norm $n^{-1}$, by Remark \ref{rem:affine} we have
\begin{align*}
n^{\ell_N}(\phi_{w(\ell_N-N)}^\eta(P)-x)&=n^{\ell_N}(\phi_{w(\ell_N-N)}^\eta(P)-\phi_{w(\ell_N-N)}^\eta\circ(\phi_{w(\ell_N-N)}^\eta)^{-1}(x))\\
&=n^N(P-(\phi_{w(\ell_N-N)}^\eta)^{-1}(x)).
\end{align*}
Similarly, we obtain 
$n^{\ell_N}(\phi_{w(\ell_N)}^\eta(\sq)-x)=\sq-(\phi_{w(\ell_N)}^\eta)^{-1}(x)$ and $x_N=(\phi_{w(\ell_N)}^\eta)^{-1}(x)$. Furthermore, $(\phi_{w(\ell_N-N)}^\eta)^{-1}(x)=(\phi_{w(\ell_N-N)}^\eta)^{-1}\circ\phi_{w(\ell_N)}^\eta(x_N)$, thus
\[(n^N(P-(\phi_{w(\ell_N-N)}^\eta)^{-1}\circ\phi_{w(\ell_N)}^\eta({\bf 0}))-x_N)\setminus(\sq-x_N)\subset X_N.\]

Since $y\in Y_N\cap \overline{B}({\bf 0},r)\setminus(\sq-y_N)$, by calculations similar to those above we obtain 
\begin{equation}\label{eqn:yN}
 y\in (n^N(Q-(\phi_{w(\ell_N-N)}^\eta)^{-1}\circ\phi_{w(\ell_N)}^\eta({\bf 0}))-y_N)\setminus(\sq-y_N).
 \end{equation}
Therefore, there exists $u\in \mathcal{A}_n^{2N+k-1}$ such that 

\[ y\in (n^N(\psi_{n,u}^2(\sq)-(\phi_{w(\ell_N-N)}^\eta)^{-1}\circ\phi_{w(\ell_N)}^\eta({\bf 0}))-y_N)\setminus(\sq-y_N),\]
which implies that there is a point 
\[p\in (n^N(\psi_{n,u}^2(\partial\sq)-(\phi_{w(\ell_N-N)}^\eta)^{-1}\circ\phi_{w(\ell_N)}^\eta({\bf 0}))-x_N)\setminus(\sq-x_N),\]
hence $p\in X_N$. Then by the triangle inequality and Lemma \ref{lem:closepoint} we have 
\begin{align*}
|p-y|\leq |x_N-y_N|+|(p+x_N)-(y+y_N)| \leq |x_N-y_N|+n^{-N-k+1}\sqrt{2} \leq n^{-N-k+4}
\end{align*} 
since $(p+x_N),(y+y_N)\in n^N\psi_{n,u}^2(\sq)$, and the latter set has diameter $n^{-N-k+1}\sqrt{2}$.

The other limit can be proven in the same way, yielding the result.\end{proof}

In the following two lemmas, we prove that the sequence $(Y_N)_{N\in\N}$ has a subsequence converging in the Attouch-Wets topology to a set with $\frac{1}{5n-7}((5n-6)^k-1)$ local cut-points and that the closeness of excesses from Lemma \ref{lem:samelim} implies that this set is also the Attouch-Wets limit of a subsequence of $(X_N)_{N\in\N}$, thus it is a tangent of $K^\eta$ at $x$.

\begin{lemma}\label{lem:tanY}
The sequence $(Y_N)_{N\in\N}$ has a subsequence converging with respect to the Attouch-Wets topology to a set $L_k\in\mathfrak{C}(\R^2;{\bf 0})$ that has exactly $\frac{1}{5n-7}((5n-6)^k-1)$ many local cut-points.
\end{lemma}

\begin{proof}
Let $r>3$ and note by definition of $Y_N$, by (\ref{eqn:dist2}), and by (\ref{eqn:yN}) that there exists a minimal $N_{0,r}\in \N$ such that for every integer $N\geq N_{0,r}$, 
\[(Y_N+y_N)\cap \overline{B}({\bf 0},r)=\overline{B}({\bf 0},r)\cap (n^N(K^{n,0}-(\phi_{w(\ell_N-N)}^\eta)^{-1}\circ\phi_{w(\ell_N)}^\eta({\bf 0}))\setminus\sq)\cup K^{n,k}.\]
Furthermore, $N_{0,r}$ is nondecreasing in $r$. By (R1), for each $N\in\N$ there exists a word $v_N\in\mathcal{A}_n^N$ for which $(\phi_{w(\ell_N-N)}^\eta)^{-1}\circ\phi^\eta_{w(\ell_N)}=\psi_{n,v_N}^2$. Then by Remark \ref{rem:affine} we obtain
\[\sq\cap n^N(K^{n,0}-\psi_{n,v_N}^2({\bf 0}))=n^N(\psi_{n,v_N}^2(K^{n,0})-\psi_{n,v_N}^2({\bf 0}))= K^{n,0}.\]
Thus, 
\[(Y_N+y_N)\cap\overline{B}({\bf 0},r)=\overline{B}({\bf 0},r)\cap (n^N(((K^{n,0}\setminus\psi_{n,v_N}^2(\sq))\cup(\psi_{n,v_N}^2(K^{n,k})))-\psi_{n,v_N}^2({\bf 0}))).\]
Since $\mathfrak{C}(\R^2;{\bf 0})$ is sequentially compact in the Attouch-Wets topology, the sequence of sets 
$(Y_N+y_N)_{N\in\N}$ has a subsequence $(Y_{N_j}+y_{N_j})_{j\in\N}$ converging to a limit $L$ in the Attouch-Wets topology and further satisfying $(y_{N_j})_{j\in\N}$ converges to a point $y_0\in\sq$ as $j\to\infty$. 

For $j\in\N$, let 
\begin{align*}
D_{r,j}&:=\{(m_1,m_2)\in\Z^2\setminus\{{\bf 0}\}:[m_1,m_1+1]\times[m_2,m_2+1]\subset \overline{B}({\bf 0},r),\\
&\text{ and there exists }v\in\mathcal{A}_n^{N_j}\text{ with } (m_1,m_2)=n^{N_j}(\psi_{n,v}^2({\bf 0})-\psi_{n,v_{N_j}}^2({\bf 0}))\}.
\end{align*}
If $j\in\N$ such that $N_j\geq N_{0,r+2}$, then
\begin{align*}
\bigcup\limits_{(m_1,m_2)\in D_{r,j}}((m_1,m_2)+K^{n,0})\cup K^{n,k}&\subset(Y_{N_j}+y_{N_j})\cap\overline{B}({\bf 0},r)\\
&\subset \bigcup\limits_{(m_1,m_2)\in D_{r+2,j}}((m_1,m_2)+K^{n,0})\cup K^{n,k}
\end{align*}
and $D_{r,j}\subset D_{r+2,j}$. Then since for every $N\in\N$, every $v\in\mathcal{A}_n^N$ has $n^N(\psi_{n,v}^2({\bf 0})-\psi_{n,v_N}^2({\bf 0}))\in\Z^2$, and since $(Y_{N_j}+y_{N_j})_{j\in\N}$ converges, there exists $j_{0,r}\in\N$ such that for every integer $j\geq j_{0,r}$, $D_{r,j}=D_{r,j_{0,r}}$. Indeed, if $D_{r,j}\neq D_{r,j+1}$ then
\[\exc((Y_{N_j}+y_{N_j})\cap\overline{B}({\bf 0},r),(Y_{N_{j+1}}+y_{N_{j+1}}))\geq n^{-1}\]
or vice-versa, which cannot happen infinitely often by Lemma \ref{lem:excauchy}. Consequently, for every $r>3$ and every integer $j\geq j_{0,r+4}$, we have that $(Y_{N_j}+y_{N_j})\cap\overline{B}({\bf 0},r)=\overline{B}({\bf 0},r)\cap L$, hence by Lemma \ref{lem:cut}, Corollary \ref{cor:cutedge}, and Remark \ref{rem:edges}, $L$ has exactly $\frac{1}{5n-7}((5n-6)^k-1)$ many local cut-points inside $\sq$ and no other local cut-points. Finally, $L_k=L-y_0$ is the desired limit.\end{proof}

\begin{rem}\label{rem:measure}
The tangent $L_k$ of Lemma \ref{lem:tanY} is a countable union $\bigcup_{j\in\N}S_j(K^{n,0})$ of rescaled copies of $K^{n,0}$, and if $i,j\in\N$ are distinct, then $S_i((0,1)^2)\cap S_j((0,1)^2)=\emptyset$. Since $\dimh(K^{n,0})=\alpha_n$ and $\mathcal{H}^{\alpha_n}(K^{n,0})>0$, by countable stability of Hausdorff dimension we have that $\dimh(L_k)=\alpha_n$ and $\mathcal{H}^{\alpha_n}(L_k)>0$. 
\end{rem}

The next lemma completes the proof of Proposition \ref{prop:tangents} by showing that the closeness of the sequences $Y_{N_j}$ and $X_{N_j}$ from Lemma \ref{lem:samelim} implies that they share a common Attouch-Wets limit.

\begin{lemma}\label{lem:exacttan}
Let $(N_j)_j\subset \N$ and $L_k\in\mathfrak{C}(\R^2;{\bf 0})$ be the subsequence and limit set from Lemma \ref{lem:tanY}, respectively. Then $(X_{N_j})_{j\in\N}$ converges to $L_k$ with respect to the Attouch-Wets topology.
\end{lemma}

\begin{proof}
Since $\R^2$ has the Heine-Borel property, in order to show that $\lim_{j\to\infty} X_{N_j}= L_k$ with respect to the Attouch-Wets topology, by \cite[Lemma 2.1.2]{Beer} and \cite[Lemma 3.1.4]{Beer} it is sufficient to verify the following two conditions.
\begin{enumerate}
\item[(i)] If $V\subset \R^2$ is an open set such that $V\cap L_k\neq\emptyset$, then there exists $j_0\in\N$ such that for every integer $j\geq j_0$, $X_{N_j}\cap V\neq\emptyset$ as well.
\item[(ii)] If $\delta>0$ and $z\in\R^2$ such that $B(z,\delta)\cap L_k=\emptyset$, then for every $\e\in (0,\delta)$ there exists $j_0\in\N$ such that for every integer $j\geq j_0$, $B(z,\e)\cap X_{N_j}=\emptyset$ as well.
\end{enumerate}

To verify (i), fix such an open set $V\subset\R^2$. Then there exist $z\in V$ and $\e>0$ such that $B(z,\e)\subset V$ and $B(z,\e/2)\cap L_k\neq\emptyset$. Assume for the sake of contradiction that (i) fails to hold for this $V$, and thus $X_{N_j}\cap B(z,\e)=\emptyset$ infinitely often. Now by Lemma \ref{lem:samelim} and since $\lim_{j\to\infty} Y_{N_j}=L_k$, there exists $j_0\in\N$ such that for every integer $j\geq j_0$, $\exc(Y_{N_j}\cap \overline{B}({\bf 0},|z|+\e),X_{N_j})<\e/2$ and $B(z,\e/2)\cap Y_{N_j}\neq\emptyset$. By our assumption, there exists an integer $j\geq j_0$ such that $X_{N_j}\cap B(z,\e)=\emptyset$ and there exists a point $y\in Y_{N_j}\cap B(z,\e/2)$. Then we have
\[\exc(Y_{N_j}\cap \overline{B}({\bf 0},|z|+\e),X_{N_j})\geq \dist(y,X_{N_j})\geq\e/2,\]
which is a contradiction.

Now to verify (ii), fix $\delta>0$ and $z\in\R^2$ satisfying $B(z,\delta)\cap L_k=\emptyset$ and let $\e\in(0,\delta)$. Since $Y_{N_j}\to L_k$ as $j\to\infty$ with respect to the Attouch-Wets topology, there is a $j_0\in\N$ such that for every integer $j\geq j_0$, $Y_{N_j}\cap B(z,\frac{\e+\delta}{2})=\emptyset$. Furthermore by Lemma \ref{lem:samelim}, choosing $j_0$ large enough, we may assume that for every integer $j\geq j_0$, $\exc(X_{N_j}\cap \overline{B}({\bf 0},|z|+\delta),Y_{N_j})<\frac{\delta-\e}{2}$. Then for every integer $j\geq j_0$ and for each $p\in X_{N_j}\cap B({\bf 0},|z|+\delta)$, $\dist(p,Y_{N_j})<\frac{\delta-\e}{2}$. Since $Y_{N_j}\cap B(z,\frac{\e+\delta}{2})=\emptyset$, if $p\in B(z,\e)$, then $\dist(p,Y_{N_j})\geq \frac{\delta-\e}{2}$ and $p\in B({\bf 0},|z|+\delta)$. Therefore for each integer $j\geq j_0$, $X_{N_j}\cap B(z,\e)=\emptyset$, as desired.
\end{proof}

\section{Proof of Theorem \ref{thm:main}}\label{sec:proof}

The proof of Theorem \ref{thm:main} in the special case that $s=\a_n := \frac{\log(5n-6)}{\log(n)}$ follows from several of the propositions we have established so far. For the general case, we require the following lemma which shows that bi-H\"older equivalence is hereditary. See \cite{Li21} for a quasisymmetric version of this lemma.

\begin{lemma}\label{lem:biliptan}
Let $f:X\to Y$ be a $(1/t)$-bi-H\"older homeomorphism between two closed sets $X\subset \R^d$ and $Y\subset \R^m$ for some $t>0$. Then for every point $x_0\in X$ and every tangent $T\in\tang(X,x_0)$, there exist a tangent $T'\in\tang(Y,f(x_0))$ and a $(1/t)$-bi-H\"older homeomorphism $g:T\to T'$.   
\end{lemma}

\begin{proof}
Assume first that $t\geq 1$. Let $x_0\in X$, let $T\in\tang(X,x_0)$, and let $r_j\to 0$ be a sequence of scales for which $r_j^{-1}(X-x_0)\to T$ as $j\to\infty$ in the Attouch-Wets topology. For each $j\in\N$, define the function 
\[ f_j:r_j^{-1}(X-x_0)\to r_j^{-1/t}(Y-f(x_0)) \quad \text{with}\quad f_j(z)=r_j^{-1/t}(f(r_jz+x_0)-f(x_0)).\] 
It is easy to see that each $f_j$ is a $(1/t)$-bi-H\"older homeomorphism with a bi-H\"older coefficient $H$ independent of $j$. By McShane's extension theorem \cite[Corollary 1]{mcshane}, each $f_j$ extends to a $(1/t)$-H\"older map $F_j:\R^d\to \R^m$ with the H\"older coefficient equal to $H$. Passing to a subsequence, by the Arzel\'a-Ascoli theorem, we may assume that $\{F_j\}$ converges locally uniformly to a $(1/t)$-H\"older map $F:\R^d\to \R^m$. 

Since the maps $F_j|T$ are $(1/t)$-bi-H\"older with the same bi-H\"older coefficients, $F|_T:T \to F(T)$ is a $(1/t)$-bi-H\"older homeomorphism. We claim that $r_j^{-1/t}(Y-f(x_0))\to F(T)$ as $j\to\infty$ in the Attouch-Wets topology. First note that 
\[ r_j^{-1/t}(Y-f(x_0))=F_j(r_j^{-1}(X-x_0)).\] 
Let $R>0$, let $\e>0$, and let $j_0\in\N$ such that for every integer $j\geq j_0$, 
\begin{align*}
\|F_j-F\|_{\infty,\overline{B}({\bf 0},R)}&<\e\\
\exc(\overline{B}({\bf 0},R)\cap r_j^{-1}(X-x_0),T)&<\e\\
\exc(\overline{B}({\bf 0},R)\cap T,r_j^{-1}(X-x_0))&<\e.
\end{align*}
For an integer $j\geq j_0$, given $z\in \overline{B}({\bf 0},R)\cap r_j^{-1}(X-x_0)$ there exists $y\in T$ with $|z-y|<\e$, so
\[|F_j(z)-F(y)|\leq |F_j(y)-F_j(z)|+|F_j(z)-F(z)|< H \e^{1/t} +\e.\]
Therefore,
\[\lim_{j\to \infty} \exc(\overline{B}({\bf 0},H R^{1/t})\cap F_j(r_j^{-1}(X-x_0)),F(T))= 0.\] 
The other excess estimate follows via a similar calculation, thus $r_j^{-1/t}(Y-f(x_0))\to F(T)$ as $j\to\infty$ in the Attouch-Wets topology. This concludes the case $t\geq 1$.

Now let $t\in (0,1)$. Then $f^{-1}:Y\to X$ is a $t$-bi-H\"older homeomorphism. Fix $x_0\in X$, let $T\in\tang(X,x_0)$, and let $r_j\to 0$ be a sequence of scales such that $r_j^{-1}(X-x_0)\to T$ as $j\to\infty$ in the Attouch-Wets topology. Then $r_j^t\to 0$ is a sequence of scales, and there is a convergent subsequence $r_{j_k}^{-t}(Y-f(x_0))\to S$ as $k\to\infty$ in the Attouch-Wets topology. As in the case $t\geq 1$, we have $g_{j_k}:r_{j_k}^{-t}(Y-f(x_0))\to r_{j_k}^{-1}(X-x_0)$, a sequence of $t$-bi-H\"older homeomorphisms with uniform coefficients which extend to $t$-H\"older functions $G_{j_k}:\R^m\to \R^d$, and the maps $G_{j_k}$ converge locally uniformly to a $t$-H\"older map $G:\R^m\to \R^d$. Then 
\[ G_{j_k}(r_{j_k}^{-t}(Y-f(x_0)))\to G(S)\] 
as $j\to\infty$ in the Attouch-Wets topology, and since $G_{j_k}(r_{j_k}^{-t}(Y-f(x_0)))=r_{j_k}^{-1}(X-x_0)$ converges, it must converge to $T$. Therefore $G(S)=T$, and by locally uniform convergence $G|_S$ is a $t$-bi-H\"older homeomorphism onto $T$. Finally, $G^{-1}:T\to S$ is a $(1/t)$-bi-H\"older homeomorphism, as desired.
\end{proof}

We can now prove Theorem \ref{thm:main}.

\begin{proof}[Proof of Theorem \ref{thm:main}]
Fix $s>1$. Let $n\geq 4$ be an even integer such that $\alpha_n\in (1,s)$ and let $\eta \in \mathscr{C}_n$ be a random function. By Proposition \ref{thm:holder} there exists a $\frac1{\a_n}$-H\"older surjection $f:[0,1] \to K^{\eta}$ and by Proposition \ref{prop:reg}, $\mathcal{H}^{\alpha_n}(K^{\eta})>0$. Finally, by Proposition \ref{prop:tangents}, for $\mathcal{H}^{\a_n}$-a.e. point $x\in K^{\eta}$ and for all integers $k\geq 0$, there exists $T_{n,k} \in \tang(K^{\eta},x)$ such that $T_{n,k}$ has exactly $\frac{1}{5n-7}((5n-6)^k-1)$ many cut points. By Lemma \ref{lem:cuthom}, it follows that tangents $T_{n,0}, T_{n,1},\dots$ are pairwise topologically distinct.

By Assouad's embedding theorem \cite{Assouad}, there exists $N\in\N$ and a bi-Lipschitz embedding of the metric space $(K^{\eta},|\cdot|^{\a_n/s})$ into $\R^N$ and denote by $\Gamma$ the embedded image. This mapping produces an $(\a_n/s)$-bi-H\"older homeomorphism $g: K^{\eta} \to \Gamma$. Therefore, there exists some $C>1$ such that for every Borel $B \subset K^{\eta}$, 
\[ C^{-1}\mathcal{H}^{\a_n}(B) \leq \mathcal{H}^{s}(g(B)) \leq C\mathcal{H}^{\a_n}(B), \]
so $\mathcal{H}^s(\Gamma)>0$. Moreover, the map $g\circ f : [0,1] \to \Gamma$ is a $(1/s)$-H\"older surjection. By Lemma \ref{lem:biliptan}, we have that for every point $x\in K^{\eta}$ and every tangent $T\in\tang(K^{\eta},x)$ there exists a tangent $T'\in \tang(\Gamma,g(x))$ such that $T$ and $T'$ are $(\alpha_n/s)$-bi-H\"older equivalent; in particular, $T$ and $T'$ are homeomorphic to each other. Therefore, $\tang(\Gamma,x)$ contains infinitely many topologically distinct elements.
\end{proof}

\section{Tangents of self-similar sets}\label{sec:selfsim}
In this section we study tangents of self-similar sets. The main goals are the following proposition and the proof of Theorem \ref{thm:ss}.

\begin{proposition}\label{lem:selfsimtan}
Let $\{\phi_i:\R^N\to \R^N\}_{i=1}^m$ be an IFS of similarities satisfying the OSC for some open set $U$ with $U\cap Q\neq\emptyset$, where $Q$ is the attractor, and let $s= \dimh{Q}$. There exists $C_0>1$ depending only on the Lipschitz norms of $\phi_1,\dots,\phi_m$ such that for $\mathcal{H}^s$-almost every $x\in Q$, every $T\in\tang(Q,x)$, and every $R>0$, there exist similarities $\{f_i:Q\to T\}_{i=1}^{n}$ with $n\leq C_0$ with Lipschitz norms in $[C_0^{-1}R, R]$ such that
\[\overline{B}({\bf 0},R)\cap T\subset \bigcup\limits_{j=1}^{M_R} f_j^R(Q)\subset \overline{B}({\bf 0},2R)\cap T,\]
and for all distinct $i,j\in\{1,\dots,n\}$, $f_i^R(U)\cap f_j^R(U)=\emptyset$ and $f_i^R(U\cap Q)$ is an open subset of $T$. In particular, $\dimh(T) = \dimh(Q)$ and $\mathcal{H}^{\dimh(Q)}(T)>0$.
\end{proposition}

\begin{proof}
Without loss of generality, assume that $\diam{Q}=1$ and that $0<\Lip(\phi_1)\leq\cdots\leq\Lip(\phi_m)<1$. By \cite[Theorem 2.2]{Schief} and by \cite[Theorem 2.1, Lemma 2.5]{Fan}, we may assume that $\mathcal{H}^s(Q\cap U) = \mathcal{H}^s(Q)$. In particular, $Q\cap U \neq \emptyset$.

Fix $x\in Q\cap U$, $T \in \tang(Q,x)$, and $R>0$. Let $r_j\to 0$ be a decreasing sequence of positive scales such that $r_1<1/R$ and the sets $D_j:=(r_j)^{-1}(Q-x)$ converge to $T$ in the Attouch-Wets topology as $j\to\infty$. Define for each $j\in\N$ the similarity $g_j:\R^N\to \R^N$, given by $g_j(y)=(r_j)^{-1}(y-x)$, and note that $g_j(Q)=D_j$. Define the alphabet $A=\{1,\dots,m\}$ and for a word $w\in A^*$, define $L_w:=\Lip(\phi_w)$. Following \cite[\textsection2.3]{BV2}, for any $\delta\in (0,1)$ define
\[A^*(\delta):=\{i_1 \dots i_n\in A^*: n\geq 1\text{ and } L_{i_1 i_2\dots i_n}<\delta\leq L_{i_1 i_2\dots i_{n-1}}\}\]
and $A^*(1):=\{\eps\}$, the set containing only the empty word. 

Note that if $w\in A^*(r_j R)$, then
\begin{equation}\label{eq:A_w}
L_1r_j R \leq  L_w<r_j R.
\end{equation}
For each $j\in\N$ define 
\[ W^R_j:=\{w\in A^*: \phi_w(Q)\cap \overline{B}(x,r_j R)\neq\emptyset\} \quad \text{and}\quad C^R_j:= W^R_j\cap A^*(r_j R).\] 
Note that for each $j$, $\overline{B}(x,r_j R)\cap Q\subset \bigcup_{w\in C^R_j} \phi_w(Q)$. Additionally, by the OSC, if $v,w\in A^*$ are distinct words satisfying $\phi_w(U)\cap \phi_v(U)\neq \emptyset$, then there exists $u\in A^*\setminus\{\eps\}$ such that either $w=vu$ or $v=wu$, which further implies that $\frac{L_w}{L_v}\leq L_1$ or $\frac{L_v}{L_w}\leq L_1$. Moreover, if $j\in\N$ and $w,v\in C^R_j$, then $\frac{L_w}{L_v}>L_1$ (or vice-versa); thus if $w,v\in C^R_j$ are distinct, $\phi_w(U)\cap \phi_v(U)=\emptyset$. Roughly speaking, each set $\{\phi_w:w\in C^R_j\}$ respects the ``disjoint images" part of the open set condition for $U$. Observe that for each $j\in \N$ and each $w\in C^R_j$ by \eqref{eq:A_w}, the fact that $\diam(Q)=1$ and the fact that $\phi_w(Q)\cap \overline{B}(x,r_j R)\neq \emptyset$, we have $\phi_w(Q)\subset \overline{B}(x,2 r_j R)\cap Q$. Hence, for each $j\in\N$,
\begin{equation}\label{eq:localss}
\overline{B}(x,r_j R)\cap Q\subset \bigcup\limits_{w\in C^R_j} \phi_w(Q)\subset \overline{B}(x,2 r_j R)\cap Q.
\end{equation}

Next, since $Q$ is self-similar, it follows that $\mathcal{H}^s(Q)>0$ \cite[\textsection5.3]{Hutchinson} and, in fact, $Q$ is Ahlfors $s$-regular \cite[Lemma 2.4]{BV2}. That is, there exists $c_1\geq 1$ such that for every $y\in Q$ and $r\in (0,1)$,
\[(c_1)^{-1} r^s\leq \mathcal{H}^s(\overline{B}(y,r)\cap Q)\leq c_1 r^s.\]
Moreover, by \cite[Theorem 5.3.1]{Hutchinson}, for every $w\in A^*$ we have $\mathcal{H}^s(\phi_w(Q))=L_w^s$. 

We claim that $\card(C^R_j) \leq (2/L_1)^sc_1$ for all $j\in\N$. To this end, fix $j\in\N$. By \eqref{eq:localss},
\[\card(C^R_j) (L_1 r_j R)^s \leq \sum_{w\in C^R_j} \mathcal{H}^s(\phi_w(Q)) = \mathcal{H}^s(\bigcup\limits_{w\in C^R_j}\phi_w(Q))\leq\mathcal{H}^s(\overline{B}(x,2 r_j R)\cap Q)\leq (2R)^s c_1 (r_j)^s,\]
where the equality follows from the fact that sets $\phi_w(U)$ for $w\in C^R_j$ are mutually disjoint.
The claimed inequality follows.

By the preceding claim, the sequence $(\card (C^R_j))_{j\in\N}$ is a bounded sequence of positive integers, so passing to a subsequence, we may assume that it is constant. That is, there exists an integer $M_R \in [1,(2/L_1)^sc_1]$ such that $\card(C^R_j)=M_R$ for all $j\in\N$.
We write $C_j^R=\{w_{1,j}^R,\dots,w_{M_R,j}^R\}$. Note that for every $j\in\N$ and $i\in\{1,\dots,M_R\}$ the map $g_j\circ \phi_{w_{i,j}^R}:\R^N\to \R^N$ is a similarity with $L_1 R\leq \Lip(g_j\circ \phi_{w_{i,j}^R})\leq R$. Moreover,
\[D_j\cap \overline{B}({\bf 0},R) \subset \bigcup_{i=1}^{M_R}(g_j\circ \phi_{w_{i,j}^R}(Q)) \subset D_j\cap \overline{B}({\bf 0},2R),\] 
and if $i_1,i_2\in\{1,\dots,M_R\}$ are distinct, then $g_j\circ \phi_{w_{i_1,j}^R}(U)\cap g_j\circ \phi_{w_{i_2,j}^R}(U)=\emptyset$.

Since the similarities $g_j\circ \phi_{w_{i,j}}^R$ have uniformly bounded Lipschitz norms and are pointwise uniformly bounded, by the Arzel\`a-Ascoli theorem, and passing to a subsequence, we may assume that for each $i\in\{1,\dots,M_R\}$, $g_{j}\circ \phi_{w_{i,j}^R} \to f_i^R$ locally uniformly as $j\to\infty$. Since $D_j\to S$ as $j\to\infty$ in the Attouch-Wets topology, for each $i\in\{1,\dots,M_R\}$, $f_i^R(Q)\subset S$. Furthermore by the properties of the maps $g_j\circ \phi_{w_{i,j}^R}$ and by local uniform convergence, we obtain analogous properties for the $f_i^R$. Precisely, for each $i\in \{1,\dots,M_R\}$ the map $f_i^R:\R^N\to\R^N$ is a similarity with $L_1 R\leq \Lip(f_i^R)\leq R$. Moreover,
\[\overline{B}({\bf 0},R)\cap S\subset \bigcup_{i=1}^{M_R} f_i^R(Q)\subset \overline{B}({\bf 0},2R)\cap S \]
and if $i,j\in\{1,\dots,M_R\}$ are distinct, then $f_i^R(U)\cap f_j^R(U)=\emptyset$.

To complete the proof, we show that for every $i\in\{1,\dots,M_R\}$, $f_i^R(U\cap Q)$ is an open subset of $S$. Since $f_i^R$ is a similarity, $f_i^R(U)$ is an open subset of $\R^N$, which implies that $f_i^R(U)\cap S$ is an open subset of $S$. By properties of functions $f_i^R$ above we have that 
\[ f_i^R(U\cap Q)\cap B({\bf 0},R)=f_i^R(U)\cap S\cap B({\bf 0},R).\] 
Furthermore, if $\{f_1^{2R},\dots,f_{M_{2R}}^{2R}\}$ is constructed in a similar manner (replacing $R$ by $2R$), then for every $i\in\{1,\dots,M_R\}$ there exists unique $j\in\{1,\dots,M_{2R}\}$ such that $f_i^R(Q)\subset f_j^{2R}(Q)$. Indeed, the sets $C_j^R$ and $C_j^{2R}$ satisfy $\bigcup_{v\in C_j^R}\phi_v(Q)\subset \bigcup_{u\in C_j^{2R}} \phi_u(Q)$, thus for each $v\in C_j^R$, there exists a unique $u\in C_j^{2R}$ satisfying $\phi_v(U\cap Q)\subset \phi_u(U\cap Q)$. Then it must hold that $f_i^R(U\cap Q)$ is an open subset of $f_j^{2R}(U\cap Q)\cap B({\bf 0},2R)$ which is an open subset of $S\cap B({\bf 0},2R)$. Therefore, $f_i^R(U\cap Q)$ is an open subset of $S$.
\end{proof}

\subsection{Tangents of self-similar sponges}\label{sec:sponge}
We say that a set $F\subset \R^N$ is \emph{locally self-similar} if for every pair of points $x,y\in F$ and every pair of scales $\rho_1,\rho_2>0$, there exists a similarity $g:\R^N \to\R^N$ such that $g(B(x,\rho_1)\cap F)$ is a subset of $B(y,\rho_2)\cap F$ and is open relative to $F$. Note that if $F$ is locally self-similar, then it cannot have a positive and finite number of local cut-points.

\begin{lemma}\label{lem:selfsimsponge}
Let $K$ be a self-similar sponge. Then for every $x\in (0,1)^N\cap K$, every tangent $T\in\tang(K,x)$ is locally self-similar. 
\end{lemma}

\begin{proof}
Suppose that $K$ is the attractor of $\{S_i\}_{i=1}^m$ with $S_i(y) = k^{-1}y+p_i$ where $k$ and $\{p_1,\dots,p_m\}$ are as in \eqref{eq:defnsponge}. If $m=1$, then the result is trivial, hence we may assume that $m\geq 2$. Set $A=\{1,\dots,m\}$. Note that for each $n\in\N$, $v\in A^n$, and $i\in A$, 
\[ S_{iv}(y)=k^{-n-1}y+S_{iv}({\bf 0}) = k^{-n-1}y + k^{-1} S_v({\bf 0})+S_i({\bf 0})\] 
and that $S_{iv}({\bf 0})\in\{0,\frac{1}{k^{n+1}},\dots,1-\frac1{k^{n+1}}\}^N$. Thus by a simple inductive argument, for every $n\in\N$ and every $w\in A^n$, $S_w(y)=k^{-n}y+ S_w({\bf 0})$, where $S_w({\bf 0})\in \{0,\frac{1}{k^n},\dots,\frac{k^n-1}{k^n}\}^N$.

If $(0,1)^N\cap K=\emptyset$, then the result is trivial, so assume $(0,1)^N\cap K\neq\emptyset$. Fix $x\in (0,1)^N\cap K$, $T\in\tang(K,x)$, $p,q\in T$, $\rho_1,\rho_2>0$, and a sequence of positive scales $r_j\to 0$ such that the sets $X_j:=(r_j)^{-1}(K-x)$ converge to $T$ in the Attouch-Wets topology as $j\to\infty$.

We first show that there exist a set $U\subset K\cap (0,1)^N$ open relative to $K$, and a similarity map $G:U\to B(p,\rho_1)\cap T$ which is surjective. For $j\in\N$, define $g_j:K\to X_j$ by $g_j(y)=(r_j)^{-1}(y-x)$, and set $B_j:= B(r_j p+x,\rho_1 r_j)\cap K$. Since $X_j\to T$ as $j\to\infty$ in the Attouch-Wets topology, we have 
$$\lim_{j\to \infty}\exc(B(p,\rho_1/2)\cap T,g_j(B_j))= 0,$$ 
so passing to a subsequence, we may assume that $B_j\neq \emptyset$ for all $j\in\N$. Additionally, since $\lim_{j\to\infty}\diam(B_j)= 0$, since $x\in K\cap (0,1)^N$, and since $\lim_{j\to\infty}\dist(x,B_j)= 0$,
again passing to a subsequence, we may assume that $B_j\subset K\cap (0,1)^N$ for all $j\in\N$. 

For each $j\in\N$, let $n_j\in\N$ such that
\[2\rho_1 r_j <k^{-n_j}\leq 2\rho_1 k r_j\]
and define
\[ W_j:=\{v\in A^{n_j}: S_v([0,1]^N)\cap B_j\neq\emptyset\}.\]
The sets $W_j$ are nonempty and a simple computation shows that
\[B_j\subset \bigcup\limits_{v\in W_j} S_v(K)\subset K\cap B(r_j p+x,4k\sqrt{N}\rho_1 r_j).\]
Thus we may assume that for all $j\in\N$, $\bigcup_{v\in W_j} S_v(K)\subset K\cap (0,1)^N$. Furthermore, by the OSC, for every pair of distinct words $w,v\in W_j$, $S_v((0,1)^N)\cap S_w((0,1)^N)=\emptyset$. Therefore, by the doubling property of $\R^N$, there exists $M' \in\N$ such that $\card(W_j) \leq M'$ for all $j\in\N$.

We claim that $\bigcap_{w\in W_j} S_w([0,1]^N)\neq\emptyset$ for every $j\in\N$. We first show the claim for two words only. Fix distinct $v,w\in W_j$ and let $z_v,z_w\in [0,1]^N$ such that 
\[ |S_v(z_v)-(r_jp+x)|<\rho_1 r_j \quad \text{and}\quad |S_w(z_w)-(r_j p+x)|<\rho_1 r_j,\] 
thus by the triangle inequality we obtain 
\[\dist(S_v([0,1]^N),S_w([0,1]^N))<2\rho_1 r_j <k^{-n_j}.\] 
Therefore, $S_w([0,1]^N)\cap S_v([0,1]^N)\neq\emptyset$ and the claim in this special case holds.

For $l\in\{1,\dots,N\}$ and $w\in W_j$, denote by $S_w({\bf 0})^{(l)}$ the $l$-th coordinate of $S_w({\bf 0})$. Note that if $w,v\in W_j$ are distinct and $S_v({\bf 0})^{(l)}-S_w({\bf 0})^{(l)}>0$ for some $l\in\{1,\dots,N\}$, then $S_v({\bf 0})^{(l)}-S_w({\bf 0})^{(l)}= k^{-n_j}$. This further implies that if $u\in W_j \setminus \{w,v\}$, then $S_v({\bf 0})^{(l)}-S_u({\bf 0})^{(l)}\geq 0$. 

Fix now a word $v\in W_j$. For each $w\in W_j\setminus\{v\}$ and $l\in\{1,2,\dots,N\}$, let 
\[
I_w^l := 
\begin{cases}
[0,1], &\text{ if $S_v({\bf 0})^{(l)} - S_w({\bf 0})^{(l)} = 0$}\\
\{0\}, &\text{ if $S_v({\bf 0})^{(l)} - S_w({\bf 0})^{(l)} = k^{-n_j}$}\\
\{1\}, &\text{ if $S_v({\bf 0})^{(l)} - S_w({\bf 0})^{(l)} = -k^{-n_j}$.}\\
\end{cases}
\] 
For each $l\in\{1,\dots,N\}$, if there exists $w\in W_j\setminus\{v\}$ with $I_w^l=\{0\}$, then for every $u\in W_j\setminus\{v\}$ either $I_u^l=\{0\}$ or $I_u^l=[0,1]$, and similarly in the case that $I_w^l=\{1\}$. In particular, $I_w^l\cap I_u^l\neq\emptyset$ for every pair of words $w,u\in W_j\setminus\{v\}$. Therefore $\bigcap_{w\in W_j\setminus\{v\}} (I_w^1\times\dots\times I_w^N)\neq\emptyset$. By Lemma \ref{lem:capface}, $S_v(I_w^1\times\dots\times I_w^N)=S_v([0,1]^N)\cap S_w([0,1]^N)$ for each $w\in W_j\setminus\{v\}$, and by the special case this intersection is nonempty. Thus we obtain $\bigcap_{w\in W_j} S_w([0,1]^N)\neq\emptyset$, proving the claim.

Since the sequence $(\card(W_j))_j$ is bounded, passing to a subsequence, we may assume that $\card(W_j) = M$ for all $j\in\N$ for some $M\in\N$. Write $W_j=\{v^j_1,\dots,v^j_M\}$ and note that for every $i\in\{1,\dots,M\}$ and $j\in\N$, $g_j\circ S_{v_i^j}:K\to B(p,4k\sqrt{N}\rho_1)\cap X_j$ is a similarity with $\Lip(g_j\circ S_{v_i^j}) \in (2\rho_1,2\rho_1k]$. By the Arzel\`a-Ascoli theorem and passing to a subsequence, we may assume that that for every $i\in\{1,\dots,M\}$, the maps $g_j\circ S_{v_i^j}$ converge uniformly to a similarity $f_i:K\to T$ as $j\to\infty$. Moreover, for some $L \in (2\rho_1,2\rho_1k]$, we have $\Lip(f_i) = L$ for all $i\in\{1,2,\dots,M\}$. Additionally, 
\[ B(p,\rho_1)\cap T\subset \bigcup_{i=1}^M f_i(K),\] 
and by the OSC, if $i,j\in\{1,\dots,M\}$ are distinct, then $f_i((0,1)^N)\cap f_j((0,1)^N)=\emptyset$.

By Lemma \ref{lem:capface}, for every $j\in\N$ there exists a face $C_j$ of $[0,1]^N$ such that $\bigcap_{i=1}^M S_{v_i^j}([0,1]^N)=S_{v_1^j}(C_j)$. Thus, for every $j\in\N$ there exist distinct $e_{j,1},\dots,e_{j,M}\in\{0,1\}^N$ satisfying $S_{v_i^j}(e_{j,i})=S_{v_1^j}(e_{j,1})$ for every $i\in \{1,\dots,M\}$. Passing to a subsequence, we may assume that there exist distinct $e_1,\dots,e_M\in\{0,1\}^N$ such that $e_{j,i}=e_i$ for every $j\in\N$ and $i\in\{1,\dots,M\}$. Then $f_i(e_i)=f_1(e_1)$ and $S_{v_i^j}(e_i) = S_{v_1^j}(e_1)$ for every $j\in\N$ and $i\in\{1,\dots,M\}$. Therefore, 
\begin{align}\label{eq:f_j}
f_i\circ S_{v_i^j}^{-1}(y) = L(S_{v_i^j}^{-1}(y) - e_i) + f_i(e_i) = L k^{-n_j}(y-S_{v_i^j}(e_i)) &= L(S_{v_1^j}^{-1}(y) - e_1) + f_1(e_1)\\
&= f_1\circ S_{v_1^j}^{-1}(y). \notag
\end{align}

Define now $G=f_1\circ(S_{v_1^1})^{-1}$. By \eqref{eq:f_j}, $G$ is a similarity satisfying 
\[B(p,\rho_1)\cap T\subset G\left(\bigcup_{i=1}^M S_{v_i^1}(K)\right)\subset T.\]
Further, since $\bigcup_{i=1}^M S_{v_i^1}(K)\subset K\cap (0,1)^N$, $G$ maps the open subset $G^{-1}(B(p,\rho_1)\cap T)$ of $K\cap (0,1)^N$ onto $B(p,\rho_1)\cap T$.

By Proposition \ref{lem:selfsimtan}, there exists a similarity $f:\R^N\to\R^N$ such that $q\in f(K)\subset T$ and such that $f((0,1)^N\cap K)\cap B(q,\rho_2)$ is an open subset of $T$. Let $z\in f((0,1)^N\cap K)\cap B(q,\rho_2)$ and let $w\in A^*$ satisfy $f^{-1}(z)\in S_w(K)\subset (0,1)^N\cap K$. Then $f\circ S_w$ is a similarity for which $f\circ S_w((0,1)^N\cap K)$ is an open subset of $f((0,1)^N\cap K)\cap B(q,\rho_2)$, and therefore $f\circ S_w((0,1)^N\cap K)$ is an open subset of $T\cap B(q,\rho_2)$. Thus, 
\[ (f\circ S_w)\circ G^{-1}:B(p,\rho_1)\cap T\to B(q,\rho_2)\cap T\] 
is a similarity such that $(f\circ S_w)\circ G^{-1}(B(p,\rho_1)\cap T)$ is an open subset of $T\cap B(q,\rho_2)$.
\end{proof}

\begin{rem}
Lemma \ref{lem:selfsimsponge} can be easily extended to show that at typical points of a self-similar sponge, only locally self-similar tangents are admitted, even if the sponge does not intersect $(0,1)^N$. To argue this, we would need only make the simple observation that in such a case, the sponge would be contained in an $M$-face of $[0,1]^N$ for some $M<N$ for which the sponge intersects the interior of that face. Then we can repeat the argument in the lower dimension $M$.
\end{rem}

\subsection{Proof of Theorem \ref{thm:ss}}

Let $s>1$, and given an even integer $n\geq 4$ recall numbers $\a_n$ from Section \ref{sec:measure}, sets $K^{n,0}$ from \eqref{eq:setK}, and tangents $\{T_{n,k}\}_{k\in\N}$ of Proposition \ref{prop:tangents}. Fix an even integer $n\geq 4$ such that $\a_n<s$ and a random choice function $\eta\in\mathscr{C}_n$. Let $I:K^\eta\to \G_s$ be the snowflake map composed with the bi-Lipschitz embedding given in Section \ref{sec:proof}. Recall that $I$ is an $(\a_n/s)$-bi-H\"older homeomorphism onto its image $\G_s$, and $\G_s\subset \R^d$ for some $d\in\N$. Furthermore, $\G_s$ is a $(1/s)$-H\"older curve with $\mathcal{H}^s(\G_s)\in(0,\infty)$. Fix a point $x_0\in K^\eta$ such that for every $k\in\N$, $T_{n,k}\in\wtan(K^\eta,x_0)$, and recall by Lemma \ref{lem:biliptan} for every $k\in\N$ there exists a tangent $T_k^*\in\wtan(\G_s,I(x_0))$ and an $(\a_n/s)$-bi-H\"older homeomorphism $I^*_k:T_{n,k}\to T^*_k$ induced by the map $I$. The next proposition completes the proof of Theorem \ref{thm:ss}.

\begin{proposition}\label{prop:structuretan}
Let $k\in\N$, $\{S_i:\R^{d'}\to \R^{d'}\}_{i=1}^m$ be an IFS of similarities satisfying the OSC for some open set $U\subset \R^{d'}$, and let $Q$ be the attractor. Further assume $Q\cap U\neq\emptyset$. Let $Q_k$ be the set of $z\in Q\cap U$ such that $T^*_k$ is bi-Lipschitz equivalent to some $T\in\tang(Q,z)$. If $Q_k \neq \emptyset$, then $\dimh(Q)=s$ and for $\mathcal{H}^{s}$-almost every $y\in Q$, each $S\in\tang(Q,y)$ is bi-H\"older equivalent to a locally self-similar subset of $\R^2$. In particular, $\mathcal{H}^{s}(Q_k)=0$.
\end{proposition}
\begin{proof}
First, if $Q_k\neq\emptyset$, then by Proposition \ref{lem:selfsimtan}, $\dimh(Q)=s$. Let $z\in Q_k$, let $T\in\wtan(Q,z)$ be bi-Lipschitz homeomorphic to $T^*_k$ for some $k\in\N$, and let $g:T^*_k\to T$ be a bi-Lipschitz homeomorphism. Then by Remark \ref{rem:measure}, there exists a countable collection of similarities $\{\phi_j:\R^2\to \R^2\}_{j\in\N}$ such that $T_{n,k}=\bigcup_{j\in\N}\phi_j(K^{n,0})$ and such that for every pair of distinct $i,j\in\N$, $\phi_i((0,1)^2)\cap\phi_j((0,1)^2)=\emptyset$. Then for each $j\in\N$, the composition $(g\circ I^*_k\circ \phi_j):K^{n,0}\to T$ is an $(\a_n/s)$-bi-H\"older homeomorphism onto its image, further satisfying $(g\circ I^*_k\circ \phi_j)(K^{n,0}\cap (0,1)^2)$ is an open subset of $T$. Additionally, by Lemma \ref{lem:selfsimtan} and its proof, for every $R>0$ there is an integer $M_R\in\N$ and a collection of similarities $f_1,\dots,f_{M_R}:\R^{d'}\to \R^{d'}$ such that
\[\overline{B}({\bf 0},R)\cap T\subset\bigcup_{j=1}^{M_R}f_j(Q)\subset \overline{B}({\bf 0},2R)\cap T,\]
further satisfying for every pair of distinct $i,j\in\{1,\dots,M_R\}$, $f_i(U)\cap f_j(U)=\emptyset$ and $f_i(U\cap Q)$ is an open subset of $T$.

Now let $i\in\N$ and $j\in\{1,\dots,M_R\}$ such that
\[V:=(g\circ I^*_k\circ \phi_i)(K^{n,0}\cap (0,1)^2)\cap f_j(Q\cap U)\neq\emptyset,\]
and note that $V$ is an open subset of $T$. Then $f_j^{-1}(V)$ is an open subset of $Q\cap U$ and the composition
\[(f_j^{-1}\circ g\circ I^*_k\circ \phi_i):(\phi_i^{-1}\circ (I_k^*)^{-1}\circ g^{-1})(V)\to f_j^{-1}(V)\]
is an $(\a_n/s)$-bi-H\"older homeomorphism. Recall the similarities $\psi^2_{n,1},\dots \psi_{n,5n-6}^2:\R^2\to \R^2$ from Section \ref{sec:IFS}, the similarities generating the self-similar sponge $K^{n,0}$. Recall that $K^{n,0}\cap (0,1)^2\neq\emptyset$ and recall from the definition of a self-similar sponge that the system $\{\psi_{n,1}^2,\dots,\psi_{n,5n-6}^2\}$ satisfies the OSC via the open set $(0,1)^2$. Then there is a word $u\in\mathcal{A}_n^*$ satisfying \[\psi_{n,u}^2(K^{n,0})\subset (\phi_i^{-1}\circ (I_k^*)^{-1}\circ g^{-1})(V)\]
and such that the composition
\[(f_j^{-1}\circ g\circ I^*_k\circ \phi_i):\psi_{n,u}^2(K^{n,0})\to Q\cap U\]
is an $(\a_n/s)$-bi-H\"older homeomorphism onto its image. Let now
\[D:=(f_j^{-1}\circ g\circ I^*_k\circ \phi_i)(\psi^2_{n,u}(K^{n,0}\cap(0,1)^2)).\]
Note that $D$ is a nonempty open subset of $Q\cap U$ and for every word $w\in\{1,\dots,m\}^*$, $S_w(D)$ is a nonempty open subset of $Q\cap U$ as well. Therefore, letting $B:=\bigcup_{w\in\{1,\dots,m\}^*}S_w(D)$, we have that $B$ is a nonempty open subset of $Q\cap U$, further satisfying for every pair of distinct $s,t\in\{1,\dots,m\}$, $S_t(B)\subset B$ and $S_t(B)\cap S_s(B)=\emptyset$. Then since $B$ is an open subset of $Q\cap U$, there exists an open set $B_0\subset \R^{d'}$ such that $B_0\cap Q=B$ and $B_0\subset U$. Let $B':=\bigcup_{w\in\{1,\dots,m\}^*}S_w(B_0)$. Then it is easy to verify that the OSC is obtained by the open set $B'$ for the system of maps $\{S_1,\dots,S_m\}$ and $B'\cap Q=B$ is nonempty, therefore by \cite[Theorem 2.1, Lemma 2.5]{Fan} $\mathcal{H}^s\res Q=\mathcal{H}^s\res B$. We claim that for every point $y\in B$ and every tangent $S\in \wtan(Q,y)$, there is a point $\tilde{y}\in K^{n,0}\cap (0,1)^2$ and some $\tilde{S}\in\wtan (K^{n,0},\tilde{y})$ such that $\tilde{S}$ is $(\a_n/s)$-bi-H\"older homeomorphic to $S$. Assuming the claim, we can conclude the proof by recalling that by Lemma \ref{lem:selfsimsponge}, the set $\tilde{S}$ is locally self-similar, in particular has either $0$ or infinitely-many local cut-points, hence by Lemma \ref{lem:cuthom} the tangent $S$ cannot be bi-Lipschitz homeomorphic to (or homeomorphic to at all) any set $T^*_k$ for any $k\in\N$.

All that remains is to prove the claim, and to this end let $y\in B$ and let $S\in\wtan(Q,y)$. Since $B$ is an open subset of $Q$, $S\in\wtan(B,y)$, thus there exists a word $w\in\{1,\dots,m\}^*$ with $S\in\wtan(S_w(D),y)$. Note that the composition
\[(S_w\circ f_j^{-1}\circ g\circ I_k^*\circ \phi_i)^{-1}:S_w(D)\to \psi^2_{n,u}(K^{n,0}\cap (0,1)^2)\]
is an $(\a_n/s)$-bi-H\"older homeomorphism, therefore by Lemma \ref{lem:biliptan} the tangents $S$ and $\tilde{S}$ are $(\a_n/s)$-bi-H\"older homeomorphic.
\end{proof}

\appendix

\section{A Lipschitz curve and a point with extreme tangent space}\label{sec:liptan2}

Let $\mathfrak{C}_U(\R^N;{\bf 0})$ denote the collection of all closed sets $X\subset \R^N$ that contain the origin ${\bf 0}$ such that every component of $X$ is unbounded. In Lemma \ref{lem:unbounded} we showed that for every non-constant Lipschitz curve in $\R^N$, the tangent space at every point is a subset of $\mathfrak{C}_U(\R^N;{\bf 0})$. In this appendix we prove that the tangent space can in fact be all of $\mathfrak{C}_U(\R^N;{\bf 0})$.

\begin{theorem}\label{thm:Liptan}
For each $N\in\{2,3,\dots\}$ there exists a Lipschitz curve $f:[0,1] \to \R^N$ such that $f(0)={\bf 0}$ and $\tang(f([0,1]),{\bf 0}) = \mathfrak{C}_U(\R^N;{\bf 0})$.
\end{theorem}

Fix, for the rest of the appendix, an integer $N\geq 2$. For each $k\in\N$ denote by $\mathcal{G}_k$ the collection of all connected graphs $G = (V,E)$ such that 
\begin{enumerate}
\item $\{{\bf 0}\}\cup  \{ 2^{-k}\textbf{m} : \textbf{m} \in\partial [-2^k,2^k]^N \cap \mathbb{Z}^N\} \subset V \subset \{ 2^{-k}\textbf{m} : \textbf{m} \in[-2^k,2^k]^N \cap \mathbb{Z}^N\}$,
\item for $v,v' \in V$ we have $\{v,v'\}\in E$, if and only if $|v-v'|=2^{-k}$.
\end{enumerate}
We can enumerate $\bigcup_{k\in\N}\mathcal{G}_k = \{G_1,G_2, \dots\}$
and for $n\in\N$, we define $k_n\in\N$ to be the unique positive integer with $G_n \in \mathcal{G}_{k_n}$.
Note that for all $n\in\N$ we have $\mathcal{H}^1(\text{Im}(G_n)) \leq (2^{k_n+1}+1)^N$.

Let $(r_n)_{n\geq 0}$ be a decreasing sequence of positive numbers such that $r_0=1$ and for each $n\in\N$
\[ r_n \leq \min\left\{ \frac1{2^{n+1}\mathcal{H}^1(\text{Im}(G_n))} ,  \frac{r_{n-1}}{2^{n+k_{n+1}+1}} \right\}.\]

For each $n\in\N$ let $\phi_n : \R^N \to \R^N$ be the map $\phi_n(x) = r_n x$, and let 
\[ H_n=\phi_n(\text{Im}(G_n))\setminus( (-r_{n+1},r_{n+1})^N).\] 

Since $\text{Im}(G_n)$ is a connected set and it contains ${\bf 0}$, the set $H_n\cap [-r_{n+1} ,r_{n+1} ]^N$ is nonempty; in fact, every point therein is the center of an $(N-1)$-face of the cube $[-r_{n+1} ,r_{n+1} ]^N$. Therefore, 
\[ H_n\cap [-r_{n+1} ,r_{n+1} ]^N = H_n\cap \partial[-r_{n+1} ,r_{n+1} ]^N \subset H_{n+1}\cap  \partial[-r_{n+1} ,r_{n+1} ]^N \subset \phi_{n+1}(\text{Im}(G_{n+1})). \]

Define
\[H=\{{\bf 0}\} \cup \bigcup_{n=1}^\infty H_n.\]
We show below that there exists a Lipschitz surjection $f:[0,1] \to H$ with $f(0)={\bf 0}$ and that $\tang(H,{\bf 0}) = \mathfrak{C}_U(\R^N;{\bf 0})$. We split the proof in several steps.

\begin{lemma}
There exists a Lipschitz surjection $f:[0,1] \to H$ with $f(0)={\bf 0}$.
\end{lemma}

\begin{proof}
First note that $H$ is the countable union of compact sets $H_n$ which converge in Hausdorff distance to $\{{\bf 0}\}$ which is contained in $H$. Therefore, $H$ is compact. 

Now we claim that every point of $H$ can be connected to ${\bf 0}$. To see that, fix $n\in\N$ and $p\in H_n$. Since $\text{Im}(G_n)$ is connected, there exists 
$p_1 \in H_n\cap H_{n+1}$ and a path $\gamma_1$ in $H_n$ that joins $p$ with $p_1$. Assume that for some $k\in\N$ we have defined a point $p_k \in H_{n+k}\cap H_{n+k+1}$. By connectedness of $\text{Im}(G_{n+k+1})$, there exists $p_{k+1} \in H_{n+k+1}\cap H_{n+k+2}$ and a path $\gamma_{k+2}$ in $H_{n+k+1}$ that connects $p_k$ to $p_{k+1}$. The concatenation of paths $\gamma_1,\gamma_2,\dots$ produces a path that connects $p$ to ${\bf 0}$. Therefore, $H$ is connected.

Finally, by the choice of scales $r_n$ we get 
\[\mathcal{H}^1(H)=\sum_{n=1}^\infty\mathcal{H}^1(H_n)\leq \sum_{n=1}^\infty r_n\mathcal{H}^1(\text{Im}(G_n)) \leq  1.\]

Therefore, $H$ is a continuum with $\mathcal{H}^1(H)\leq 1$ and by \cite[Theorem 4.4]{AO} there exists a Lipschitz surjection $f:[0,1] \to H$ such that $\Lip(f) \leq 2$ and $f(0)={\bf 0}$. 
\end{proof}

It remains to show that $\tang(H,{\bf 0}) = \mathfrak{C}_U(\R^N;{\bf 0})$. To this end, fix for the rest of this section a set $T \in \mathfrak{C}_U(\R^N;{\bf 0})$.

For each $j\in \N$ set 
\begin{align*} 
Q_j &=\{ 2^{-j}\textbf{m} : \textbf{m} \in [1-4^j,4^j-1]^N \cap \mathbb{Z}^N\},\\
W_j &= \{v \in Q_j : \dist(v,T\cap [2^{-j}-2^j,2^j-2^{-j}]^N) \leq 2^{2-j}\sqrt{N}\},\\
X_j &= \bigcup_{\substack{v,v' \in W_j \\ |v-v'| \leq 2^{-j}}} [v,v'].
\end{align*}
Since ${\bf 0}\in T$ we have that ${\bf 0}\in W_j$. 

\begin{lemma}\label{lem:boundary}
Every component of $X_j$ intersects the boundary $\partial [2^{-j}-2^j,2^j-2^{-j}]^N$.
\end{lemma}

\begin{proof}
Let $B$ be a component of $X_j$. Fix a vertex $v\in B\cap V_j$ and note that there is a point $p\in T\cap [2^{-j}-2^j,2^j-2^{-j}]^N$ such that $|v-p|\leq 2^{2-j}\sqrt{N}$. Let $K\subset T$ be the component of $T\cap [2^{-j}-2^j,2^j-2^{-j}]^N$ which contains $p$. Since every component of $T$ is unbounded, $K\cap\partial [2^{-j}-2^j,2^j-2^{-j}]^N\neq\emptyset$. Now for any point $q\in K$, let $Z_q=\{w\in W_j: |w-q|\leq 2^{2-j}\sqrt{N}\}$ and let 
\[Y_q=\bigcup_{\substack{w,w' \in Z_q \\ |w-w'| \leq 2^{-j}}} [w,w'].\]
We claim that $Y_q$ is connected. To this end, let $z_0\in Z_q$ be a point in $Z_q$ with minimal distance to $q$, and let $z\in Z_q$ be any other vertex. We now construct a chain of points $z_0,z_1,\dots,z_M\in Z_q$ such that $z_M=z$, and such that for every $i\in\{0,1,\dots,M-1\}$, $[z_i,z_{i+1}]\subset Y_q$. We proceed by induction, noting that the base vertex $z_0$ has already been established. Fix some integer $i\geq 0$ and assume we have already some point $z_i\in Z_q$, furthermore satisfying that if $i\geq 1$ then $[z_{i-1},z_i]\subset Y_q$. Let $z_i=(z_i^1,\dots,z_i^N)$. If $z_i=z$, then we terminate the process. If not, then let $z=(z^1,\dots,z^N)$ and let $m\in \{1,\dots,N\}$ be the first index such that $z_i^m\neq z^m$. If $z_i^m<z^m$, then let $z_{i+1}^m=z_i^m+2^{-j}$, and if instead $z_i^m>z^m$, then let $z_{i+1}^m=z_i^m-2^{-j}$. For each $\ell\in\{1,\dots,N\}$ other than $\ell=m$, let $z_{i+1}^\ell=z_i^\ell$, and let $z_{i+1}=(z_{i+1}^1,\dots,z_{i+1}^N)$. Note that since $|z_0-q|$ is minimal, we have that for every $\ell\in\{1,\dots,N\}$, the difference in coordinates $|z_0^\ell-q^\ell|$ is also minimal among points in $Z_q$. Furthermore, by the construction of $z_{i+1}$, we have that for every $\ell\in\{1,\dots,N\}$, $z_{i+1}^\ell$ is between $z_0^\ell$ and $z^\ell$, possibly equal to one of these. Then $|z_i-z_{i+1}|=2^{-j}$ and $|z_{i+1}-q|\leq |z-q|\leq 2^{2-j}\sqrt{N}$, hence $z_{i+1}\in Z_q$ and $[z_i,z_{i+1}]\subset Y_q$. This process terminates after finitely many steps, completing the construction, and therefore $Y_q$ is connected.

Note that for points $q,u\in K$ which have $|q-u|\leq 2^{-j}$, we have that $Y_q\cap Y_u\neq\emptyset$. Now let $q\in K\cap \partial[2^{-j}-2^j,2^j-2^{-j}]^N$, and since $K$ is connected there exists finite chain of points $\{p_1,\dots,p_M\}\subset K$ such that $p_1=p$, $p_M=q$, and for every $i\in\{1,\dots,M-1\}$, $|p_i-p_{i+1}|\leq 2^{-j}$. Then $\bigcup_{i=1}^M Y_{p_i}\subset X_j$ is a connected set containing $v$ and intersecting with $\partial[2^{-j}-2^j,2^j-2^{-j}]^N$, thus it is contained in $B$, and therefore $B\cap \partial[2^{-j}-2^j,2^j-2^{-j}]^N\neq \emptyset$.
\end{proof}

\begin{lemma}\label{lem:kyoobs}
For each $j\in\N$ there exists $n_j \in \N$ such that $G_{n_j} \in \mathcal{G}_{2j}$ and that $X_j$ is a re-scaled copy of $\text{Im}(G_{n_j})\cap [-1+4^{-j},1-4^{-j}]^N$.
\end{lemma}
 Before proceeding to the proof, note that while every $\text{Im}(G_{n_j})\subset [-1,1]^N$ is connected, intersecting with $[-1+4^{-j},1-4^{-j}]^N$ in effect removes the boundary, which may disconnect the image. Essentially, this lemma states that every $X_j$ extends to some $\text{Im}(G_{n_j})$ (up to re-scaling) by attaching the components of $X_j$ along the boundary $\partial[-1,1]^N$.
\begin{proof}[Proof of Lemma \ref{lem:kyoobs}]
If we define
\begin{align*}
V^j&=4^{-j}\left(2^j W_j\cup(\partial [-4^j,4^j]^N\cap \Z^N)\right),\\
E^j&=\{\{v,v'\}:v,v'\in V^j,|v-v'|=4^{-j}\},
\end{align*}
then $(V^j,E^j)$ is connected by Lemma \ref{lem:boundary}, hence it is equal to $G_{n_j} \in \mathcal{G}_{2j}$ for some $n_j\in\N$. Furthermore, $\text{Im}(G_{n_j})\cap [-1+4^{-j},1-4^{-j}]^N=2^{-j} X_j$, as desired.
\end{proof}

It follows that
\begin{align*}
\exc(T\cap [2^{-j}-2^j,2^j-2^{-j}]^N,2^j\text{Im}(G_{n_j})) &\leq 2^{2-j}\sqrt{N}\\ 
\exc(2^j\text{Im}(G_{n_j})\cap [2^{-j}-2^j,2^j-2^{-j}]^N,T) &\leq 2^{2-j}(\sqrt{N}+1/4).
\end{align*}
This implies that for every $R>0$,
\begin{equation}\label{eqn:exc}
\lim_{j\to\infty}\exc(T\cap \overline{B}({\bf 0},R),2^j \text{Im}(G_{n_j}))=0, \quad \lim_{j\to\infty}\exc(2^j\text{Im}(G_{n_j})\cap \overline{B}({\bf 0},R),T)=0.
\end{equation}

\begin{proof}[{Proof of Theorem \ref{thm:Liptan}}]
Let $\rho_j=2^j(r_{n_j})^{-1}>0$. We claim that $\lim_{j\to\infty} \rho_j H= T$, with respect to the Attouch-Wets topology. 

Recall that $k_{n_j}=2j$ for each $j\in\N$, thus $r_{n_j+1}\rho_j\leq 2^{-n_j-j+1}$ and $\rho_j r_{n_j}=2^j$. Hence, for every $j\in\N$,
\[(\rho_j H_{n_j})\setminus(2^{-n_j-j+1}[-1,1]^N)=(2^j\text{Im}(G_{n_j}))\setminus(2^{-n_j-j+1}[-1,1]^N).\]
Fix $R>0$, and let $j_0\in\N$ such that $2^{j_0-1}\geq R$. For every integer $j\geq j_0$, we have 
\begin{align*}
\rho_j H\cap \overline{B}({\bf 0},R)&\subset \left( (\rho_j H_{n_j}) \cup (2^{-n_j-j+1} [-1,1]^N)\right)\cap \overline{B}({\bf 0},R)\\
&=\left( 2^j \text{Im}(G_{n_j}))\cup(2^{-n_j-j+1}[-1,1]^N)\right)\cap \overline{B}({\bf 0},R).
\end{align*}
Thus by design of $G_{n_j}$ and by the monotonicity and subadditivity properties of excess (see Remark \ref{rem:excess}) , we have
\begin{align*}
\exc(\rho_j H\cap \overline{B}({\bf 0},R),T)&\leq \exc(\left( 2^j \text{Im}(G_{n_j}))\cup(2^{-n_j-j+1}[-1,1]^N)\right)\cap \overline{B}({\bf 0},R),T)\\
&\leq \exc(2^j\text{Im}(G_{n_j})\cap \overline{B}({\bf 0},R),T)+\exc(2^{-n_j-j+1} [-1,1]^N,T).
\end{align*}
As $j\to\infty$, by \eqref{eqn:exc} the first term in the latter sum approaches $0$. Furthermore, since ${\bf 0}\in T$ and as $j\to\infty$ $\diam(2^{-n_j-j+1}[-1,1]^N)$ approaches $0$, as $j\to\infty$ the second term in the latter sum approaches $0$ as well.

We now show that $\lim_{j\to\infty}\exc(T\cap \overline{B}({\bf 0},R),\rho_j H)=0$. Note that 
\[ (\rho_j H_{n_j})\setminus(\rho_j r_{n_j+1} [-1,1]^N)\subset \rho_j H,\] 
so by monotonicity and triangle inequality for excess (as in Remark \ref{rem:excess}), we have
\begin{align*}
\exc(T\cap \overline{B}({\bf 0},R),\rho_j H)&\leq \exc(T\cap \overline{B}({\bf 0},R),(\rho_j H_{n_j})\setminus(\rho_j r_{n_j+1} [-1,1]^N)\\
&=\exc(T\cap \overline{B}({\bf 0},R), (2^j \text{Im}(G_{n_j}))\setminus(\rho_j r_{n_j+1}[-1,1]^N))\\
&\leq \exc(T\cap \overline{B}({\bf 0},R), 2^j \text{Im}(G_{n_j}))\\
&+\exc(2^j \text{Im}(G_{n_j}),(2^j \text{Im}(G_{n_j}))\setminus(\rho_j r_{n_j+1}[-1,1]^N)).
\end{align*}
As $j\to\infty$, the first term of the latter sum approaches $0$ by \eqref{eqn:exc} and by design of the sets $G_{n_j}$. In the second term of the latter sum, the sets differ only inside $\rho_j r_{n_j+1}[-1,1]^N$, therefore the excess is at most $\diam(\rho_j r_{n_j+1}[-1,1]^N)\leq 2\rho_j r_{n_j+1} \sqrt{N}$, which approaches $0$ as $j\to\infty$. This completes the proof, by definition of convergence $\lim_{j\to\infty} \rho_j H$ in the Attouch-Wets topology.
\end{proof}

\bibliography{qssbib}
\bibliographystyle{amsbeta}

\end{document}